\documentclass[
12pt]
{amsart}

\usepackage[margin=1in, marginpar=.7in, marginparsep=.1in]{geometry}
\usepackage{hyperref}
\usepackage{cite}
\usepackage{amssymb}
\usepackage{amsmath}
\usepackage{amsthm}
\usepackage{amsfonts}
\usepackage{graphicx}%
\usepackage{comment}
\usepackage{bbm}
\usepackage{xcolor}
\usepackage[toc,page]{appendix}
\usepackage{mathtools}

\newtheorem{theorem}{Theorem}[section]
\newtheorem{assumption}[theorem]{Assumption}
\newtheorem{proposition}[theorem]{Proposition}
\newtheorem{corollary}[theorem]{Corollary}
\newtheorem{lemma}[theorem]{Lemma}
\newtheorem{remark}[theorem]{Remark}
\newtheorem{example}[theorem]{Example}

\newtheorem*{acknowledgement}{Acknowledgement}
\theoremstyle{definition}
\newtheorem{definition}[theorem]{Definition}
\newtheorem{notation}[theorem]{Notation}

\newtheorem{conjecture}[theorem]{Conjecture}

\numberwithin{equation}{section}

\global\newcount\NNN \global\NNN=1
\newcommand{\ffractal}
{\marginpar{{\color{violet}$\the\NNN$ \global\advance\NNN by 1 }}}

\newcommand{\R}{\mathbb{R}}  
\newcommand{\E}{\mathbb{E}} 
\newcommand{\Prob}{\mathbb{P}}

\newcommand{\Hei}{\ensuremath{\mathbb{H}}}
\newcommand{\SUtwo}{{\ensuremath{\operatorname{SU}(2)}}}

\newcommand{\Lcal}{\mathcal{L}}
\newcommand{\Ucal}{\mathcal{U}}

\newcommand{\Xcal}{\mathcal{X}}

\begin{document}

\title[DMMS: spectrum,  irreducibility, small deviations]{Dirichlet metric measure spaces: spectrum, irreducibility, and small deviations}

\author[Carfagnini]{Marco Carfagnini}
\address{ School of Mathematics and Statistics\\
University of Melbourne \\
Parkville VIC 3010, Australia.
}
\email{marco.carfagnini@unimelb.edu.au}

\author[Gordina]{Maria Gordina{$^{\dag }$}}
\thanks{\footnotemark {$\dag$} Research was supported in part by NSF Grant DMS-2246549.}
\address{ Department of Mathematics\\
University of Connecticut\\
Storrs, CT 06269,  U.S.A.}
\email{maria.gordina@uconn.edu}

\author[Teplyaev]{Alexander Teplyaev{$^{\ddag }$}}
\thanks{\footnotemark {$\ddag$} Research was supported in part by NSF Grant DMS-2349433 and the Simons Foundation.}
\address{ Department of Mathematics\\
University of Connecticut\\
Storrs, CT 06269,  U.S.A.}
\email{teplyaev@uconn.edu}



\begin{abstract}
We show that for ultracontractive irreducible Dirichlet metric measure spaces, the Dirichlet spectrum is discrete for a restriction to any connected open set without any assumption on regularity of the boundary. The main applications include small deviations of the corresponding Hunt process and large time asymptotics for the generalized heat content. Our examples include Riemannian and sub-Riemannian manifolds, as well as non-smooth and fractal spaces. 
\end{abstract}

\maketitle

\tableofcontents

\section{Introduction}\label{s.intro}

Let $D\subset \R^{n}$ be an open bounded set and $-\frac{1}{2} \Delta_{D}$  be the Dirichlet Laplacian restricted to $D$. That is, the operator $-\frac{1}{2} \Delta_{D}$  can be viewed as the restriction of  $-\frac{1}{2} \Delta_{\R^{n}}$ to $D$ with zero (Dirichlet) boundary conditions. The spectral theory for such operators has been studied in \cite[Chapter 6]{EvansPDEBook2nd}, \cite[Corollary 5.1.2]{SoggeBook2017}, \cite[Equation (3.3), p. 39]{ZelditchBook2017}, including classification of its spectrum, and in the case of a discrete spectrum regularity of its eigenfunctions. In this case the Laplacian  $-\frac{1}{2} \Delta_{\R^{n}}$ is elliptic, therefore one can use partial  differential equations techniques to tackle such questions when the boundary $\partial D$ satisfies some regularity assumptions. 

One can also treat the operator  $-\frac{1}{2} \Delta_{D}$ from a probabilistic point of view. If $\left\{ B_{t} \right\}_{t\geqslant 0}$ is a Brownian motion in $\R^{n}$ starting at an interior point in  $D$, then $-\frac{1}{2} \Delta_{D}$ is the infinitesimal generator of $B_{t}^{D}$, i.e. the Brownian motion killed when it reaches the boundary $\partial D$. In addition, the first non-zero eigenvalue $\lambda_{1}(D)$ of  $-\frac{1}{2} \Delta_{D}$  appears in the small ball probabilities for $B_{t}$ as follows
\begin{equation}\label{eqn.intro1}
\lim_{\varepsilon \rightarrow 0} -\varepsilon^{2} \log \Prob \left( \tau_{D} >1 \right) = \lambda_{1}(D),
\end{equation}
where $\tau_{D}$ denotes the exit time from $D$ of a Brownian motion starting in $D$. 

Finally, from a functional analytic point of view, the operator  $-\frac{1}{2} \Delta_{D}$  is the generator of a Dirichlet form restricted to $D$. More precisely, it is the restriction to $L^{2}\left( D, dx \right)$ of the Dirichlet form $\mathcal{E}$ on $L^{2} \left( \R^{n}, dx \right)$ given by 
\[
\mathcal{E}(f, g):= \int_{\R^{n}} \langle \nabla_{\R^{n}} f, \nabla_{\R^{n}} g \rangle dx , \quad f,g \in \mathcal{D}_{\mathcal{E}} \subset L^{2} \left( \R^{n} , dx \right),
\]
where $dx$ denotes the Lebesgue measure. 

The goal of this paper is to apply these three points of view to study spectral properties of generators of Dirichlet forms restricted to subsets of a metric measure space $\left( \Xcal, d,  m \right)$ and probabilistic applications for the corresponding stochastic processes. Namely, we consider the generator $A$ of a Dirichlet form $\left( \mathcal{E}, \mathcal{D}_{\mathcal{E}} \right)$ on $L^{2} \left( \Xcal , m\right)$, and then its restriction  $A_{\Ucal}$ to a bounded open set $\Ucal \subset \Xcal$, without any regularity or smoothness assumptions on the boundary $\partial \Ucal$.  Spectral properties of  $A_{\Ucal}$ such as that its spectrum is discrete, $L^{p}$-bounds of the eigenfunctions, and irreducibility are then used to prove small deviation principles and large-time asymptotics of the heat content for the corresponding process. In many examples, the kernel of the stochastic process associated to $\left( \mathcal{E}, \mathcal{D}_{\mathcal{E}} \right)$ satisfies the upper bound 
\begin{equation}\label{eqn.intro2}
p_{t}(x,y) \leqslant C t^{-\frac{\alpha}{\beta}} \exp \left( - \left( \frac{d^{\beta} (x,y)}{C t} \right)^{\frac{1}{\beta -1}} \right),
\end{equation}
where $\alpha>0$ turns out to be the \emph{Hausdorff dimension} of $\Xcal$ and $\beta >1$  is known as the \emph{walk dimension} of the heat kernel $p_{t}$, \cite{Barlow1998a, BarlowGrigoryanKumagai2009}. The off-diagonal upper bound \eqref{eqn.intro2} holds for the Brownian motion on the Sierpi\'{n}ski gasket \cite{BarlowPerkins1988}, on nested fractals  \cite{Kumagai1993a}, affine nested fractals  \cite{FitzsimmonsHamblyKumagai1994},   Sierpi\'{n}ski carpets \cite{BarlowBass1999a}, and fractal fields \cite{HamblyKumagai2003}. Note that off-diagonal lower estimates do not hold in general, see \cite{HamblyKumagai1999}. Both lower and upper off-diagonal sub-Gaussian estimates for the heat kernel on infinite graphs have been shown in \cite{GrigoryanTelcs2001, GrigoryanTelcs2012}. The value $\beta =2$ corresponds to the setting in Section~\ref{s.SturmDirichlet}, see  \cite[Remark 2.14]{KajinoMurugan2020}.   In \cite[Lemma 2.1]{DaviesSimon1984} and \cite[Theorem 3.44]{BarlowBass1999a} a series expansion for the heat kernel is given in the ultracontractive case. The mathematical literature on the subject is vast, and while we cannot give a comprehensive review, we mention \cite{ZhangZhu2019, Metivier1976a, MarianoWang2023} dealing with different non-Euclidean settings such as $\operatorname{RCD}$-spaces, Dirichlet spaces and non-elliptic differential operators.

The paper is organized as follows. In Section~\ref{s.Preliminaries} we describe general conditions to ensure compactness of the semigroup $P_{t}^{\Ucal}:= \exp ( tA_{\Ucal} )$ associated to the operator $A_{\Ucal}$. These conditions are given in terms of the heat kernel $p_{t} (x, y)$, i.e. the kernel for the semigroup $P_{t}:= \exp ( tA )$. We discuss assumptions under which the operator $-A_{\Ucal}$ has a discrete spectrum, the first eigenvalue $\lambda_{1} (\Ucal)$ is strictly positive, and the semigroup $P_{t}^{\Ucal}$ is irreducible. The latter is then used to show that $\lambda_{1} (\Ucal)$ is simple without any regularity assumptions on $\Ucal$, see Theorem \ref{thm.irr.simple.evalue}. We conclude this section by providing upper and lower bounds on the $L^{p}$-norms of the eigenfunctions. As shown in many examples in Section~\ref{s.Examples}, these results apply to non-ellipitc operators. Moreover, as we do not assume any regularity and smoothness of the boundary $\partial \Ucal$, our results are novel also for more regular operators such as uniformly elliptic ones, where spectral properties are often proven by means of PDEs techniques relying on smoothness of the boundary. 

In Section~\ref{s.Examples} we describe several classes of examples for which these assumptions hold. They include local and non-local Dirichlet forms satisfying upper and lower heat kernel estimates and sub-Riemannian Dirichlet forms on Lie groups. We also recover well-known results on Riemannian manifolds, as well as new results since our techniques work without regularity assumptions on $\partial \Ucal$. We collect our main results in Theorem  \ref{thm.main.sturm.setting}, Theorem \ref{thm.main.local.dir.forms}, Theorem \ref{thm.main.non.local.dir.forms}, and Theorem \ref{thm.main.local.dir.forms.lie.groups}.

In Section~\ref{s.small.deviations} we discuss analogues of the small deviation principle  \eqref{eqn.intro1} in the setting of Dirichlet metric measure spaces. We prove analogues of \eqref{eqn.intro1} under different assumptions. First we consider the case when the process or eigenvalues satisfy some scaling properties. Then we assume that we have a group $G$ acting on the metric  space $\Xcal$ and introduce the notion of $G$-\emph{dilation structures} on Dirichlet metric measure spaces. $G$-dilation structures are inspired by the work of Varopoulos  in \cite{Varopoulos1985a}, and the key observation that  \eqref{eqn.intro1}  can be viewed in terms of the action of the multiplicative group $G= \R_{>0}$ on the metric space $\Xcal = \R^{n}$. We can then use such $G$-dilation structures to study small deviations on metric measure spaces, where the action is not necessarily by a continuous group such as $\R_{>0}$. For example, on fractal spaces one has the action of $\mathbb{Z}$, see e.~g.  \cite[Proposition 3.3.1]{KigamiBook2001}, \cite[(1.4.8)]{StrichartzBook2006}, \cite[Equation(1.2)]{BarlowPerkins1988}, and \cite{FitzsimmonsHamblyKumagai1994}. The main results of this section are collected in Theorem~\ref{thm.small.dev}, Theorem~\ref{thm.small.dev.2}, and Theorem~\ref{thm.small.dev.3}.

In Section~\ref{s.Hei.SU2} we prove a small deviation principle for a hypoelliptic diffusion on the group   $\operatorname{SU}\left( 2 \right)$ using a contraction of $\operatorname{SU}\left( 2 \right)$ to the Heisenberg group, which has been  introduced first in  physics literature starting with \cite{InonuWigner1953, Segal1951c}.  This is an example of what one can call an approximate dilation structure, as the Heisenberg group has a dilation, and we use the contraction to introduce such an approximate dilation structure on $\operatorname{SU}\left( 2 \right)$. Our main result is Theorem   \ref{thm-lim-eigen}, where we show that eigenvalues for the sub-Laplacian on  $\operatorname{SU}\left( 2 \right)$ converge to eigenvalues for the sub-Laplacian in the unit ball in $\Hei$, after a proper renormalization. We conclude with Section~\ref{s.heat.content} giving some applications to the long-time asymptotics of the heat content on metric measure spaces.

\begin{acknowledgement}
The authors thank Michael Hinz, Jun Kigami, Takashi Kumagai and Mathav Murugan for helpful suggestions and discussions. 
 \end{acknowledgement}

\section{Preliminaries and general results}\label{s.Preliminaries}

\subsection{Closed  non-negative quadratic forms and symmetric semigroups on  Hilbert spaces}

We start by recalling some general facts about strongly continuous semigroups, their infinitesimal generators, and Dirichlet forms. We refer to \cite[Section 1.3, pp.~16-25]{FukushimaOshimaTakedaBook2011}, \cite[Chapter 1]{BouleauHirschDirichletFormsBook}, and \cite{ChenFukushimaBook2012, MaRocknerBook, ReedSimonI} for more details. 

Suppose $H$ is a separable Hilbert space, and we denote by $B\left( H \right)$ the space of bounded linear operators on $H$ equipped with the operator norm $\Vert \cdot \Vert$, and by $I$ the identity operator on $H$. A \emph{strongly continuous semigroup} on a Hilbert space $H$ is a map $P_{t} :[0, \infty) \longrightarrow B\left( H \right)$ such that
\begin{align*}
& P_{0}=I,
\\
& P_{t}P_{s} = P_{t+s} \text{ for all } t,s \geqslant 0,
\\
& \Vert P_{t}x-x\Vert \xrightarrow[t\to 0^{+}]{} 0 \text{ for any } x \in H,
\end{align*}
that is,  $\left\{ P_{t} \right\}_{t \geqslant 0}$ is continuous in the strong operator topology. We say that a semigroup is a \emph{contraction semigroup} if

\[
\Vert P_{t} \Vert \leqslant 1 \text{ for all } t \geqslant 0.
\]
Then the infinitesimal generator $\mathcal{L}$ of a strongly continuous semigroup $\{ P_{t} \}_{t \geqslant 0}$ on $H$ is defined by
\begin{align*}
& \mathcal{L} x :=  \lim_{t\rightarrow 0^{+}} \frac{P_{t} x -x}{t}, x \in \mathcal{D}\left(  \mathcal{L} \right),
\\
& \mathcal{D}\left( \mathcal{L} \right) := \left\{ x \in H: \lim_{t\rightarrow 0^{+}} \frac{P_{t} x -x}{t}  \text{ the limit exists}   \right\},
\end{align*}
where the limit is taken in the topology of $H$. Recall that the Hille-Yosida theorem for contraction semigroups says that $\left( \mathcal{L},  \mathcal{D}\left( \mathcal{L} \right)  \right)$ is a closed densely defined  linear operator on $H$. Moreover,  if $P_{t}$ is a symmetric operator for each $t>0$, then the infinitesimal generator is a negative (non-positive definite) self-adjoint operator by \cite[Lemma 1.3.1]{FukushimaOshimaTakedaBook2011}. On the other hand, if $\mathcal{L}$ is a negative  (non-positive definite) self-adjoint operator on $H$, then $P_{t} := \exp(t\mathcal{L})$, defined using the spectral theorem, is a strongly continuous semigroup with the infinitesimal generator $\mathcal{L}$ by \cite[Lemma 1.3.2]{FukushimaOshimaTakedaBook2011}. Strongly continuous families of semigroups are connected to non-negative definite closed quadratic forms by \cite[Theorem 1.3.1]{FukushimaOshimaTakedaBook2011} and \cite[Section VIII.6]{ReedSimonI}, which states that there is a one-to-one correspondence between non-negative definite closed quadratic forms $\mathcal{E}$ on $H$ and negative self-adjoint operators $\mathcal{L}$ on $H$.   More precisely,

\begin{align*}
	& \mathcal{D}_{\mathcal{E}} := \mathcal{D} ( \sqrt{-\mathcal{L}} ),
	\\
	& \mathcal{E} (u,v) := \langle \sqrt{-\mathcal{L}} u, \sqrt{-\mathcal{L}} v \rangle.
\end{align*}

We now give several equivalent conditions for $\mathcal{L}$ to have a discrete spectrum. Note that E.~B.~Davies in \cite[Theorem 2.1.5]{DaviesBook1989} gives the same conclusion as in  Theorem~\ref{thm-discrete-compact}, but under an extra assumption that the semigroup $\{P_{t} \}_{t\geqslant 0}$ is ultracontractive, that is, $\Vert P_{t} f \Vert_{L^{\infty}} \leqslant c \Vert f \Vert_{L^{2}}$  as defined on \cite[p. 59]{DaviesBook1989}. Ultracontractivity is equivalent to the existence of an integral kernel for the operator $P_{t}$ by \cite{CarlenKusuokaStroock1987}, as we discuss later in Section~\ref{s.EigenfunctionsBounds}. In Theorem~\ref{thm-discrete-compact} we do not assume ultracontractivity. The proof is given in Appendix~\ref{s.ProofThm2.1}. 

\begin{theorem}\label{thm-discrete-compact}
Let $\{P_{t} \}_{t\geqslant 0}$ be a strongly continuous semigroup of symmetric operators on a Hilbert space $H$ with the infinitesimal generator $\mathcal{L}$, then the  following statements are equivalent.

\begin{enumerate}
	\item $P_{t}$ is compact for all $t>0$;
	\item $P_{t_{0}}$ is compact for some $t_{0}>0$;
	\item $\mathcal{L}$ has a discrete spectrum, that is, it has a pure point spectrum with isolated eigenvalues of finite multiplicity.
\end{enumerate}
\end{theorem}

\begin{notation}\label{notation.basis}
Throughout the paper we denote by $\{\varphi_{n}\}_{n=1}^{\infty}$ the orthonormal basis consisting of eigenfunctions of $\mathcal{L}$, where for each repeated eigenvalue $\lambda_{n}$ we index the corresponding eigenfunctions consequently according to its (finite) multiplicity. 
\end{notation}

Using this notation the conclusion of Theorem~\ref{thm-discrete-compact} can be written as

\begin{align}
& P_{t} u= \sum_{n=1}^{\infty} e^{-\lambda_{n} t} \sum_{k=1}^{M_{n}} \langle u, \varphi_{n,k} \rangle \varphi_{n,k}= \sum_{n=1}^{\infty} e^{-\lambda_{n} t} \langle u, \varphi_{n} \rangle \varphi_{n} \label{eqn.decomp.semigroup}
\end{align}
for every $u \in H$ and $t>0$.

\subsection{Setting and basic assumptions}\label{s.setting}
We start by introducing basic notation, for details we refer to Appendix~\ref{s.RestrictedDirichletForm}. The main results of this subsection are Proposition~\ref{thm.spectral.gap} and Proposition~\ref{thm-lambda-1}, where   we prove that $-A_{\Ucal}$ has a spectral gap under two different sets of assumptions on $\Ucal$, that is, we show that the first eigenvalue $\lambda_{1}$ is positive. 

\begin{assumption}\label{assumption.general} All metric spaces are assumed to be locally compact and  separable. In addition, all the measures on these metric spaces are Radon $\sigma$-finite on the Borel $\sigma$-algebra over the metric space. We also assume that these measures have full support.
\end{assumption}

Throughout the paper we consider $H= L^{2}(\Xcal, \mu )$ to be the Hilbert space of square-integrable $\mu$-measurable real-valued  functions on a metric measure space $\left( \Xcal, d, \mu \right)$. We further assume that $\left( \mathcal{E}, \mathcal{D}_{\mathcal{E}} \right)$ is a regular Dirichlet form on  $L^{2}(\Xcal, \mu)$. We denote by $A$ the negative self-adjoint generator for the Dirichlet form $\left( \mathcal{E}, \mathcal{D}_{\mathcal{E}} \right)$.

If $\Ucal$ is a Borel subset of $\Xcal$, then one can define  the \emph{restricted Dirichlet form} $\left( \mathcal{E}, \mathcal{D}_{\mathcal{E}} (\Ucal) \right)$ by considering the domain to be the  functions $\mathcal{D}_{\mathcal{E}} (\Ucal)\subset\mathcal{D}_{\mathcal{E}}$ that vanish outside of $\Ucal$, see \cite{FukushimaOshimaTakedaBook2011, ChenFukushimaBook2012}.  We denote by $A_{\Ucal}$ the negative self-adjoint operator on $L^2\left( \Ucal, \mu \right)$ corresponding to the Dirichlet form $\left( \mathcal{E}, \mathcal{D}_{\mathcal{E}} (\Ucal) \right)$. We refer to Appendix~\ref{s.RestrictedDirichletForm} for more details.  Informally, $A_{\Ucal}$ can be thought as a restriction of $A$ conditioned on being zero on the complement of $\Ucal$, although in fact the domain of $A_{\Ucal}$  usually is not a subset of the domain of $A$. 
If $\Ucal$ is open and $\left( \mathcal{E}, \mathcal{D}_{\mathcal{E}}  \right)$ is a regular Dirichlet form on $ L^{2}(\Xcal, \mu )$, then $\left( \mathcal{E}, \mathcal{D}_{\mathcal{E}} (\Ucal) \right)$ is a regular Dirichlet form on $L^2 (\Ucal, \mu )$ as well.

\begin{notation} We denote the \emph{semigroups} corresponding to the Dirichlet forms $\left( \mathcal{E}, \mathcal{D}_{\mathcal{E}}\right)$ and $\left( \mathcal{E}, \mathcal{D}_{\mathcal{E}}\left( \Ucal \right) \right)$ by
\begin{align*}
& P_{t}:= \exp (tA), \qquad P_{t}^{\Ucal} := \exp ( t A_{\Ucal} )
\end{align*}
respectively.
\end{notation}

\begin{assumption}\label{assumption.heat.kernel}
We assume that there exists a measurable function $p_{t}: \Xcal \times \Xcal \rightarrow \R$ such that 
\[
P_{t} f (x) = \int_{\Xcal} p_{t} (x,y) f(y)  d\mu (y)
\]
for all $f\in L^{2} (\Xcal, \mu )$ and for all $t>0$.  We refer to $p_{t}$ as the heat kernel. 
\end{assumption}

Note that the existence of the heat kernel for $P_{t}$ implies the existence of the heat kernel $p^{\Ucal}_{t}$ for $P^{\Ucal}_{t}$ given by Dynkin-Hunt's formula  \eqref{eqn.Dirichlet.HK}. Existence of the heat kernel is a delicate question in general and it is related to existence of the fundamental solution for the heat equation for the operator $A$, the infinitesimal generator of $P_{t}$, see  \cite[Lemma 2.1.2]{DaviesBook1989}, \cite[Section 4]{Sturm1996a} and \cite{CarlenKusuokaStroock1987}.  In particular, by \cite[Lemma 2.1.2]{DaviesBook1989} the existence of the heat kernel is implied by ultracontractivity, see Definition~\ref{d.Ultracontr}.

\begin{notation}\label{n.EssupHK} For any set $V  \subseteq \Xcal$ we denote by
\begin{align*}
& M_{V}\left( t \right):=\operatorname*{ess~sup}_{(x, y)\in V\times V} p_{t}\left( x, y \right),
\\
& M_{V}: [0, \infty ) \longrightarrow [0, \infty]. 
\end{align*}
\end{notation}

\begin{lemma}\label{lem-comp}
Let $\Ucal \in \mathcal{B}\left( \Xcal \right)$ be open with $\mu (\Ucal) <\infty$, and assume that for some $t_{0}>0$
\begin{equation}\label{eq-sup-ess}
 M_{\Ucal}\left( t_{0} \right)<\infty.
\end{equation}
  If $\{ P_{t}^{\Ucal}\}_{t\geqslant 0}$ is strongly continuous, then $P^{\Ucal}_t$ is compact for all $t > 0$.

\end{lemma}

\begin{proof}
By Dynkin-Hunt's formula \eqref{eqn.Dirichlet.HK} we know that $p^{\Ucal}_{t}$ exists and $\int_{E} p^{\Ucal}_{t} (x, y) d\mu (x) d\mu (y) \leqslant \int_{E} p_{t} (x,y)  d\mu (x) d\mu (y)$ for any $E\subset \Xcal \times \Xcal$. Therefore,  \eqref{eq-sup-ess} implies
\begin{align*}
	\Vert p^{\Ucal}_{t_{0}} \Vert_{L^{2} (\Ucal \times \Ucal,  \mu \times \mu )} \leqslant  \Vert p_{t_{0}} \Vert_{L^{2} (\Ucal \times \Ucal, \mu \times \mu )} < \infty,
\end{align*}
and hence $P^{\Ucal}_{t_{0}}$ is compact since it is a Hilbert-Schmidt operator. Thus $P^{\Ucal}_{t}$ is compact for all $t>0$ by Theorem~\ref{thm-discrete-compact}.  
\end{proof}

Assumption~\eqref{eq-sup-ess} in Lemma \ref{lem-comp} and assumption \eqref{eqn.heat.kernel.small.one}
 in Proposition \ref{thm.spectral.gap} are satisfied for many examples as we discuss in Section~\ref{s.Examples}. 

\begin{proposition}[Spectral gap for the restriction]\label{thm.spectral.gap}
Suppose $( \mathcal{E}, \mathcal{D}_{\mathcal{E}})$ is a regular Dirichlet form on $L^{2}\left( \mathcal{X}, \mu \right)$. Let $\Ucal$ be an open set in $\Xcal$ such that $0<\mu(\Ucal)<\infty$, and let $P^{\Ucal}_{t}$ be the semigroup associated with the Dirichlet form $( \mathcal{E}, \mathcal{D}_{\mathcal{E}}\left( \Ucal \right))$ with the infinitesimal generator $A_{\Ucal}$.  

If $M_{\Ucal}\left( t \right)<\infty$ for some $t>0$, then $A_{\Ucal}$ has a discrete spectrum. 

If in addition there exists a $t_{\Ucal}>0$ such that
\begin{equation}\label{eqn.heat.kernel.small.one}
M_{\Ucal}\left( t_{\Ucal} \right)=\operatorname*{ess~sup}_{(x,y)\in\Ucal\times\Ucal} p_{t_{\Ucal}}(x, y)< \frac{1}{\mu(\Ucal)^{2}},
\end{equation}
then the first eigenvalue $\lambda_{1}^{\Ucal}$ of $-A_{\Ucal}$ is strictly positive.
\end{proposition}

\begin{proof}
By \cite[Section 3.2, Theorem 3.3]{GrigoryanHuLau2014} the Dirichlet form $( \mathcal{E}, \mathcal{D}_{\mathcal{E}} (\Ucal) )$ is regular since $\Ucal$ is open, and hence $P_{t}^{\Ucal}$ is a strongly continuous semigroup. The assumptions in Theorem~\ref{thm-discrete-compact} are satisfied, and \eqref{eq-sup-ess} implies that by Lemma~\ref{lem-comp} we have that $P_{t}^{\Ucal}$ is compact for some $t$. Thus by Theorem~\ref{thm-discrete-compact} the operator $A_{\Ucal}$ has a discrete spectrum and $\lambda_{1}^{\Ucal} \geqslant 0$. 

Now assume in addition \eqref{eqn.heat.kernel.small.one}. Suppose $\lambda_{1}^{\Ucal}=0$, then by the spectral mapping theorem \eqref{eqn.spec.mapp.thm} we have that $1\in \sigma_{pp} ( P_{t}^{\Ucal})$ for any $t>0$ and hence
\[
1\leqslant \Vert P_{t}^{\Ucal} \Vert
\]
for all $t>0$ since $\sigma_{pp} ( P_{t}^{\Ucal}) \subset \{ z\in \mathbb{C}: \;  \vert z \vert \leqslant\Vert P_{t}^{\Ucal} \Vert  \}$. Thus
\begin{align*}
1\leqslant \Vert P_{t}^{\Ucal} \Vert^{2}  \leqslant \mu(\Ucal)^{2} M_{\Ucal}\left( t \right)
\end{align*}
for all $t>0$, that is,
\begin{align*}
\frac{1}{\mu(\Ucal)^{2}} \leqslant M_{\Ucal}\left( t \right)
\end{align*}
for any $t>0$, which contradicts \eqref{eqn.heat.kernel.small.one}.
\end{proof}

\begin{definition} 
A Borel set $E \in \mathcal{B}\left( \Xcal \right)$ is called \emph{$P_{t}$-invariant} if $P_{t}\left( \mathbbm{1}_{E}f \right)=0$ $\mu$-a.e. on $E$ for every $t > 0$ and $f \in L^{2}\left( \Xcal, \mu \right)$. The semigroup $\left\{ P_{t} \right\}_{t \geqslant 0}$ is called \emph{irreducible} if for any $P_{t}$-invariant set $E$ either $\mu\left( E \right) = 0$ or $\mu\left( E^{c} \right) = 0$.
\end{definition}

\begin{proposition}[Spectral gap and irreducibility]\label{thm-lambda-1}
Suppose that an open set $\Ucal \subset \Xcal$ is such that $0<\mu(\Ucal)<\infty$ and $0< \mu(\Ucal^c) < \infty$, and that the semigroup $\{ P_{t}^{\Ucal}\}_{t\geqslant 0}$ is strongly continuous. If  $P_{t}$ is irreducible and if \eqref{eq-sup-ess} holds for some $t_{0}>0$, then $\lambda_{1}^\Ucal>0$.
\end{proposition}
\begin{proof}
We argue by contradiction. Assume that  $\lambda_1^\Ucal=0$. Then by Theorem~\ref{thm-discrete-compact} and the Krein-Rutman Theorem in \cite{KreinRutman1948} we have that there exists a non-negative function $f\in L^2(\Ucal, \mu)$ such that  $P_{t}^{\Ucal} f=f$ for all $t>0$.
We can assume that $\|f\|_{L^1(\Ucal, \mu)}=1$, and therefore $d \mu_{f}:=f  d\mu $ is a probability measure which is invariant under the semigroup. Moreover, if $\mu_{f}$ is the initial probability distribution of the Hunt Markov process $X_t^\Ucal$, see  \cite[Theorem A.2.10, page 400]{FukushimaOshimaTakedaBook2011}, then $\mu_{f}$ is the probability distribution of the process $X_t^\Ucal$ for all $t>0$. Hence 
\begin{equation}\label{eqn.exit.time-}
\Prob^{\mu_{f}} \left( \tau_{\Ucal} <\infty \right) = 0.
\end{equation}
Therefore $\Ucal$ is an invariant set of the Hunt process $X_t$, that is, $\mathbb P^{\mu_{f}} \{ X_t\in\Ucal\text{ for all }t>0 \} =1$, and the semigroup $P_t$, that is,  $\mu$-almost everywhere $\mathbbm{1}_{\Ucal^c}P_t\mathbbm{1}_{\Ucal}=0$, which contradicts irreducibility of $P_t$.
\end{proof}

Let us recall several properties of semigroups and corresponding Dirichlet forms.

\begin{definition}[Ultracontractivity]\label{d.Ultracontr} Let $P_{t}$ be a Markov semigroup on $L^{2}\left(\Xcal,  \mu  \right)$, where $\mu$ is a $\sigma$-finite measure on a countably generated $\sigma$-algebra. We say that $P_{t}$ is \emph{ultracontractive} if
\begin{equation}
\Vert P_{t} f \Vert_{L^{\infty}} \leqslant c_t \Vert f \Vert_{L^{2}},
\end{equation}
where the corresponding norm is denoted by $\Vert P_{t} f \Vert_{2 \to\infty}\leqslant c_t$.
\end{definition}

This is the original definition by E.~B.~Davies as found in \cite[p.~59]{DaviesBook1989}. In \cite[Lemma~2.1.2]{DaviesBook1989} he showed that ultracontractivity is equivalent to the existence of an integral (heat) kernel for the semigroup $P_{t}$ satisfying
\begin{equation}\label{eqn.ultra.contract-Davies}
0 \leqslant p_{t}(x,y) \leqslant a_{t}<\infty 
\end{equation}
almost everywhere on $\Xcal \times \Xcal$ for some $a_{t}\geqslant 0$. Moreover, in the proof of \cite[Lemma 2.1.2]{DaviesBook1989} E.~B.~Davies showed that for symmetric Markov semigroups ultracontractivity is equivalent to the fact that $P_{t}$ is a bounded operator from $L^{1}\left(\Xcal,  \mu  \right)$ to $L^{\infty}\left(\Xcal,  \mu  \right)$, also known as $L^{1} \rightarrow L^{\infty}$ ultracontractivity used by Carlen, Kusuoka and Stroock in  \cite{CarlenKusuokaStroock1987}. 
This ultracontractivity allows us to have heat kernel eigenfunction expansions, as stated in  Proposition~\ref{prop.eigenfunctions.expansion} which is a version of  \cite[Theorem~2.1.4]{DaviesBook1989}. 

\begin{proposition}[Eigenfunction expansion of the Dirichlet heat kernel]\label{prop.eigenfunctions.expansion}
Let $\{P_{t} \}_{t \geqslant 0}$ be a strongly continuous contraction semigroup on $L^{2} (\Xcal, \mu )$, and let $\Ucal$ be an open set in $\Xcal$ with $0<\mu(\Ucal)<\infty$. 
If \eqref{eqn.ultra.contract-Davies} is satisfied, then the series
	\begin{align*}
		\sum_{n=1}^{\infty} e^{-\lambda_{n} t} \varphi_{n} (x) \varphi_{n}(y)
	\end{align*}
	converges uniformly on $\Ucal \times \Ucal \times [\varepsilon, \infty)$ for any $\varepsilon>0$. Moreover,
	\begin{align}
		& p^{\Ucal}_{t} (x,y) = \sum_{n=1}^{\infty} e^{-\lambda_{n} t}  \varphi_{n} (x) \varphi_{n}(y), \label{eqn.heat.kernel.expansion}
		\\
		& \Prob^{x} \left( \tau_{\Ucal} >t \right) = \sum_{n=1}^{\infty} e^{-\lambda_{n} t}  \varphi_{n} (x) \int_{\Ucal}  \varphi_{n} (y)  d\mu (y) \notag
	\end{align}
	for any  $x, y \in \Ucal$, and $t>0$.
\end{proposition}

\begin{notation}
 Let us assume that the conditions of Proposition \ref{thm.spectral.gap} are satisfied, so that the differential operator $A_{\Ucal}$ has a discrete spectrum with eigenfunction and eigenvalues denoted by $\{\varphi_{n} , \lambda_{n} \}_{n=1}^{\infty}$.  Throughout the paper we denote by $\{ c_{n} \}_{n=1}^{\infty}$ the integrals 
\begin{equation}\label{eqn.c.n}
c_{n}=c_{n} (\Ucal):= \int_{\Ucal} \varphi_{n} (y)  d\mu (y).
\end{equation}
In Section \ref{s.EigenfunctionsBounds} we show that $c_{n} < \infty$ by providing explicit bounds on $\Vert \varphi \Vert_{L^{\infty} (\Xcal ,\mu)}$.
\end{notation}

The proof of Proposition \ref{prop.eigenfunctions.expansion}  can be found in \cite[Theorem 2.1.4]{DaviesBook1989}. There it is shown that the eigenfunction expansion \eqref{eqn.heat.kernel.expansion} holds under the assumption that $P_{t}$ is ultracontractive, and as we mentioned earlier it is also proven there that ultracontractivity is equivalent to heat kernel bounds  \eqref{eqn.ultra.contract-Davies}.

While ultracontractivity is a property of the semigroup, Nash-type inequalities are properties of Dirichlet forms. We note that for Dirichlet forms the Nash inequality is stronger than ultracontractivity and  
corresponds to stronger heat kernel estimates.  

\begin{definition}[Nash inequality]\label{d.Nash}
We say that the Dirichlet from $\left( \mathcal{E}, \mathcal{D}_{\mathcal{E}} \right)$ on $L^{2}\left( \Xcal, \mu \right)$ satisfies a \emph{Nash inequality} with the parameter $\nu$ if there is a $C>0$ such that
\begin{equation}\label{eqn.Nash.inequal}
\Vert f \Vert_{L^{2} (\Xcal, \mu )}^{2 + \frac{4}{\nu}} \leqslant C \mathcal{E} (f,f) \Vert f \Vert_{L^{1} (\Xcal, \mu )}^{\frac{4}{\nu}} \text{ for all }  f \in \mathcal{D}_{\mathcal{E}}.
\end{equation}
\end{definition}

Moreover, \cite{CarlenKusuokaStroock1987} showed that \eqref{eqn.Nash.inequal} is equivalent to the $L^{1} \rightarrow L^{\infty}$ ultracontractivity of the heat semigroup with the specific power function depending on the parameter $\nu$, namely,
\begin{equation}\label{eqn.ultra.contract}
\Vert P_{t} f \Vert_{L^{\infty} (\Xcal, \mu)} \leqslant C \, t^{-\frac{\nu}{2}} \Vert f \Vert_{L^{1} (\Xcal, \mu )},
\end{equation}
for all $f \in L^{1} (\Xcal, \mu )$ and $t>0$, or equivalently
\begin{equation}\label{e.HKPolyUpperBound}
M_{\Xcal}\left( t \right) =\operatorname*{ess~sup}_{(x, y)\in \Xcal \times \Xcal } p_{t}\left( x, y \right) \leqslant C  t^{-\frac{\nu}{2}},
\end{equation}
for all $t>0$, where $M_{\Xcal}\left( t \right)$ is defined in Notation \ref{n.EssupHK}.

\begin{remark}\label{rem.Varopoulos}
The polynomial heat kernel upper bound \eqref{e.HKPolyUpperBound} was considered by Varopoulos on Lie groups with $\nu = n $, where $n$ is the dimension of the  group, see \cite{Varopoulos1985a, Varopoulos1988a}, \cite[p. 5]{Varopoulos1990a}.
\end{remark}

\subsection{Irreducibility on  $\Ucal$}\label{s.irr}
Before we proceed to the main subject of this section, we formulate an additional assumption on $\Xcal$.

\begin{assumption}\label{assumption.pre-compact.balls}  $\Xcal=\left( \Xcal, d \right)$ is a metric space such that metric balls in $\Xcal$ is pre-compact.
\end{assumption}
In this section we assume that Assumption~\ref{assumption.pre-compact.balls} holds. For example, by \cite[Proposition~2.5.22]{BuragoBuragoIvanovBookEng} if we assume $\left( \Xcal, d \right)$ is a complete locally compact length space, then every closed ball in $X$ is compact.  This assumption is equivalent to $X$ being boundedly compact, and under such an assumption $\Xcal$ is complete. Recall that we only consider Radon measures, therefore under this assumption open and bounded sets have finite measure. 

The main result of this section is Theorem~\ref{thm.irr.simple.evalue}, where we prove irreducibility of the semigroup $P_{t}^{\Ucal}$. As a consequence of irreducibility of the semigroup, we prove that the first eigenvalue $\lambda_{1}^{\Ucal}$ of $-A_{\Ucal}$ is simple. We use the definition of irreducibility of Dirichlet forms and corresponding semigroups found in \cite[p. 55]{FukushimaOshimaTakedaBook2011}. For a definition of irreducible semigroups on Banach lattices we refer to \cite[Section 14.3]{BatkaiKramarRhandiBook2017}. We will use the following characterization of irreducible semigroups in \cite[Example 14.11]{BatkaiKramarRhandiBook2017}.

Let $P_{t}$ be a strongly continuous semigroup on $L^{p} (\Ucal, \mu )$, $1\leqslant p \leqslant \infty$ with the generator $\mathcal{L}$, which is defined using \cite[Theorem 1.3.3]{DaviesBook1989}. Let $s(\mathcal{L}):= \sup \{ \operatorname{Re} (\lambda), \lambda \in \sigma(\mathcal{L}) \}$ be the spectral bound which in our case is equal to the  spectral radius   \cite[p. 48]{BatkaiKramarRhandiBook2017},  
 and $R_{\lambda} = (\mathcal{L}-\lambda I)^{-1}$ be the resolvent operator, for $\lambda$ in the resolvent set of $\mathcal{L}$.

\begin{lemma}[Example 14.11 in \cite{BatkaiKramarRhandiBook2017}]\label{lemma.irredu}
Let $P_{t}$ be any strongly continuous positive semigroup on $L^{p} (\Ucal, m )$, $1\leqslant p < \infty$.  Then $P_{t}$ is irreducible if and only if for any positive $f \in  L^{p} (\Ucal, \mu)$ we have that
\begin{align*}
R_{\lambda}f (x) >0, \text{ for $\mu$-a.e. }  x\in \Ucal \text{ and some }  \lambda >s(\mathcal{L}).
\end{align*}
\end{lemma}

\begin{theorem}[Irreducibility of the semigroup and simplicity of the ground state]\label{thm.irr.simple.evalue}
Suppose that the heat kernel $p_{t}(x, y)$ exists for all $t$ and all $x,y\in \Xcal$. Let  $\Ucal$ be an open  path-connected set in $\Xcal$ such that $0<\mu(\Ucal)<\infty$,  and assume that for any $y\in \Ucal$ and any $r$ small enough and any $x\in B_{r} (y)$ there exists   $t_{0} = t_{0} (x,y,r)$ such that for any $z\in  \Ucal^c$ and any $s<t<t_{0}$ one has that
\[
p_{t}(x,y) - p_{s}(z,y) >0.
\]
Then $P_{t}^{\Ucal}$ is irreducible.	

Moreover, if $P_{t}^{\Ucal}$ is a compact operator for some $t$, then $\lambda_{1}^{\Ucal}$ is a simple eigenvalue and there exists a corresponding eigenfunction $\varphi$ such that $\varphi (x) >0$ for $\mu$-a.e. $x\in \Ucal$.
\end{theorem}

\begin{proof}
Let us first prove that $P_{t}^{\Ucal}$ is irreducible.	 For any $x_{0}, y_{0} \in \Ucal$ let $\gamma$ be a continuous path from $x_{0}$ to $y_{0}$ in $\Ucal$ and
\[
\varepsilon:=\frac12d(\gamma,\partial\Ucal)>0.
\]
We can cover $\gamma$ by finitely many open balls of radius $\delta_{\varepsilon}$. By Dynkin-Hunt's formula \eqref{eqn.Dirichlet.HK} we know that $p^{\Ucal}_{t}(x,y)>0$ if $t \leqslant t_{0}=t_{0}(x,y, \delta_{\varepsilon})$. Indeed,
\begin{align*}
& p^{\Ucal}_{t}(x,y)= p_{t}(x,y) - \E^x \left[ \mathbbm{1}_{\{ \tau_{\Ucal} < t
\} } \,  p_{ t- \tau_{\Ucal}} \left( X_{\tau_{\Ucal}}, y\right) \right]
\\
& \geqslant p_{t}(x,y) - \E^x \left[  p_{ t- \tau_{\Ucal}} \left( X_{\tau_{\Ucal}}, y\right) \right] = \E^x \left[ p_{t}(x,y) - p_{ t- \tau_{\Ucal}} \left( X_{\tau_{\Ucal}}, y\right) \right] >0
\end{align*}
for any $t<t_{0}$ and $x\in B_{\delta_{\varepsilon}}(y)$. Therefore, the chaining argument and the Chapman-Kolmogorov equation imply that $p^{\Ucal}_{t}(x_{0},y_{0})>0$ for all small enough $t>0$. We can now prove the irreducibility of $P_{t}^{\Ucal}$. We can use \cite[Exercise 1.3.1]{FukushimaOshimaTakedaBook2011} to express the resolvent $R_{\mu}$ in terms of the semigroup $P_{t}^{\Ucal}$. Then, by Lemma \ref{lemma.irredu} for any $\lambda \in \R$ and any $f>0$, and for a.e. $x \in \Ucal$ we have that
\begin{align*}
R_{\lambda}f (x) = \int_{0}^{\infty} e^{-\lambda t} (P_{t}^{\Ucal} f)(x)dt=0
\end{align*}
if and only if  $(P_{t}^{\Ucal} f)(x)=0$ for a.e. $t>0$, since $P_{t}^{\Ucal}$ is a positive operator. Thus, $R_{\lambda}f (x) =0$ if and only if
\begin{align*}
\int_{\Ucal} f(y) p^{\Ucal}_{t} (x,y)  d\mu (y)=0,
\end{align*}
that is, if and only if for $\mu$-a.e. $x\in \Ucal$, $p^{\Ucal}_{t} (x,y)$ is zero on a set of positive
measure, which is not possible since $p_{t}^{\Ucal} (x,y) >0$ for all small enough $t>0$.   Hence $P_{t}^{\Ucal}$ is irreducible.

Let us now prove that the first eigenvalue is simple. By Proposition~\ref{thm.spectral.gap} the operator $A_{\Ucal}$ has a spectral gap $\lambda_{1}^{\Ucal}>0$ since $\Ucal$ is open and  $0 < \mu (\Ucal) < \infty$, and $P_{t}^{\Ucal}$ is a strongly continuous semigroup which is bounded for some $t>0$. Its spectral radius is given by $e^{-\lambda_{1}^{\Ucal} t}$, and $K:= \left\{ f\in L^{2} (\Ucal, m ), f \geqslant 0  \text{ a.s.} \right\}$ is a cone in $L^{2} (\Ucal, m )$ such that $P_{t}^{\Ucal} (K) \subset K$. Thus by the Krein-Rutman theorem in \cite{KreinRutman1948}, there exists an eigenfunction $\varphi$ of  $P_{t}^{\Ucal}$ with the eigenvalue $e^{-\lambda_{1}^{\Ucal} t}$ such that $\varphi \in K \backslash \{ 0\}$. By Theorem~\ref{thm-discrete-compact} we know that $\varphi$ is an eigenfunction of $- A_{\Ucal}$ with the eigenvalue $\lambda_{1}^{\Ucal}$. Let us assume that $\varphi(x)=0$ for some $x\in \Ucal$. Then
\begin{align*}
& 0= \varphi(x) = e^{\lambda_{1}^{\Ucal} t} \int_{\Ucal} \varphi(y)p^{\Ucal}_{t} (x, y) d\mu (y) \geqslant 0,
\end{align*}
and hence $\varphi(y)p^{\Ucal}_{t} (x,y)=0$ for a.e. $y\in \Ucal$. The set
\begin{align*}
A:= \{ z\in \Ucal:  \varphi(z) >0 \}
\end{align*}
has positive measure since $\varphi\in K \backslash \{0\}$. Thus, $p^{\Ucal}_{t} (x,y)=0$ for almost every $y\in A$, which is a contradiction  with the irreducibility of $P_{t}^{\Ucal}$.

The semigroup $P_{t}^{\Ucal}$ is irreducible and its generator $A_{\Ucal}$ is self-adjoint, and we proved that there exists $\varphi \in \ker (- \lambda_{1}^{\Ucal} - A_{\Ucal} )$ such that $\varphi>0$. Thus, by \cite[Proposition 14.42 (c)]{BatkaiKramarRhandiBook2017} it follows that $\operatorname{dim} \ker (- \lambda_{1}^{\Ucal} - A_{\Ucal} )=1$.
\end{proof}

\subsection{Regularity and $L^{p}$ bounds for eigenfunctions}\label{s.EigenfunctionsBounds} 
We follow the same assumptions as in Subsection~\ref{s.setting} and \ref{s.irr}.
Let us set 
\begin{align*}
& c_{t, p, q}   := \left( \int_{\Xcal}  \Vert p _{t} (x, \cdot ) \Vert^{q}_{L^{p} } d\mu (x) \right)^{\frac{1}{q}} ,  
& c_{t, p}   :=c_{t, p, 2}    ,
& & A_{p,q,a} := \inf_{t>0} e^{a t} c_{t,p.q},
\end{align*}
where $a>0$ is a fixed number. 

\begin{theorem}\label{thm.Lp.bounds}
Let $\{P_{t} \}_{t \geqslant 0}$ be a strongly continuous contraction semigroup on $L^{2} (\Xcal, \mu )$. Assume that $P_{t}$ is compact for some $t >0$. Then its infinitesimal generator $\mathcal{L}$ admits an eigensystem $\{\varphi_{n}, \lambda_{n} \}_{n=1}^{\infty}$ with the following properties. 

(1) The eigenfunctions satisfy the following $L^{p}$-bounds

\begin{align}
& \Vert \varphi_{n} \Vert_{L^{p} (\Xcal, \mu )} 
\leqslant A_{p, 2,\lambda_{n}}  \; \text{ for any } \; 1 < p \leqslant \infty,  \label{eqn.Lp.bounds}
\\
&\Vert \varphi_{n} \Vert_{L^{1} (\Xcal, \mu )} 
\leqslant \inf \limits_{1<p,q < \infty, \; \frac{1}{p} + \frac{1}{q} =1} \Vert \varphi_{n} \Vert_{L^{p} (\Xcal, \mu )}  A_{ q,1  , \lambda_{n}} ,\notag
\\
& \Vert \varphi_{n} \Vert_{L^{p} (\Xcal, \mu )} 
\geqslant   \frac{1}{A_{q,2,\lambda_{n}} }  \text{ for any }  1 < p < \infty \;  \text{ where } \; \frac{1}{q} = 1 - \frac{1}{p}, \notag
\\
&  \Vert \varphi_{n} \Vert_{L^{\infty} (\Xcal, \mu )} \geqslant \sup_{1<p<2} \left( A_{p,2,\lambda_{n}}   \right)^{ \frac{p}{2-p}}  .  \notag
\end{align}

(2) 
If $M_{\Xcal}\left( t \right) < \infty$ for all $t>0$ 
and $C(\lambda):=\inf_{t>0}  M_{\Xcal}\left( t \right)  e^{\lambda t}$, then
\begin{align}\label{e-est-eig-p}
& \Vert \varphi_{n} \Vert_{L^{p}(\Xcal, \mu ) } 
\leqslant \mu(\Xcal)^{\frac{1}{p} + \frac{1}{2}} C(\lambda_n) 
\leqslant  \mu(\Xcal)^{ \frac{1}{2}} \max\left( 1, \mu (\Xcal) \right) C(\lambda_n),
\end{align}
for any $1\leqslant p \leqslant \infty$. 

(3) 
If $p_{t_{0}}  (x,y)$ is H\"older continuous on $\Xcal \times \Xcal$ for some $t_{0}>0$, then $\varphi_{n} (x)$ is continuous in $\Xcal$ for any $n\in \mathbb{N}$.
\end{theorem}
\begin{remark} 
Note that  the right-hand side of \eqref{e-est-eig-p} is independent of $p$. We also mention that sufficient conditions for H\"older and Lipschitz continuity of the heat kernel on metric measure spaces have been studied recently in  \cite{BifulcoMugnolo2023}. 
\end{remark}

In our paper we are mostly interested in the following corollary. 
\begin{corollary}\label{cor.bound.efunction}
If $\Ucal\subset\Xcal$ is an open set in $\Xcal$ such that $0< \mu (\Ucal) < \infty$,  we can apply Theorem \ref{thm.Lp.bounds} to $P_{t}^{\Ucal}$ with 
\begin{equation}
c_{t,p,q} =c_{t, p,q}  (\Ucal) := \left( \int_{\Ucal}  \Vert p _{t}^\Ucal (x, \cdot ) \Vert^{q}_{L^{p} } d\mu (x) \right)^{\frac{1}{q}},
\end{equation}
and 
\begin{align*}
C(\lambda ) = C(\lambda, \Ucal ) := \inf_{t>0}  M_{\Ucal}\left( t \right)  e^{\lambda t}.
\end{align*}
In particular, the following estimates hold
\begin{equation}\label{eq-Linfty}
\Vert \varphi_{n} \Vert_{L^{\infty}(\Ucal, \mu ) } \leqslant \mu(\Ucal)^{\frac{1}{2}} C(\lambda_{n}, \Ucal ),
\end{equation} 
\begin{equation}\label{eq-L1}
  \Vert \varphi_{n} \Vert_{L^{1}(\Ucal, \mu ) } \leqslant \mu(\Ucal)^{\frac{5}{2}} C(\lambda_{n}, \Ucal )^{2},
\end{equation} 
\begin{equation}
1=\Vert \varphi_{n} \Vert_{L^{2}(\Ucal, \mu ) } \leqslant \mu(\Ucal)  C(\lambda_{n}, \Ucal ).
\end{equation} 
\end{corollary}

\begin{proof}[Proof of Theorem \ref{thm.Lp.bounds}]
The operator $\mathcal{L}$ admits an eigensystem $\{\varphi_{n}, \lambda_{n} \}_{n=1}^{\infty}$ by Theorem \ref{thm-discrete-compact} since  $P_{t}$ is compact for some $t >0$.

(1)   For any $1\leqslant  q\leqslant \infty$ and $f\in L^{q}(\Xcal, \mu)$
\begin{align*}
& \Vert P_{t}^{} f \Vert^{2}_{L^{2}(\Xcal)} = \int_{\Xcal} \left[ \int_{\Xcal} f(y) p^{}_{t} (x,y) d\mu (y) \right]^{2}  d\mu (x) 
\\
& \leqslant \int_{\Xcal}  \Vert f \Vert^{2}_{L^{q}(\Xcal) } \Vert p^{}_{t} (x, \cdot ) \Vert^{2}_{L^{\frac{q}{q-1}}(\Xcal)} d\mu (x)
\\
& =  \Vert f \Vert^{2}_{L^{q}(\Xcal) } \int_{\Xcal} \Vert p^{}_{t} (x, \cdot ) \Vert^{2}_{L^{\frac{q}{q-1}}(\Xcal)}  d\mu (x) = \Vert f \Vert^{2}_{L^{q}(\Xcal) }  c_{t, \frac{q}{q-1}}  ^{2}.
\end{align*}
Thus we can view $P_{t}$ as an operator from  $L^{q}(\Xcal, \mu)$ to $L^{2}(\Xcal, \mu)$ for every $1\leqslant q \leqslant \infty$, and it satisfies
\begin{equation}\label{eqn.2a}
\Vert P_{t}^{} \Vert_{L^{q}(\Xcal)\to L^{2}(\Xcal)} \leqslant c_{t, \frac{q}{q-1}}  .
\end{equation}
If $1 \leqslant q <\infty$, then the adjoint $(P_{t}^{})^{\ast}  :  L^{2}(\Xcal, \mu) \rightarrow  L^{p}(\Xcal, \mu)$  satisfies
\[
\Vert (P_{t}^{})^{\ast} \Vert_{L^{2}(\Xcal)\to L^{p}(\Xcal)} \leqslant c_{t,p}, 
\]
where $p$ is the conjugate of $q$. Thus 
\[
\Vert P_{t}^{} \Vert_{L^{2}(\Xcal)\to L^{p}(\Xcal)} \leqslant c_{t,p},
\]
for any $1< p \leqslant \infty$. Let $\varphi_{n}$ be an eigenfunction for $P_{t}^{}$ with the eigenvalue $e^{-\lambda_{n} t}$, then it follows that
\begin{align*}
\Vert \varphi_{n} \Vert_{L^{p} (\Xcal) }\leqslant  c_{t,p} e^{\lambda_{n} t} \Vert \varphi_{n} \Vert_{ L^{2}(\Xcal)},
\end{align*}
and the proof is completed since the left-hand side does not depend on $t$ and the eigenfunctions form an orthonormal system in $L^{2}(\Xcal, \mu )$.

Let us now prove the $L^{1}$-bound. For any conjugate exponents $1< p,q< \infty$, and for every $t>0$ we have that
\begin{align*}
& \Vert \varphi_{n} \Vert_{L^{1} (\Xcal) }  = \int_{\Xcal} \vert \varphi_{n} (x) \vert  d\mu (x)= \int_{\Xcal} e^{\lambda_{n}t} \left| \int_{\Xcal} \varphi_{n} (y) p_{t}^{} (x,y)  d\mu (y) \right|  d\mu (x)
\\
& \leqslant \int_{\Xcal} e^{\lambda_{n}t}  \int_{\Xcal} \vert  \varphi_{n} (y) p_{t}^{} (x,y)  \vert d\mu(y)  d\mu (x) 
\\
& \leqslant    \int_{\Xcal} e^{\lambda_{n}t}  \Vert \varphi_{n} \Vert_{L^{p} (\Xcal)} \Vert p_{t}^{} (x, \cdot) \Vert_{L^{q} (\Xcal)}  d\mu (x).
\end{align*}

We now prove the lower bounds. Let $1\leqslant p < \infty$ and $q>1$ its conjugate.  Then by the $L^{q}$-upper bound and H\"older inequality we have that
\begin{align*}
1 = \Vert \varphi_{n} \Vert^{2}_{L^{2} (\Xcal, \mu )} \leqslant  \Vert \varphi_{n} \Vert_{L^{p} (\Xcal, \mu )} \Vert \varphi_{n} \Vert_{L^{q} (\Xcal, \mu )} \leqslant \Vert \varphi_{n} \Vert_{L^{p} (\Xcal, \mu )} \inf_{t>0} e^{\lambda_{n} t} c_{t,q} (\Xcal).
\end{align*}
If $p=\infty$ then by a standard interpolation inequality we can write
\begin{align*}
&1 = \Vert \varphi_{n} \Vert^{2}_{L^{2} (\Xcal, \mu )} \leqslant  \Vert \varphi_{n} \Vert^{\lambda}_{L^{q} (\Xcal, \mu )} \Vert \varphi_{n} \Vert^{1-\lambda}_{L^{\infty} (\Xcal, \mu )},
\end{align*}
for any $1< q < 2$ and $\lambda := \frac{q}{2}$. The result then follows from the $L^{q}$-upper bound.

(2) By defintion of $M_{\Xcal} (t)$ we have that $p_{t} (x,y) \leqslant M_{\Xcal} (t)$ for any $x,y \in \Xcal$, and then for any $p>1$
\begin{align*}
& c^{2}_{t,p} (\Xcal) = 
\int_{\Xcal} \Vert p_{t}^{}(x, \cdot ) \Vert^{2}_{L^{p} (\Xcal, \mu )} d\mu (x)  
= \int_{\Xcal} \left( \int_{\Xcal} p^{}_{t} (x,y)^{p}  d\mu (y) \right)^{\frac{2}{p}}  d\mu (x)
\\
& \leqslant \int_{\Xcal} \left( \int_{\Xcal} M_{\Xcal} (t)^{p}  d\mu (y) \right)^{\frac{2}{p} }  d\mu (x) = M_{\Xcal} (t)^{2} \mu (\Xcal)^{\frac{2}{p} +1},
\end{align*}
and hence \eqref{e-est-eig-p} follows. The case $p=1$ then follows by taking the limit as $p\rightarrow 1^{+}$.

(3) For any $f\in L^{2}(\Xcal, \mu )$
\[
x\longmapsto \int_{\Xcal} f(y) p_{t_{0}}^{} (x,y)  d\mu (y)
\]
is continuous for $x\in \Xcal$ since $p_{t_{0}}^{} (x,y)$ is continuous  for $x,y \in \Xcal$. Then
\[
\varphi_{n}(x) = e^{\lambda_{n} t_{0}}  \left( T_{t_{0}}^{} \varphi_{n} \right) (x) = e^{\lambda_{n} t_{0}} \int_{\Xcal} \varphi_{n}(y) p_{t_{0}}^{} (x,y)  d\mu (y)
\]
is continuous for $x\in \Xcal$.
\end{proof}

\begin{remark}
In \cite[Equation (4.5.1)]{KigamiBook2001}, J.~Kigami showed that on a certain class of fractals 
\[
\Vert \varphi \Vert_{\infty} \leqslant c \lambda^{\frac{\alpha}{2}} \Vert \varphi \Vert_{2},
\]
where $\varphi$ is an eigenfunction with eigenvalue $\lambda$ and $\alpha$ is some explicit constant depending on the structure of the fractal. If one assumes the bound $p_{t} (x,y) \leqslant c t^{-\gamma}$, then   \eqref{e-est-eig-p} with $p=\infty$ takes the form of \cite[Equation (4.5.1)]{KigamiBook2001}.
\end{remark}

\section{Dirichlet forms satisfying heat kernel estimates}\label{s.Examples}
 
We   start with applications of the results in Section~\ref{s.Preliminaries}  relying on heat kernel estimates such as \eqref{eqn.ultra.contract-Davies}
or \eqref{e.HKPolyUpperBound}.

\subsection{Dirichlet forms following K.-T.~ Sturm}\label{s.SturmDirichlet}

In this section we work under the setting in  \cite{Sturm1996a}, which is a particular case of local Dirichlet forms considered in Section~\ref{s.LocalDF}. Let $\left( \mathcal{E}, \mathcal{D}_{\mathcal{E}} \right)$ be a strongly regular, strongly local Dirichlet form on $L^{2}(\Xcal, \mu)$, and let us denote by $\rho$ its intrinsic metric. We refer to Appendix~\ref{s.RestrictedDirichletForm} for more details. We now consider the metric measure space $(\Xcal, \rho, \mu )$. We additionally assume that $(\Xcal, \rho)$ satisfies Assumption~\ref{assumption.pre-compact.balls}, and therefore it is complete.

\begin{definition}[Measure doubling property]\label{d.doubling}
We say that a measure metric space $\left( \Xcal, d, \mu \right)$ is \emph{doubling} if there exists a constant $D$ such that for any $x \in \Xcal$ and $r>0$ one has that
\begin{equation*}\label{eqn.doubling}
\mu\left( B_{2r} (x) \right) \leqslant 2^{D} \mu \left(B_{r}(x)\right).
\end{equation*}
\end{definition}

\begin{definition}[Weak Poincar\'{e} inequality]\label{assumption.1c}
The Dirichlet form $\mathcal{E}$ on the measure metric space $\left( \Xcal, d, \mu \right)$ satisfies a \emph{weak Poincar\'{e} inequality} if there exists a (Poincar\'{e}) constant $C_{p}$ such that  for all $x \in \Xcal$ and all $r>0$  we have
\begin{equation*}\label{eqn.weak.poincare}
\int_{B_{r}(x)} \vert u - u_{x.r} \vert^{2}  d\mu  \leqslant C_{p} r^{2}  \int_{B_{2r}(x)} d \Gamma (u,u),
\end{equation*}
for all $u\in \mathcal{D}_{\mathcal{E}}$, where $u_{x,r}:= \frac{1}{m\left( B_{r}(x) \right)} \int_{B_{r}(x)} u d\mu $.
\end{definition}

\begin{definition}[Parabolic Harnack inequality]\label{d.Harnack}
The Dirichlet form $\mathcal{E}$ on the measure metric space $\left( \Xcal, d, \mu \right)$ satisfies a \emph{parabolic Harnack inequality} if there exists a constant $C_{H}$ such that  for all $t\in \R$
\begin{equation*}\label{eqn.harnack.parabolic}
\sup_{(s,y) \in Q^{-}} u (s,y) \leqslant C_{H} \inf_{(s,y)\in Q^{+}}  u(s,y),
\end{equation*}
whenever $u$ is a non-negative local solution of the parabolic equation $\left( A -\frac{\partial}{\partial t} \right) u=0$ on $Q:= ( t- 4r^{2}, t ) \times B_{2r} (x)$. Here $A$ is the infinitesimal generator of the Dirichlet form $\mathcal{E}$, and $Q^{-}:= ( t-3r^{2}, t-2r^{2}) \times B_{r}(x)$ and $Q^{+}:= ( t-r^{2}, t) \times B_{r}(x)$.
\end{definition}

Denote $V(r, x):= \mu( B_{r} (x))$, and recall the following definition. 

\begin{definition}\label{def.Ahlfors}
A metric measure space $\left(\Xcal, d, \mu \right)$ is called \emph{Ahlfors $\alpha$-regular} if there exists a constant $c\geqslant 1$ such that 
\begin{align*}
c^{-1} r^{\alpha} \leqslant \mu \left( B_{r} (x) \right) \leqslant c r^{\alpha},
\end{align*}
for all $x\in \Xcal$ and $r\in (0, \text{diam} (\Xcal) )$.
\end{definition}

In the setting of \cite{Sturm1996a}, our results give the following Theorem.

\begin{theorem}\label{thm.main.sturm.setting}
Suppose $\left( \Xcal, d, \mu \right)$ is a complete metric measure space satisfying the measure doubling property. We assume that $\mathcal{E}$ is a strongly regular Dirichlet form satisfying a weak Poincar\'e inequality, and that Assumption \ref{assumption.heat.kernel} is satisfied. Then 

(1) For any open set $\Ucal$ in $\Xcal$ such that $0< \mu (\Ucal) < \infty$, the operator $-A_{\Ucal}$ has a discrete spectrum and  the first eigenvalue $\lambda_{1}^{\Ucal}$ of $-A_{\Ucal}$ is strictly positive. 

(2)  Eigenfunctions of $-A_{\Ucal}$ satisfy $L^{p}$-upper bounds in terms of the $L^{p}$-norm of $V(\sqrt{t}, \cdot )$.

(3) If in addition the metric measure space $\left( \Xcal, d, \mu \right)$ is Ahlfors $\alpha$-regular, then eigenfunctions satisfy the following $L^{p}$ bounds  $\Vert \varphi_{n} \Vert_{L^{p} (\Ucal)} \leqslant C \lambda_{n}^{\frac{\alpha}{2}}$ for some finite constant $C>0$. The eigenfunction expansion \eqref{eqn.heat.kernel.expansion} holds. 

\end{theorem}

\begin{proof}
Let us denote by $p_{t}(x,y)$ the kernel of the semigroup associated to $\mathcal{E}$. Its existence is guaranteed by Assumption \ref{assumption.heat.kernel}. Under our assumptions, we have 
a parabolic Harnack inequality by \cite[Theorem~3.5]{Sturm1996a}, and then  by \cite[Corollary 4.2 and Corollary 4.10]{Sturm1996a}  there exist constants $C_{1}$ and $C_{2}$ such that
\begin{align}
& p_{t} (x,y) \leqslant \frac{C_{1}}{ \sqrt{V(\sqrt{t}, x)}\sqrt{V(\sqrt{t}, y)}} \exp\left( - \frac{\rho^{2} (x,y)}{4t} \right) \left(1+ \frac{\rho (x,y)^{2}}{t} \right)^{\frac{N}{2}}, \label{eqn.gaussian.upper.bound}
\\
& p_{t} (x,y) \geqslant \frac{1}{C_{2}} \frac{1}{ V(\sqrt{t}, x)} \exp\left( -C_{2} \frac{\rho^{2} (x,y)}{t} \right), \label{eqn.gaussian.lower.bound}
\end{align}
uniformly for all $x,y\in \Xcal$, and all $t>0$. Note that the polynomial term in \eqref{eqn.gaussian.upper.bound} can be absorbed into the Gaussian term. More precisely, there exists a $k>0$ such that  for every $\varepsilon >0$ there exists a constant $C_{\varepsilon}$ such that
\begin{align}
p_{t} (x,y) \leqslant \frac{C_{\varepsilon}}{ \sqrt{V(\sqrt{t}, x)}\sqrt{ V(\sqrt{t}, y)}} \exp\left( - \frac{\rho^{2} (x,y)}{(4+\varepsilon)k t} \right).\label{eqn.gaussian.upper.bound.better}
\end{align}
In particular, if $\varepsilon=1$, it follows that
\begin{equation}\label{eqn.nice.upper.bound}
p_{t} (x,y) \leqslant \frac{c \exp\left( - \frac{\rho^{2} (x,y)}{5 t k^{\prime}} \right)}{ \sqrt{V(\sqrt{t}, x)}\sqrt{V(\sqrt{t}, y)}}  =  \frac{c \exp\left( - \frac{\rho^{2} (x,y)}{t k} \right)}{ \sqrt{V(\sqrt{t}, x)}\sqrt{V(\sqrt{t}, y)}}
\end{equation}
for some constant $k>0$, for any $x,y\in \Xcal$.

Then \eqref{eqn.heat.kernel.small.one} is satisfied, and therefore by Proposition~\ref{thm.spectral.gap} the generator $-A_{\Ucal}$ has a discrete spectrum and  and  the first eigenvalue $\lambda_{1}^{\Ucal}$ of $-A_{\Ucal}$ is strictly positive.

To see (2), we can use  \eqref{eqn.nice.upper.bound}  and Theorem~\ref{thm.Lp.bounds} to get $L^{p}$-upper bounds on eigenfunctions in terms of the $L^{p}$-norm of $V(\sqrt{t}, \cdot)$.

Finally for (3) we observe that if $\left(\Xcal, d, \mu \right)$ is Ahlfors $\alpha$-regular for some $\alpha>0$, then the Gaussian bounds become 

\begin{align}
& p_{t} (x,y) \leqslant \frac{C_{1}}{t^{\frac{\alpha}{2}}} \exp\left( - \frac{\rho^{2} (x,y)}{K_{1} t} \right), \label{eqn.gaussian.upper.bound.ahlfors}
\\
& p_{t} (x,y) \geqslant \frac{C_{2} }{ t^{\frac{\alpha}{2}}  }  \exp\left( - \frac{\rho^{2} (x,y)}{K_{2}t} \right) \label{eqn.gaussian.lower.bound.ahlfors}
\end{align}
for some constants $C_{1}, C_{2}, K_{1}, K_{2}$ and for all $x,y \in \Xcal$. These constants depend on the doubling constant and the Poincar\'e constant.  The $L^{p}$-upper bounds on the eigenfunctions then become
\[
\Vert \varphi_{n} \Vert_{L^{p} (\Ucal)} \leqslant C \lambda_{n}^{\frac{\alpha}{2}},
\]
for any $1\leqslant p \leqslant \infty$, and some constant $C$. The eigenfunction expansion for the Dirichlet heat kernel \eqref{eqn.heat.kernel.expansion} holds.  Irreducibility of the semigroup follows from Theorem~\ref{thm.irr.simple.evalue}, see also Theorem~\ref{thm.main.local.dir.forms} in the next section.
\end{proof}

\begin{example}[Riemannian manifolds]
Let $M=\left( M, g \right)$ be a complete Riemannian manifold. We denote by $d=d_{g}$ be the Riemannian distance, and by $m=m_{g}$ the Riemannian volume. Then the canonical energy form $\mathcal{E}$ is given  by 
\[
\mathcal{E} (f, f):= \int_{M} \vert \nabla f \vert^{2}  dm,
\]
where $\nabla=\nabla_{g}$  is the gradient associated to the metric $g$.  Note that the parabolic Harnack  inequality is equivalent to the volume doubling property and the Poincar\'e inequality \cite{Saloff-Coste1992a, Grigoryan1991, ChenLi1997}.  If a Riemannian manifold has non-negative Ricci curvature, then the Riemannian volume measure is uniformly doubling by the Bishop-Gromov comparison theorem. More generally, for a complete separable metric measure space  $\left( M, d, m \right)$, where $m$ is a locally finite measure, the curvature-dimension inequality $\operatorname{CD}\left( K, N \right)$ implies doubling by \cite[Corollary 2.4]{Sturm2006b}. We refer to \cite{Sturm2006b} for more details.

Moreover, if the Ricci curvature is bounded from below then parabolic Harnack inequality holds by \cite{LiYau1986}, and proven later independently by \cite{Grigoryan1991} and \cite{SaloffCoste1995a}. In this setting the heat kernel $p_{t} (x, y)$ always exists and it is a smooth function in $(t,x,y)$ satisfying
\begin{align}
\frac{c_{1}}{V(\sqrt{t}, x)} \exp\left( -\frac{d(x,y)^{2}}{c_{2} t} \right) \leqslant p_{t} (x,y) \leqslant \frac{c_{3}}{V(\sqrt{t}, x)} \exp\left( -\frac{d(x,y)^{2}}{c_{4} t} \right), \label{eqn.bound.Rmn.mnf}
\end{align}
where $V(r,x)$ denotes the volume of the ball $B_{r} (x)$. For Riemannian manifolds \eqref{eqn.bound.Rmn.mnf} is equivalent to the Poincar\'e inequality and the volume doubling property,  \cite{Grigoryan1991, GrigoryanBook2009, LiYau1986, Saloff-Coste1992b}.

Note that if $M$ is compact and connected, then  \eqref{eqn.bound.Rmn.mnf} becomes
\begin{align}\label{eqn.bound.compact.Rmn.mnf}
\frac{c_{1}}{t^{\frac{n}{2}}} \exp\left( -\frac{d(x,y)^{2}}{c_{2} t} \right) \leqslant p_{t} (x,y) \leqslant \frac{c_{3}}{t^{\frac{n}{2}}} \exp\left( -\frac{d(x,y)^{2}}{c_{4} t} \right), 
\end{align}
where $n:= \dim M$, since every compact connected $n$-dimensional Riemannian manifold is Ahlfors $n$-regular. Thus, if $M$ is a compact connected Riemannian manifold, and  $\Ucal$ is an open set in $M$, we recover the well-known results such as existence of the spectral gap $\lambda_{1}^{\Ucal}$  for the Dirichlet Laplacian, and the $L^{p}$-bounds
\[
\Vert \varphi_{k} \Vert_{L^{p} (\Ucal)} \leqslant C \lambda_{k}^{\frac{n}{2}},
\]
for any $1\leqslant p \leqslant \infty$, where $C$ is a constant depending on $m(\Ucal)$ and $\dim M$, but independent of $p$ and $k$. In particular, the eigenfunction expansion of the Dirichlet heat kernel $p^{\Ucal}_{t} (x,y)$ \eqref{eqn.heat.kernel.expansion} holds. We remark that if $M$ is not compact, then Theorem~\ref{thm.Lp.bounds} gives an $L^{p}$-bound for the eigenfunctions in terms of $L^{p}$ norms of $V(\sqrt{t}, \cdot )$. $L^{p}$-bounds for the eigenfunctions are known \cite[Corollary 5.1.2]{SoggeBook2017}, though depending on $p$. Moreover, by \eqref{eqn.bound.compact.Rmn.mnf}  and Theorem  \ref{thm.main.sturm.setting} it follows that the semigroup is irreducible and the spectral gap $\lambda_{1}^{\Ucal}$ is simple and it admits a strictly positive eigenfunction.
\end{example}

\subsection{Local and non-local Dirichlet forms}\label{s.LocalDF}

The goal of this section is to prove irreducibility of the restricted semigroup for both local and non-local Dirichlet forms satisfying upper and lower heat kernel bounds \eqref{eqn.upper.and.lower.estimate.Phi}. By Theorem \ref{thm.irr.simple.evalue} we have that the spectral gap $\lambda_{1}^{\Ucal}$ is simple and the corresponding eigenfunction can be chosen to be a.e. positive. 

\subsubsection{Local Dirichlet forms}
Let $(\Xcal, d, \mu)$ be a locally compact separable metric measure space and $(\mathcal{E}, \mathcal{D}_{\mathcal{E}})$ be a regular local  Dirichlet form on $L^{2} (\Xcal, \mu)$. In many examples the heat kernel of the corresponding self-similar diffusion process satisfies the upper bound
\begin{equation}\label{eq.upper.bound.diffusion}
p_{t}(x,y) \leqslant C t^{-\frac{\alpha}{\beta}} \exp \left( - \left( \frac{d^{\beta} (x,y)}{C t} \right)^{\frac{1}{\beta -1}} \right),
\end{equation}
where $\alpha>0$ turns out to be the \emph{Hausdorff dimension} of $\Xcal$ and $\beta >1$  is known as the \emph{walk dimension} of the heat kernel $p_{t}$, \cite{Barlow1998a, BarlowGrigoryanKumagai2009}. Equation~\eqref{eq.upper.bound.diffusion} implies a Nash inequality \eqref{eqn.Nash.inequal} with $\nu=2\beta/\alpha$ as follows
\begin{equation}\label{eqn.Nash.inequalFractal}
\Vert f \Vert_{L^{2} (\Xcal, \mu)}^{2 + \frac{2\alpha}{\beta}} \leqslant C \mathcal{E} (f,f) \Vert f \Vert_{L^{1} (\Xcal, \mu )}^{\frac{2\alpha}{\beta}},  \text{ for all }  f\in \mathcal{D}_{\mathcal{E}},
\end{equation}
and the corresponding  $L^{1} \rightarrow L^{\infty}$ ultracontractivity of the heat semigroup 
\begin{equation}\label{eqn.ultra.contractFractal}
\Vert P_{t} f \Vert_{L^{\infty} (\Xcal, \mu )} \leqslant C \, t^{-\frac{\alpha}{\beta}} \Vert f \Vert_{L^{1} (\Xcal, \mu )},
\end{equation}
for all $f \in L^{1} (\Xcal, \mu )$ and all $t>0$. Further, it is equivalent to the following heat kernel estimate
\begin{equation}\label{eqn.eqn.bound}
M_{\Xcal}\left( t \right)  \leqslant C \, t^{-\frac{\alpha}{\beta}},
\end{equation}
for all $t>0$, where $M_{\Xcal}\left( t \right)$ is defined in Notation~\ref{n.EssupHK}. We recall that the heat kernel is said to be stochastically complete if 
\begin{align*}
\int_{\Xcal} p_{t}(x,y) d\mu (y) =1, \text{ for $\mu$-a.e. } x \in \Xcal,
\end{align*}

\begin{theorem}\label{thm.main.local.dir.forms}
Let $(\mathcal{E}, \mathcal{D}_{\mathcal{E}})$ be a regular local Dirichlet form on $L^{2} (\Xcal, \mu)$ satisfying a Nash inequality~\ref{eqn.Nash.inequalFractal} with constants $\alpha$ and $\beta$.  Let $\Ucal$ be an open  set in $\mathcal{X}$ such that $0< \mu (\Ucal) < \infty$. Then 

(1) the operator $A_{\Ucal}$ has a discrete spectrum and the first eigenvalue $\lambda_{1}^{\Ucal}$ of $-A_{\Ucal}$ is strictly positive. The eigenfunctions satisfy 
\[
\Vert \varphi_{n} \Vert_{L^{p}} \leqslant c \lambda_{n}^{\frac{\alpha}{\beta}},
\]
for any $1\leqslant p \leqslant \infty$, where $c$ is a constant depending on $\Ucal$, $\alpha$, and $\beta$, but independent of $p$ and $n$. 

(2) Let us assume that $(\Xcal, d, \mu)$ satisfy  the chain condition \cite[Definition 3.3]{GrigoryanKumagai2008}, and $\Ucal$ is  path-connected. If the heat kernel satisfies 
\begin{equation}\label{eqn.upper.and.lower.estimate.Phi}
c_{1} \,  t^{-\frac{\alpha}{\beta}} \,  \Phi \left( c_{2} \, \frac{d(x,y)}{t^{\frac{1}{\beta}}} \right) \leqslant p_{t} (x,y) \leqslant c_{3} \,  t^{-\frac{\alpha}{\beta}} \,  \Phi \left( c_{4} \, \frac{d(x,y)}{t^{\frac{1}{\beta}}} \right),
\end{equation}
for some positive constants $c_{1}, c_{2}, c_{3}, c_{4}$, $\alpha$, and $\beta$, where $\Phi$ is a positive decreasing function on $[0,\infty)$; then $P_{t}^{\Ucal}$ is irreducible. 
\end{theorem}

\begin{proof}
(1)   The spectral gap $\lambda_{1}^{\Ucal} >0$  for $-A_{\Ucal}$ exists by Proposition  \ref{thm.spectral.gap}. The $L^{p}$-bounds for the eigenfunctions follow from Corollary~\ref{cor.bound.efunction} and the heat kernel bound \eqref{eqn.eqn.bound}. In particular, the eigenfunctions expansion of the Dirichlet heat kernel $p^{\Ucal}_{t} (x,y)$ \eqref{eqn.heat.kernel.expansion} holds. 

(2) By \cite[Theorem 4.1]{GrigoryanKumagai2008} there is a $\beta\geqslant 2$ such that

\begin{equation}\label{eqn.bound.exp}
c_{1}t^{-\frac{\alpha}{\beta}}  \exp \left( - c_{2} \left( \frac{d(x,y)^{\beta}}{t} \right)^{\frac{1}{\beta -1}} \right) \leqslant p_{t}(x,y) \leqslant c_{3}t^{-\frac{\alpha}{\beta}}  \exp \left(- c_{4} \left( \frac{d(x,y)^{\beta}}{t} \right)^{\frac{1}{\beta -1}} \right)
\end{equation}
for some positive constants $c_{1},\ldots, c_{4}$. By Theorem \ref{thm.irr.simple.evalue} it is enough to show that for every $y\in \Ucal$ and any $r>0$ small enough and any $x\in B_{r}(y)$ there exists a $t_{0} = t_{0} (x,y,r)$ such that for any $z\in \Ucal^{c}$ and any $0<s<t<t_{0}$
\begin{align}\label{eqn.kinda.there}
p_{t}(x,y) - p_{s} (z,y) >0.
\end{align}
By \eqref{eqn.bound.exp} one has that
\begin{align*}
& p_{t} (x,y) - p_{s} (z,y) \geqslant  c_{1} t^{-\nu} \exp\left(- \frac{a}{t^{\mu}} \right) -  c_{3} s^{-\nu} \exp\left(- \frac{b}{s^{\mu}} \right),
\end{align*}
for any $0<s<t$, and $x,y \in \Ucal$, and $z\in \partial \Ucal$, where
\begin{align*}
\nu:= \frac{\alpha}{\beta}, \quad \mu:= \frac{1}{\beta -1}, \quad a:= c_{2} d(x,y)^{\frac{\beta}{\beta-1}}, \quad b:= c_{4} d (y, \partial \Ucal)^{\frac{\beta}{\beta-1}}.
\end{align*}
Thus,
\begin{align*}
& p_{t} (x,y) - p_{s}(z,y) \geqslant  c_{1} s^{-\nu} \exp\left( - \frac{b}{s^{\mu}} \right) \left[ \left( \frac{s }{t}\right)^{\nu} \exp\left( -\frac{a}{t^{\mu}} + \frac{b}{s^{\mu}} \right) - \frac{c_{3}}{c_{1}} \right],
\end{align*}
and hence it is enough to prove that
\begin{equation}\label{eqn.equation}
\left( \frac{s }{t}\right)^{\nu} \exp\left( -\frac{a}{t^{\mu}} + \frac{b}{s^{\mu}} \right) > \frac{c_{3}}{c_{1}}.
\end{equation}
Let $F(t) : =\left( \frac{s }{t}\right)^{\nu} \exp\left( -\frac{a}{t^{\mu}} + \frac{b}{s^{\mu}} \right)$ for $s$ fixed, then
\begin{align*}
F^{\prime}(t) = \frac{1}{t^{\mu+1}} \left( a\mu -\nu t^{\mu}  \right) F(t) >0,
\end{align*}
for all $t < t_{0} := \left( \frac{a\mu}{\nu} \right)^{\frac{1}{\mu}}$. If we choose $r$ small enough such that for $s$ small enough we have
\begin{align*}
F(s) = \exp\left( \frac{b - a}{s^{\mu}} \right) > \frac{c_{3}}{c_{1}},
\end{align*}
then \eqref{eqn.equation} follows.
\end{proof}

\subsubsection{Non-local Dirichlet forms}\label{sec.lie.groups} 
We now assume that $(\mathcal{E}, \mathcal{D}_{\mathcal{E}})$ is a regular non-local Dirichlet form on $L^{2} (\Xcal, \mu)$, where  $(\Xcal, d, \mu)$ is a locally compact separable metric measure space. For more details on non-local Dirichlet forms we refer to  \cite{BarlowGrigoryanKumagai2009} and references therein. We have the following version of Theorem~\ref{thm.main.local.dir.forms}. 

\begin{theorem}\label{thm.main.non.local.dir.forms}
Let $(\Xcal, d, \mu)$ be a locally compact separable metric measure space and $(\mathcal{E}, \mathcal{D}_{\mathcal{E}})$ be a regular non-local Dirichlet form on $L^{2} (\Xcal,  \mu)$ satisfying Nash inequality \eqref{eqn.Nash.inequalFractal}.  Let $\Ucal$ be an open  set in $\mathcal{X}$ such that $0<\mu (\Ucal) < \infty$, then 

(1) The operator $A_{\Ucal}$ has a discrete spectrum and the first eigenvalue $\lambda_{1}^{\Ucal}$ of $-A_{\Ucal}$ is strictly positive.  The eigenfunctions satisfy 
\[
\Vert \varphi_{n} \Vert_{L^{p}} \leqslant C \lambda_{n}^{\frac{\alpha}{\beta}},
\]
for any $1\leqslant p \leqslant \infty$, where $C$ is a constant depending on $\Ucal$, $\alpha$, and $\beta$, but independent of $p$ and $n$.

(2) Let us assume that $(\Xcal, d, \mu)$ satisfies  the chain condition \cite[Definition 3.3]{GrigoryanKumagai2008}, and $\Ucal$ is  path-connected. If the heat kernel satisfies   \eqref{eqn.upper.and.lower.estimate.Phi}  then $P_{t}^{\Ucal}$ is irreducible. 
\end{theorem}

\begin{proof}
(1) This follows as in Theorem \ref{thm.main.local.dir.forms} part (1).

(2) By \cite[Theorem 4.1]{GrigoryanHuLau2014} we have that $\beta <2$ and
\begin{equation}\label{eqn.bound.polyn}
t^{-\frac{\alpha}{\beta}} \left( 1+ c_{1} \frac{d(x,y)}{t^{\frac{1}{\beta}}} \right)^{- (\alpha + \beta)} \leqslant p_{t}(x,y) \leqslant t^{-\frac{\alpha}{\beta}} \left( 1+ c_{2} \frac{d(x,y)}{t^{\frac{1}{\beta}}} \right)^{- (\alpha + \beta)}.
\end{equation}
By Theorem \ref{thm.irr.simple.evalue} it is enough to show that for every $y\in \Ucal$ and any $r>0$ small enough and any $x\in B_{r}(y)$ there exists a $t_{0} = t_{0} (x,y,r)$ such that for any $z\in \Ucal^{c}$ and any $0<s<t<t_{0}$
\begin{align}\label{eqn.kinda.there2}
p_{t}(x,y) - p_{s} (z,y) >0.
\end{align}
The proof now proceeds as in Theorem \ref{thm.main.local.dir.forms} part (2), and we add it for completeness. By \eqref{eqn.bound.polyn} one has that
\begin{align*}
& p_{t} (x,y) - p_{s} (z,y) \geqslant  t^{-\frac{\alpha}{\beta}} \left( 1+ c_{1} \frac{d(x,y)}{t^{\frac{1}{\beta}}} \right)^{- (\alpha + \beta)}     -   s^{-\frac{\alpha}{\beta}} \left( 1+ c_{2} \frac{d(z,y)}{s^{\frac{1}{\beta}}} \right)^{- (\alpha + \beta)}
\\
& \geqslant  t^{-\frac{\alpha}{\beta}} \left( 1+ c_{1} \frac{d(x,y)}{t^{\frac{1}{\beta}}} \right)^{- (\alpha + \beta)}     -   s^{-\frac{\alpha}{\beta}} \left( 1+ c_{2} \frac{d(y, \partial \Ucal)}{s^{\frac{1}{\beta}}} \right)^{- (\alpha + \beta)},
\end{align*}
since $d( y, \partial \Ucal) \leqslant d( z,y)$ for any $z\in \partial \Ucal$ and $y\in \Ucal$. Set $a:= c_{1} d(x,y)$ and $b:=c_{2} d( y, \partial \Ucal) $. Then
\begin{align*}
& p_{t} (x,y) - p_{s} (z,y) \geqslant t^{-\frac{\alpha}{\beta}} \left( 1+  \frac{a}{t^{\frac{1}{\beta}}} \right)^{- (\alpha + \beta)}     -   s^{-\frac{\alpha}{\beta}} \left( 1+  \frac{b}{s^{\frac{1}{\beta}}} \right)^{- (\alpha + \beta)}
\\
& =  s^{-\frac{\alpha}{\beta}} \left( 1+  \frac{b}{s^{\frac{1}{\beta}}} \right)^{- (\alpha + \beta)} \left[ \frac{t}{s} \left( \frac{ s^{\frac{1}{\beta}} +b }{ t^{\frac{1}{\beta}} +a } \right)^{\alpha + \beta}  -1 \right],
\end{align*}
and hence it is enough to prove that
\begin{equation}\label{eqn.almost.there.pol}
\frac{t^{\nu}}{s^{\nu}} \frac{ s^{\mu} +b }{ t^{\mu} +a } >1,
\end{equation}
where $\nu:= \frac{1}{\alpha +\beta}$ and $\mu:= \frac{1}{\beta}$. Let us fix $s$ and consider the function $F(t) := \frac{t^{\nu}}{s^{\nu}} \frac{ s^{\mu} +b }{ t^{\mu} +a }$ for $s\leqslant t < \infty$. Then
\begin{align*}
& F^{\prime}(t) = \frac{s^{\mu} +b}{s^{\nu}} \cdot\frac{a\nu - (\mu-\nu)t^{\mu}}{t^{\nu+1}},
\end{align*}
which is positive for any $t< t_{0} := \left(\frac{a\nu}{\mu-\nu} \right)^{\frac{1}{\mu}}$. Thus, if we choose $r$ small enough such that
\[
F(s) = \frac{s^{\mu} +b}{ s^{\nu} +a}>1,
\]
that is, $b>a$, then \eqref{eqn.almost.there.pol} follows.
\end{proof}

\subsection{Sub-Riemannian Dirichlet forms on Lie groups} 

In this section we apply Theorem~\ref{thm.main.local.dir.forms} to Lie groups equipped with a left-invariant sub-Riemannian structure. We restrict consideration to this setting, even though one can also treat sub-elliptic operators as in \cite[p. 233]{Sturm1995b} or \cite{KusuokaStroock1988}. Note that in some such settings we do not have a Dirichlet form, as is the case for Grushin operators in \cite{BoscainNeel2020, BoscainPrandi2016}.

Suppose $G$ is a connected Lie group with Lie algebra $\mathfrak{g}$, and let $X_{1}, \ldots, X_{j}$ be left-invariant vector fields satisfying H\"ormander's condition, i.e. the span of their iterated brackets is $\mathfrak{g}$.  By \cite[pp.~950-951]{DriverGrossSaloff-Coste2009a} and later \cite[Section 5.1]{GordinaLaetsch2016a}, we have that the sub-Laplacian 
\begin{equation}\label{e.subLaplacian}
\mathcal{L}:= \sum_{k=1}^{j} X_{k}^{2}
\end{equation}
is  essentially self-adjoint on $L^2(G, m)$ if $m$ is the \emph{right} Haar measure on $G$.  Note that this covers Carnot groups, nilpotent groups, as well as many semi-simple Lie groups such as $\operatorname{SU}\left( 2 \right)$, $\operatorname{SL}\left( 2, \mathbb{R} \right)$ and $\operatorname{SO}\left( n \right)$, $n \geqslant 3$. Collecting these facts, as well as the results in \cite{Siebert1982, McCrudden1984, McCruddenWood1984}, we have the following fact.

\begin{theorem}\label{thm-Lie}
If $G$ is a connected Lie group equipped with a right-invariant Haar measure $m$, then the left-invariant sub-Laplacian $\mathcal{L}$ in \eqref{e.subLaplacian} corresponds to a local regular irreducible Dirichlet form on $L^2(G, m)$.
\end{theorem}

Let us denote by $p_{t}(x,y)$ the kernel of the semigroup associated to the sub-Laplacian $\mathcal{L}$. Then \cite[Theorem 3.4]{DriverGrossSaloff-Coste2009a} proved that there exist positive constants $\kappa, c_{1}, c_{2} , c_{3}$ and a positive integer $\nu \in \mathbb{N}$ such that 
\begin{equation}\label{eqn.heat.kernel.group}
c_{1} t^{-\frac{\nu}{2}} e^{-c_{2} t - c_{2} \frac{d(x,y)^{2}}{t} }   \leqslant p_{t} (x,y) \leqslant \kappa  t^{-\frac{\nu}{2}} e^{ \kappa t - c_{3} \frac{d(x,y)^{2}}{t} } ,
\end{equation}
for all $t>0$ and all $x, y\in G$. We then have the following characterization of sets satisfying \eqref{eqn.heat.kernel.small.one}.

\begin{proposition}
A subset $\Ucal$ of $G$ satisfies \eqref{eqn.heat.kernel.small.one}  if and only if 
\begin{align}
0 < m (\Ucal) < \frac{1}{\sqrt{\kappa}} \left( \frac{\nu}{2\kappa \, e} \right)^{\frac{\nu}{4}} , \label{eqn.good.set}
\end{align}
where $\nu$ and $\kappa$ are given by \eqref{eqn.heat.kernel.group}.
\end{proposition}

\begin{proof}
Let $\Ucal$ be a set satisfying \eqref{eqn.good.set}, then by \eqref{eqn.heat.kernel.group} we have that 
\begin{align*}
& M_{\Ucal} (t) = \operatorname*{ess~sup}_{(x,y) \in \Ucal \times \Ucal} p_{t}^{\Ucal} (x,y) \leqslant \kappa t^{-\frac{\nu}{2}} e^{\kappa t},
\end{align*}
and hence it is enough to show that there exists $t=t_{\Ucal} >0$ such that 
\begin{align*}
& \kappa t^{-\frac{\nu}{2}} e^{\kappa t} < \frac{1}{m(\Ucal)^{2}},  \text{ that is, }  m(\Ucal)^{2} < \frac{1}{\kappa} t^{\frac{\nu}{2}} e^{-\kappa t} =: f(t).
\end{align*} 
The result then follows by noticing that the range of $f$ is given by 
\[
\left( 0, \frac{1}{\kappa} \left( \frac{\nu }{2\kappa \, e} \right)^{\frac{\nu}{2}}  \right].
\]

On the other hand, if condition \eqref{eqn.good.set} does not hold, then 
\begin{align*}
& m(\Ucal)^{2} >  \frac{1}{\kappa} \left( \frac{\nu }{2\kappa\, e} \right)^{\frac{\nu}{2}}  \geqslant \frac{1}{\kappa} t^{\frac{\nu}{2}} e^{-\kappa t} \geqslant \frac{1}{M_{\Ucal} (t)} \quad \text{ for any } t>0,
\end{align*}
that is, $M_{\Ucal} (t) > \frac{1}{ m(\Ucal)^{2}}$ for any $t>0$.  
\end{proof}

The following is the main result of this section. 

\begin{theorem}\label{thm.main.local.dir.forms.lie.groups}
Let $G$ be a Lie group and $(\mathcal{E}, \mathcal{D}_{\mathcal{E}})$ be the Dirichlet form given by Theorem \ref{thm-Lie} with the infinitesimal generator $\mathcal{L}$. Let $\Ucal$ be an open  set in $G$ satisfying \eqref{eqn.good.set}. Then   

(1) The operator $A_{\Ucal}$ has a discrete spectrum and the first eigenvalue $\lambda_{1}^{\Ucal}$ of $-A_{\Ucal}$ is strictly positive. The eigenfunctions satisfy 
\begin{align*}
& \Vert \varphi_{n} \Vert_{L^{p} (\Ucal, m )} 
\leqslant a_{\Ucal, p} \left( \lambda_{n} + \kappa \right)^{\frac{\nu}{2}}  \text{ for any }  1 < p \leqslant \infty,
\\
& \Vert \varphi_{n} \Vert_{L^{1} (\Ucal,  m )}  \leqslant b_{\Ucal} \left( \lambda_{n} + \kappa \right)^{\nu}, 
\end{align*}
where $\nu$ and $\kappa$ are given by \eqref{eqn.heat.kernel.group}, and  $a_{\Ucal, p}$ and $b_{\Ucal}$ are explicit constants not depending on $n$.

(2) Let us assume that $\Ucal$ is also path-connected. Then $P_{t}^{\Ucal}$ is irreducible. 
\end{theorem}

\begin{proof}

(1) By \eqref{eqn.heat.kernel.group} it follows that 
\begin{align*}
M_{\Ucal} (t) \leqslant \kappa  t^{-\frac{\nu}{2}} e^{ \kappa t  } 
\end{align*}
for all $t>0$, and hence   the semigroup $P_{t}^{\Ucal}$ is compact for all $t>0$ by Lemma \ref{lem-comp}. By \eqref{eqn.heat.kernel.group} and Proposition \ref{thm.spectral.gap} it then follows that  $-\mathcal{L}_{\Ucal}$ has a discrete spectrum and $\lambda_{1}^{\Ucal} >0$. Let us now prove the eigenfunctions bound. By \eqref{eqn.heat.kernel.group} we get that 
\begin{align*}
\Vert p_{t}^{\Ucal} (x, \cdot) \Vert^{p}_{L^{p} (\Ucal, m )} \leqslant \int_{\Ucal} \kappa^{p} t^{- p \frac{\nu}{2}} e^{\kappa p t } dm(y) = m(\Ucal) \kappa^{p} t^{- p \frac{\nu}{2}} e^{\kappa p t },
\end{align*}
and hence $c_{t,p,q} (\Ucal)$ defined in Theorem \ref{thm.Lp.bounds}  satisfies 
\begin{align*}
c_{t,p,q} (\Ucal) = \left( \int_{\Ucal} \Vert p_{t}^{\Ucal} (x, \cdot) \Vert^{p}_{L^{p} (\Ucal, m )} dm(x) \right)^{\frac{1}{q}} \leqslant m(\Ucal)^{\frac{1}{p} + \frac{1}{q}} \kappa^{\frac{1}{q}} t^{-\frac{\nu}{2}} e^{\kappa t},
\end{align*}
and 
\begin{align*}
& A_{p, q, \lambda_{n} } := \inf_{t>0} e^{\lambda_{n}t} c_{t,p,q} \leqslant m(\Ucal)^{\frac{1}{p}+ \frac{1}{q}} \kappa^{\frac{1}{q}}  \left( \frac{2e}{\nu} \right)^{\frac{\nu}{2}} \left( \lambda_{n} + \kappa \right)^{\frac{\nu}{2}},
\\
& \inf_{1< p, q <\infty , \, \frac{1}{p} + \frac{1}{q} =1} \Vert \varphi_{n} \Vert_{L^{p} (\Ucal ,m)} A_{q,1,\lambda_{n}} \leqslant  m(\Ucal)^{\frac{5}{2}} \kappa^{\frac{3}{2}}  \left( \frac{2e}{\nu} \right)^{\nu} \left( \lambda_{n} + \kappa \right)^{\nu}.
\end{align*}
The result then follows by  Theorem \ref{thm.Lp.bounds}.

(2) By Theorem \ref{thm.irr.simple.evalue} it is enough to show that for every $y\in \Ucal$ and any $r>0$ small enough and any $x\in B_{r}(y)$ there exists a $t_{0} = t_{0} (x,y,r)$ such that for any $z\in \Ucal^{c}$ and any $0<s<t<t_{0}$
\begin{align*}
p_{t}(x,y) - p_{s} (z,y) >0.
\end{align*}
By \eqref{eqn.heat.kernel.group}  one has that 
\begin{align*}
& p_{t} (x,y) - p_{s} (z,y)  \geqslant c_{1} t^{-\frac{\nu}{2}} e^{- c_{2} t - c_{2} \frac{d(x,y)^{2}}{t}} - \kappa s^{-\frac{\nu}{2}} e^{ \kappa s - c_{3} \frac{d(z,y)^{2}}{s} }
\\
& \geqslant c_{1} t^{-\frac{\nu}{2}} \exp \left( - at - \frac{b}{t} \right) -  \kappa s^{-\frac{\nu}{2}}  \exp\left( \kappa s - \frac{d}{s} \right)
\\
& = c_{1} s^{-\frac{\nu}{2}} \exp\left( \kappa s - \frac{d}{s} \right) \left( \left( \frac{s}{t} \right)^{\frac{\nu}{2}} \exp\left( - at - \frac{b}{t} + \frac{d}{s} - \kappa s \right) - \frac{\kappa}{c_{1}} \right)  ,
\end{align*} 
where 
\begin{align*}
& a:= c_{2},  & b:= c_{2} d(x,y)^{2}, && d:= c_{3} d(y, \partial \Ucal)^{2}.
\end{align*}
Thus, it is enough to show that 
\begin{align}
& \left( \frac{s}{t} \right)^{\frac{\nu}{2}} \exp\left( - at - \frac{b}{t} + \frac{d}{s} - \kappa s \right) > \frac{\kappa}{c_{1}}\label{eqn.another.eqn}
\end{align}
for small $s<t$. For $t>s$,  the function 
\begin{align*}
F(t) := \left( \frac{s}{t} \right)^{\frac{\nu}{2}} \exp\left( - at - \frac{b}{t} + \frac{d}{s} - \kappa s \right)
\end{align*}
satisfies 
\begin{align*}
F^{\prime} (t) = \frac{1}{2t^{2}} \left( 2b - 2at^{2} -t\nu \right) F(t),
\end{align*}
which is positive for 
\begin{align*}
0 < t< t_{0} := \frac{\sqrt{\nu^{2} + 16 ab} - \nu} {4a}.
\end{align*}
Thus, if we choose $r$ small enough such that for $s$ small enough we have
\begin{align*}
F(s)  > \frac{\kappa}{c_{1}},
\end{align*}
then \eqref{eqn.another.eqn} follows.

\end{proof}

\begin{example}[Diffusions on closed subgroups of $\operatorname{GL}(\R^d)$]
To construct a   diffusion on  a closed subgroup $G$ of $\operatorname{GL}(\R^d)$ we can use \cite[Section 6.1, page 140, and references therein]{BougerolLacroixBook} (see also \cite[Proposition 2]{Gordina2003b}, and for infinite-dimensional groups \cite[Lemma 6.1]{Gordina2000b} and \cite[Theorem 3.3]{Gordina2000a}). Let 
\[ 
\mathfrak{g}:=Lie(A_1,...,A_k)   
 \]
be the Lie algebra   generated by $A_1,...,A_k\in \mathfrak {gl}(\R^d)$, and let $G$ be the corresponding Lie subgroup of $\operatorname{GL}(\R^d)$.
For any initial point $X_{0}:=x\in G$  we define the diffusion $X_t$ as the solution to the Stratonovich stochastic differential equation 
\begin{equation}\label{eq-sde}
X_t^{-1}d X_t=  \sum\limits_{i=1}^{k}A_i\circ dW_t^i
\end{equation}
where $W_t^1,...,W_t^k$ are independent identically distributed real-valued Brownian motions.
\end{example}

\begin{proposition}\label{prop-sde}
The diffusion  process $X_t$ corresponds to the local regular irreducible Dirichlet form on $G$   in Theorem~\ref{thm-Lie}.
\end{proposition}
\begin{proof}
We abuse notion and denote by $A_1,...,A_k$ left-invariant vector fields on $G$ corresponding to the matrices	$A_1,...,A_k$. These vector fields satisfy H\"{o}rmander's condition on $G$ by construction, and therefore we have a natural sub-Riemannian metric determined by $A_1,...,A_k$ with the corresponding horizontal gradient, and a sub-Laplacian on $G$, which is the generator of the semigroup of $X_t$.
\end{proof}

\section{Small deviations}\label{s.small.deviations}

In this section we apply spectral results to prove  small deviation principles. We give several proofs under different assumptions. We first consider the case when the process satisfies some scaling properties. Then we assume that we have a group $G$ action on the metric space $\Xcal$ and introduce $G$-dilation structures inspired by Varopoulos in \cite{Varopoulos1985a}.

Note that implicitly we only use spectral properties of the restriction to a metric ball. It is natural to ask if these metric balls have nicer properties than an arbitrary open connected set. In general we do not have such information. Even in the case of Carnot groups we only know that metric balls in a step $2$ Carnot groups are uniform domains as shown in a recent paper \cite{Greshnov2018a}, for higher step it is still an open question, and therefore we cannot rely on results such as \cite{GyryaSaloff-Coste2011}.

\subsection{Self-similar processes}
Let $X^{x}_{t}$ be an $\Xcal$-valued Hunt process with $X^{x}_{0} = x \in \Xcal$ a.s. and whose transition density is given by the heat kernel $p_{t}(x, \cdot )$. Inspired by Sato \cite{SatoKI1991} we give the following definition, though for now we restrict consideration to self-similarity with respect to the metric distance. 

\begin{definition}\label{def.self-similar}
An $\Xcal$-valued stochastic process  $X^{x}_{t}$  process with $X^{x}_{0} = x \in \Xcal$  is called \emph{(distance-)self-similar} if there is a $\beta>0$ such that
\begin{equation}\label{e-scale-l}
d (X^{x}_{t\varepsilon},x) \stackrel{(d)}{=} \varepsilon^{\frac{1}{\beta}} d(X^{x}_{t}, x),
\end{equation}
for any $\varepsilon > 0$ and any $x \in \Xcal$.
\end{definition}

\begin{remark}
The choice $f(\varepsilon) =\varepsilon^{\frac{1}{\beta}}$  is consistent with the estimate of the mean exit time
\begin{equation}\label{e-time}
c_{1} \varepsilon^{\beta} \leqslant \E^{x} \left[ \tau_{B_{\varepsilon} (x)} \right] \leqslant c_{2} \varepsilon^{\beta},
\end{equation}
which is known for many self-similar fractals, \cite[Theorem 4.3, Corollary 2.3]{Barlow2003a} \cite{BarlowBassKumagai2006, BarlowMurugan2018, BarlowBass2004} and it is stable under rough isometries. We use such a function as we do not have examples with a different choice of $f$. 
\end{remark}

\begin{theorem}\label{thm.small.dev}
Let $\{P_{t}\}_{t\geqslant 0}$ be a strongly continuous contraction semigroup on $L^{2} (\Xcal, \mu)$. Let $x\in \Xcal$ and assume that $P_{t}^{B_{1}(x)}$ is irreducible. Suppose there exist constants $\alpha, \beta >0$ such that 
\[
p_{t} (x,y) \leqslant c \, t^{-\frac{\alpha}{\beta}}
\]
for every $t, x$, and $y$. Moreover, assume that there exists a $t_{0}$ such that $p_{t_{0}}(x,y)$ is continuous for all $x, y \in \Xcal$.  If $X_{t}^{x}$ is self-similar,  then
\begin{equation}\label{eqn.small.dev}
\lim_{\varepsilon \rightarrow 0} e^{\lambda_{1} \frac{t}{\varepsilon^{\beta}}} \Prob^{x} \left( \sup_{0\leqslant s \leqslant t} d( X_{s}, x) <\varepsilon \right) = c_{1} \varphi_{1} (x),
\end{equation}
where $\lambda_{1}>0$ is the spectral gap of $A$ restricted to the unit ball $B_{1}(x)$, and  $\varphi_{1}$ is the corresponding positive eigenfunction and  $c_{1} = c_{1} ( B_{1} (x) )$ is given by \eqref{eqn.c.n}.
\end{theorem}

\begin{proof}
By the results from Section~\ref{s.Preliminaries} we have that the generator $A_{B_{1} (x)}$ of $P^{B_{1}(x)}_{t}$ has a discrete spectrum. In particular, the first eigenvalue $\lambda_{1} =\lambda_{1}^{B_{1} (x)}$ is simple, positive, and  the corresponding eigenfunction $\varphi_{1}$ is continuous and it can be chosen to be strictly positive. By \eqref{e-scale-l} and Proposition \ref{prop.eigenfunctions.expansion} we have that
\begin{align*}
&\Prob^{x} \left( \sup_{0\leqslant s \leqslant t} d( X_{s}, x) <\varepsilon \right)= \Prob^{x} \left( \sup_{0\leqslant s \leqslant t} \varepsilon^{-1} d( X_{s}, x) <1 \right)= \Prob^{x} \left( \sup_{0\leqslant s \leqslant t} d( X_{s \slash \varepsilon^{\beta}}, x) <1 \right)
\\
& = \Prob^{x} \left( \sup_{0\leqslant s \leqslant t \slash \varepsilon^{\beta} } d( X_{s }, x) <1 \right) =  \Prob^{x} \left( \tau_{B_{1} (x)} >\frac{t}{\varepsilon^{\beta}} \right)  = \sum_{n=1}^{\infty} e^{-\lambda_{n} \frac{t}{\varepsilon^{\beta}}} c_{n} \varphi_{n}(x),
\end{align*}
and the result follows since $0< \lambda_{1} < \lambda_{2} \leqslant \lambda_{3} \leqslant \cdots$, and $c_{1} ( B_{1} (x)) \varphi_{1} (x)$ is well-defined since $\varphi_{1}$ is continuous at $x$.
\end{proof}

\begin{proposition}
Let $(\Xcal, \mu)$ be an Ahlfors $\alpha$-regular space, and let $\{P_{t}\}_{t\geqslant 0}$ be a strongly continuous contraction  semigroup on $L^{2} (\Xcal, \mu)$. Assume that for some $x\in \Xcal$ we have that $P_{t}^{B_{1}(x)}$ is irreducible, and that there exist constants $\alpha, \beta >0$ such that 
\[
p_{t} (x,y) \leqslant c \, t^{-\frac{\alpha}{\beta}}
\]
for every $t, x$ and $y$. Moreover, assume that there exists a $t_{0}$ such that $p_{t_{0}}(x,y)$ is continuous for all $x, y \in \Xcal$.  Then the operator $A$ restricted to $B_{\varepsilon} (x)$ admits an eigensystem $\{\lambda_{n} (\varepsilon), \varphi_{n}^{\varepsilon} \}_{n=1}^{\infty}$. If there exists a constant $c>1$ such that
\begin{equation}\label{e-scale-n-approx}
c^{-1} \dfrac{\lambda_{n}(1)}{\varepsilon^{\beta}} \leqslant
    \lambda_{n}(\varepsilon) \leqslant c \dfrac{\lambda_{n}(1)}{\varepsilon^{\beta}},
\end{equation}
for all $n$, then
\begin{align}
& \limsup_{\varepsilon \rightarrow 0} e^{c^{-1} \frac{\lambda_{1}(1)}{\varepsilon^{\beta}}t }\Prob^{x} \left( \sup_{0\leqslant s \leqslant t} d( X_{s}, x) <\varepsilon \right) \leqslant C c \left( \frac{e \alpha}{\beta} \right)^{\frac{\alpha}{\beta}}  \label{eqn.upper.sd}
\\
&\liminf_{\varepsilon \rightarrow 0}  e^{c \frac{\lambda_{1}(1)}{\varepsilon^{\beta}}t }\Prob^{x} \left( \sup_{0\leqslant s \leqslant t} d( X_{s}, x) <\varepsilon \right) > 0,    \label{eqn.lower.sd}
\end{align}
where $C$ is such that  $m ( B_{\varepsilon} (x) ) \leqslant C \varepsilon^{\alpha}$ for any $\varepsilon >0$.
\end{proposition}

\begin{proof}
For each $\varepsilon >0$, by the results in Section \ref{s.Preliminaries} applied to $\Ucal = B_{\varepsilon}(x)$ it follows that 
\begin{align*}
& \Prob^{x} \left( \sup_{0\leqslant s \leqslant t} d( X_{s}, x) <\varepsilon \right) = \sum_{n=1}^{\infty} e^{- \lambda_{n} ( \varepsilon) t} c_{n} (\varepsilon) \varphi_{n}^{\varepsilon} (x)  \leqslant  \sum_{n=1}^{\infty}  e^{-c^{-1} \frac{\lambda_{n}(1)}{\varepsilon^{\beta}}t } \vert c_{n} (\varepsilon) \varphi_{n}^{\varepsilon} (x) \vert,
\end{align*}
that is,
\begin{align*}
& e^{c^{-1} \frac{\lambda_{1}^{\Ucal}(1)}{\varepsilon^{\beta}}t }\Prob^{x} \left( \sup_{0\leqslant s \leqslant t} d( X_{s}, x) <\varepsilon \right) \leqslant c_{1} (\varepsilon) \varphi_{1}^{\varepsilon} (x) + \sum_{n=2}^{\infty} e^{-\frac{c^{-1}t}{\varepsilon^{\beta}} ( \lambda_{n} (1) - \lambda_{1}^{\Ucal} (1) ) }\vert c_{n} (\varepsilon) \varphi_{n}^{\varepsilon} (x) \vert,
\end{align*}
where $c_{n} (\varepsilon) = c_{n} ( B_{\varepsilon} (x) )$ is given by \eqref{eqn.c.n}.

Note that
\[
\lim_{\varepsilon \rightarrow 0}\sum_{n=2}^{\infty} e^{-\frac{c^{-1}t}{\varepsilon^{\beta}} ( \lambda_{n} (1) - \lambda_{1}^{\Ucal} (1) ) } \vert c_{n} (\varepsilon) \varphi_{n}^{\varepsilon} (x)  \vert =0.
\]
Indeed, by Corollary \ref{cor.bound.efunction} we have that
\begin{align*}
& \Vert \varphi_{n}^{\varepsilon} \Vert_{L^{\infty} (B_{\varepsilon}(x))}  \leqslant d\, \max \left( \mu (B_{\varepsilon} (x), \sqrt{\mu (B_{\varepsilon} (x)}  \right)  \, \lambda_{n} (\varepsilon)^{\gamma} = d \,  \sqrt{\mu (B_{\varepsilon} (x)} \, \lambda_{n} (\varepsilon)^{\gamma},
\\
&\vert  c_{n}(\varepsilon) \vert  \leqslant   \mu (B_{\varepsilon} (x)) \Vert \varphi_{n}^{\varepsilon} \Vert_{L^{\infty} (B_{\varepsilon}(x))}  \leqslant d \, \mu (B_{\varepsilon} (x))^{\frac{3}{2}}  \, \lambda_{n} (\varepsilon)^{\gamma},
\end{align*}
where $ \gamma := \frac{\alpha}{\beta}, \, d=(e\gamma)^{\gamma}$, and hence, with $a_{n, t} := c^{-1} t  (  \lambda_{n} (1) - \lambda_{1}^{\Ucal} (1) )>0$,
\begin{align*}
& \lim_{\varepsilon \rightarrow 0}\sum_{n=2}^{\infty} e^{-\frac{a_{n,t}}{\varepsilon^{\beta}} } \vert c_{n} (\varepsilon) \varphi_{n}^{\varepsilon} (x) \vert  \leqslant d^{2} \lim_{\varepsilon \rightarrow 0}  \sum_{n=2}^{\infty} e^{-\frac{a_{n,t}}{\varepsilon^{\beta}}  }   \mu (B_{\varepsilon} (x) )^{2} \lambda_{n} (\varepsilon)^{2\gamma}
\\
& \leqslant c d^{2} \lim_{\varepsilon \rightarrow 0} \sum_{n=2}^{\infty} e^{-\frac{a_{n,t}}{\varepsilon^{\beta}} } \frac{\mu (B_{\varepsilon} (x) )^{2}}{\varepsilon^{2\alpha}}.
\end{align*}
Note that the series converges uniformly  and
\[
\lim_{\varepsilon \rightarrow 0} \frac{\mu(B_{\varepsilon} (x) )^{2}}{\varepsilon^{2\alpha}} e^{-\frac{a_{n,t}}{\varepsilon^{\beta}}} =0,
\]
and hence
\begin{align*}
& \limsup_{\varepsilon \rightarrow 0} e^{c^{-1} \frac{\lambda_{1}^{\Ucal}(1)}{\varepsilon^{\beta}}t }\Prob^{x} \left( \sup_{0\leqslant s \leqslant t} d( X_{s}, x) <\varepsilon \right) \leqslant \limsup_{\varepsilon \rightarrow 0}  c_{1} (\varepsilon) \varphi_{1}^{\varepsilon} (x)
\\
& \leqslant \limsup_{\varepsilon \rightarrow 0} c d^{2}   \frac{m (B_{\varepsilon} (x) )^{2}}{\varepsilon^{2\alpha}} \leqslant c d^{2} C .
\end{align*}

Let us now prove \eqref{eqn.lower.sd}. Note that if 
\begin{align*}
&  \liminf_{\varepsilon \rightarrow 0}   e^{c \frac{\lambda_{1}^{\Ucal}(1)}{\varepsilon^{\beta}}t }\Prob^{x} \left( \sup_{0\leqslant s \leqslant t} d( X_{s}, x) <\varepsilon \right) = \liminf_{\varepsilon \rightarrow 0}   e^{c \frac{\lambda_{1}^{\Ucal}(1)}{\varepsilon^{\beta}}t } \int_{B_{\varepsilon} (x)} p_{t}^{B_{\varepsilon} (x)}  (x,y)  d\mu (y)
\\
& = \sup_{\delta >0} \inf_{\varepsilon < \delta } e^{c \frac{\lambda_{1}^{\Ucal}(1)}{\varepsilon^{\beta}}t } \int_{B_{\varepsilon} (x)} p_{t}^{B_{\varepsilon} (x)}  (x,y)  d\mu (y) =0,
\end{align*}
then for every $\delta$ there exists an $\varepsilon_{\delta} < \delta$ such that 
\begin{align*}
e^{c \frac{\lambda_{1}^{\Ucal}(1)}{\varepsilon^{\beta}}t } \int_{B_{\varepsilon} (x)} p_{t}^{B_{\varepsilon} (x)}  (x,y)  d\mu (y) =0.
\end{align*}
That is, for every $\delta$ there exists an $\varepsilon_{\delta} < \delta$ such that 
\begin{align*}
 \int_{B_{\varepsilon} (x)} p_{t}^{B_{\varepsilon} (x)}  (x,y)  d\mu (y) =0,
\end{align*}
which leads to a contradiction by Theorem \ref{thm.irr.simple.evalue}.
\end{proof}

\subsection{Continuous dilations}\label{s.Dilations}

In this subsection we study continuous families of dilations.  We will consider more general dilations in the next section. A family of homeomorphisms  $\{ \delta_{r} \}_{r>0}$ is called a \emph{continuous $\beta$-dilation structure on the Dirichlet space} $(\Xcal, \mu, \mathcal{E}, A)$ if
\begin{align*}
& \delta_{r} : \Xcal \rightarrow \Xcal \text{ for every }  r>0,
\\
& \delta_{r} \circ \delta_{s} = \delta_{rs} \, \text{ for every } \; r, s >0,
\\
& \delta_{1} = \operatorname{Id}, \quad \delta_{r} \rightarrow \text{Id} \; \text{ as } \; r\rightarrow 1 \; \text{ uniformly on compact sets,}
\\
& \delta_{r} (\mathcal{D}_{\mathcal{E}} ) = \mathcal{D}_{\mathcal{E}}, \quad \delta_{r} ( \mathcal{D}_{A} )=  \mathcal{D}_{A} \, \text{ for every } \; r>0,
\\
& \int_{\Xcal} ( f\circ \delta_{r}) (x) d\mu (x) = r^{-d} \int_{\Xcal} f(x) d\mu(x), \; \text{ for every } f \in L^{1} \Xcal, \mu ),
\\
& \mathcal{E} (f \circ \delta_{r}, g\circ \delta_{r} ) = r^{-d + \beta} \mathcal{E} (f, g), \; \text{ for every } f, g  \in \mathcal{D}_{\mathcal{E}},
\end{align*}
for some $d, \beta \in \R$. Note that when $\beta =2$, this notion was introduced by Varopoulos in \cite{Varopoulos1985a}.

If $\{ \delta_{r} \}_{r>0}$ is a dilation structure on $(\Xcal, \mu, \mathcal{E}, A)$, then
\begin{equation}\label{eqn.smigroup.factor}
P_{t} = \delta_{r} \circ P_{r^{\beta} t} \circ \delta_{r^{-1}},
\end{equation}
where $\delta_{r}$ acts on functions by $(\delta_{r} \circ f) (x) := f(\delta_{r} (x))$, see \cite{Varopoulos1985a}. If $P_{t}$ admits a heat kernel, then by \eqref{eqn.smigroup.factor} we have that
\begin{align*}
& \varepsilon^{d} p_{t} \left( \delta_{\varepsilon} (x), \delta_{\varepsilon} (y) \right) = p_{\frac{t}{\varepsilon^{\beta}}} (x,y),
\\
& \delta_{\varepsilon} \left( X^{x}_{\frac{t}{\varepsilon^{\beta}}} \right) \stackrel{(d)}{=} X^{
\delta_{\varepsilon} (x)}_{t} \; \text{ for all } \; x\in \Xcal,
\\
& \Prob^{x} \left( \tau_{\Ucal} > \frac{t}{\varepsilon^{\beta}} \right) = \Prob^{\delta_{\varepsilon}(x)} \left( \tau_{ \delta_{\varepsilon}(\Ucal)} > t \right), \; \text{ for any } x\in \Ucal,
\end{align*}
where $\Ucal$ is a subset of $\Xcal$.

\begin{theorem}\label{thm.small.dev.2}
Let $\{P_{t}\}_{t\geqslant 0}$ be a strongly continuous contraction semigroup on $L^{2} (\Xcal, \mu)$. Let $\Ucal$ be an open set in $\Xcal$ such that $0<\mu (\Ucal) < \infty$, and $x\in \Ucal$, and assume that $P_{t}^{\Ucal}$ is irreducible. Assume that  there exists a $\gamma >0$ such that 
\[
M_{\Xcal} (t) \leqslant c \, t^{-\gamma}
\]
for all $t>0$. Moreover, assume that there exists a $t_{0}$ such that $p_{t_{0}}(x,y)$ is continuous for $x, y \in \Xcal$.  If $(\Xcal, \mu)$ admits a
continuous $\beta$-dilation structure $\{ \delta_{r} \}_{r>0}$, then
\begin{equation}\label{eqn.small.dev.2}
\lim_{\varepsilon \rightarrow 0} e^{\lambda_{1}^{\Ucal} \frac{t}{\varepsilon^{\beta}}} \Prob^{\delta_{\varepsilon} (x)} \left( \tau_{\Ucal_{\varepsilon}} >t \right) = c_{1} \varphi_{1} (x),
\end{equation}
where $\lambda_{1}^{\Ucal}>0$ is the first eigenvalue  of $A_{\Ucal}$, and  $\varphi_{1}$ is the corresponding positive eigenfunction, and $c_{1} = c_{1} (\Ucal)$ is given by \eqref{eqn.c.n}, and  $\Ucal_{\varepsilon}:= \delta_{\varepsilon} (\Ucal)$.
\end{theorem}

\begin{proof}
By the results from Section \ref{s.Preliminaries} we have that the generator $A_{\Ucal}$ of $P^{\Ucal}_{t}$ has a discrete spectrum. In particular, the first eigenvalue $\lambda_{1}^{\Ucal}$ is simple, positive, and  its first eigenfunction $\varphi_{1}$ is continuous and it can be chosen to be strictly positive. By  Proposition \ref{prop.eigenfunctions.expansion} we have that
\begin{align*}
&  \Prob^{\delta_{\varepsilon} (x)} \left( \tau_{\Ucal_{\varepsilon}} >t \right) = \Prob^{x} \left( \tau_{\Ucal} > \frac{t}{\varepsilon^{\beta}} \right)
= \sum_{n=1}^{\infty} e^{-\lambda_{n} \frac{t}{\varepsilon^{\beta}}} c_{n} \varphi_{n}(x),
\end{align*}
and the result follows since $0< \lambda_{1}^{\Ucal} < \lambda_{2} \leqslant \lambda_{3} \leqslant \cdots$, and $c_{1} \varphi_{1} (x)$ is well defined since $\varphi_{1}$ is continuous at $x$.
\end{proof}

\subsection{Group actions on Dirichlet metric measure spaces}

In this section we consider how a group action can be used to prove a small deviation principle, which is the content of Theorem~\ref{thm.small.dev.3}. Recall that in Section~\ref{s.Dilations} we considered a continuous family of dilations, and while fractals  generally do not admit such a family, we might have  \emph{discrete dilation structures} with a scaling factor $r_{0}>0$. More precisely,  dilations $\delta_{r}$ are defined for all $r=r_{0}^k$, $k\in\mathbb{Z}$, see e.~g.  \cite[Proposition 3.3.1]{KigamiBook2001} and \cite[(1.4.8)]{StrichartzBook2006}, \cite[Equation(1.2)]{BarlowPerkins1988}, and \cite{FitzsimmonsHamblyKumagai1994}.  

We introduce a general notion of a group of dilations acting on a Dirichlet space below. If $G=\R_{>0}$, then we recover continuous dilations in Section~\ref{s.Dilations}, while $G=\mathbb{Z}$ corresponds to discrete dilations.  Our idea is motivated by the notion of \emph{dilation structure} introduced by Varopoulos \cite{Varopoulos1985a}, and it coincides with \cite{Varopoulos1985a} when $G=\R_{>0}$ is the multiplicative group. For measurable group actions we follow the exposition in \cite{Gordina2017}.

Let $G$ be a topological group acting measurably on a metric measure space $\left( \mathcal{X}, d, \mu \right)$, that is, there is a measurable map
\[
\Phi: G\times \Xcal \longrightarrow \Xcal, (g, x) \longrightarrow \Phi_{g}(x) =: x_{g}
\]
such that
\begin{align*}
&\Phi_{e} (x) = x  \text{ for } \mu-\text{a.e. }  x \in \Xcal
\\
& \Phi_{g} (\Phi_{h} (x) ) = \Phi_{gh} (x)  \text{ for all } g, h\in G,  \text{ and } \mu-\text{a.e. }  x\in \Xcal.
\end{align*}
Here $e$ is the identity element in $G$. Such a group action induces an action on $L^{2}$-functions on $G$ as follows
\begin{align*}
\widetilde{\Phi}: G \times L^{2} (\Xcal, \mu ) \longrightarrow L^{2} (\Xcal, \mu ), \; (\widetilde{\Phi}_{g} f ) (x) = f( \Phi_{g} (x)).
\end{align*}
Let us denote by $(\Phi_{g})_{\ast}m$ the pushforward of $m$ under $\Phi_{g}: \Xcal \rightarrow \Xcal$.

\begin{definition}
The action of a group  $G$ on a Dirichlet metric measure space $(\Xcal, d, \mu, \mathcal{E} )$  is said to be a \emph{group action preserving the Dirichlet space class} if it is a measurable action on $\left( \mathcal{X}, d, \mu \right)$ such that 
\begin{align*}
& \widetilde{\Phi}_{g} ( \mathcal{D}_{\mathcal{E}} )= \mathcal{D}_{\mathcal{E}},\quad   \widetilde{\Phi}_{g} ( \mathcal{D}_{A} )= \mathcal{D}_{A},
\end{align*}
$(\Phi_{g})_{\ast}\mu$ and $\mu$ are  mutually absolutely continuous for all $g \in G$, and the Radon-Nikodym derivative
\begin{equation}\label{eqn.RN.der.indep}
J_{g} := \frac{d (\Phi_{g})_{\ast} \mu }{ d\mu } (x)
\end{equation}
is independent of $x\in \Xcal$.
\end{definition}

Let $\{ \Phi_{g} \}_{g\in G}$  be a group action preserving the Dirichlet space class, and set 
\[
\mathcal{Z} := \left\{ g\in G \, : \, J_{g} =0  \right\}.
\]
Note that $\mathcal{Z}$ is not the whole group $G$ and it is not a subroup since $e\notin \mathcal{Z}$, and $J_{g^{-1}} = \infty$ for $g\in \mathcal{Z}$. If $g\in G$ is such that  $J_{g} =0$ then, for any Borel set $A$ in $\Xcal$ of positive measure
\[
( \Phi_{g} )_{\ast} \mu (A) = \mu \left( \Phi^{-1}_{g} (A) \right) = \mu \left( \Phi_{g^{-1}} (A) \right) =  J_{g} \mu(A) = 0, 
\]
that is, $( \Phi_{g} )_{\ast} \mu \equiv 0$. Similarly, $( \Phi_{g^{-1}} )_{\ast} \mu (A)= \infty$ for $g\in \mathcal{Z}$ and for any  Borel set $A$ in $\Xcal$ of positive measure. To avoid this degenerate situation, in the following statements we will assume $g\in \mathcal{Z}^{c}$.

\begin{definition}\label{dfn.G.stru}
We say that a Dirichlet metric measure space $(\Xcal, d, \mu, \mathcal{E} )$ admits a $G$-\emph{dilation structure} if the action of  $G$  on $(\Xcal, d, \mu, \mathcal{E} )$ preserves the Dirichlet space class  and if there exists $ \kappa =\kappa \left( g \right)$ such that
\begin{align}\label{eqn.dform.dilation}
\mathcal{E} (f \circ \Phi_{g}, h \circ \Phi_{g}) = J_{g}^{\kappa}  \mathcal{E} (f, h), 
\end{align}
for any $f, h\in \mathcal{D}_{\mathcal{E}}$.
\end{definition}

\begin{notation}\label{notation.RN.der}
For every $g\in \mathcal{Z}^{c}$ let us set $\ell_{g} := J_{g}^{\kappa (g) -1}$, where $J_{g}$ is given by  \eqref{eqn.RN.der.indep}.
\end{notation}

\begin{example}[Euclidean dilations]
Let $\Xcal = \R^{n}$, $G=\R_{>0}$, and $\Phi_{r} (x) = \delta_{r} (x) = rx$. Then 
\begin{align*}
& \int_{\R^{n} } f (\delta_{r} (x) ) dx  = r^{-n} \int_{\R^{n}} f(x) dx, 
\\
& \int_{\R^{n} } f (\delta_{r} (x) ) dx  = \int_{\R^{n}} f(x) \frac{d \delta_{r^{-1}} (x)}{dx} dx = \int_{\R^{n}} f(x) \frac{d ( \delta_{r} )_{\ast} (x)}{dx} dx = J_{r} \int_{\R^{n}} f(x) dx, 
\end{align*}
that is, $J_{r} = r^{-n}$, and, for any $f\in \mathcal{D}_{\mathcal{E}}$
\begin{align*}
\mathcal{E} (f\circ \delta_{r} ) = r^{-n+2} \mathcal{E}(f) = J_{r}^{ \frac{n-2}{n}}  \mathcal{E}(f).
\end{align*}
Thus, when $(\Xcal, G) = (\R^{n}, \R_{>0})$ we have 
\begin{align*}
\kappa (r)= 1-\frac{2}{n} <1,
\end{align*}
for each $r\in \mathbb{R}_{>0}$.
\end{example}

\begin{example}[Carnot group dilations]
Let $\Xcal = \mathbb{G}$ be a Carnot group and $G=\R_{>0}$. Proceeding as in the previous example one has that
\[
\kappa (r) = 1-\frac{2}{Q},
\]
for every $r\in \R_{>0}$, where $Q$ denotes the homogeneous dimension of $\mathbb{G}$.
\end{example}

If \eqref{eqn.upper.and.lower.estimate.Phi} holds then the function $\kappa (g)$ given by \eqref{eqn.dform.dilation} is asymptotically less than one. More precisely, in Lemma \ref{lemma.d.alpha.beta} we will show that, for any $g\in G$, $\kappa (g^{n} )$ converges to $1- \frac{\alpha}{\beta}$ as $n$ goes to infinity,  where $\alpha$ and $\beta$ are given by  \eqref{eqn.upper.and.lower.estimate.Phi}.

\begin{proposition}\label{prop.G.stru}
Let $(\Xcal, d, \mu, \mathcal{E})$ be a Dirichlet metric measure space admitting a $G$-dilation structure for some topological group $G$. Then

(1) For $g\in \mathcal{Z}^{c}$ the Radon-Nikodym derivative in \eqref{eqn.RN.der.indep} $J_{g} >0$, and $J_{g_{1}}J_{g_{2}}=J_{g_{1}g_{2}}$ for any $g_{1}, g_{2} \in G$. In particular, $J_{g^{-1}} = \frac{1}{J_{g}}$ for any $g\in \mathcal{Z}^{c}$, and $J_{e}\equiv 1$, and $J_{g^{-1}} = \infty$ for $g\in \mathcal{Z}$.

(2) $P_{t} = \widetilde{\Phi}_{g} \circ P_{\ell_{g} t} \circ \widetilde{\Phi}_{g^{-1}}$ for any $g\in \mathcal{Z}^{c}$ and $t\geqslant 0$, where $\ell_{g} := J_{g}^{\kappa (g) -1}$. 
\end{proposition}

Note that the factor $\ell_{g}$  is the factor  we get for $(\Xcal, G) = (\R^{n}, \R_{>0})$  as
\[
\ell_{g} := J_{g}^{\kappa (g) -1}= (r^{-n} )^{-\frac{2}{n}}=r^{2},
\]
which is consistent with \eqref{eqn.smigroup.factor} when the walk dimension $\beta$ is $2$.

\begin{proof}
(1) Note that for any Borel set $A$ in $\Xcal$ of positive measure
\[
( \Phi_{g} )_{\ast} \mu (A) = \mu \left( \Phi^{-1}_{g} (A) \right) = \mu \left( \Phi_{g^{-1}} (A) \right) =  J_{g} \mu(A),
\]
which proves that $J_{g} >0$. The rest of the proof follows from the fact that $\Phi_{g} (\Phi_{h} (x) ) = \Phi_{gh} (x)$  for all  $g, h\in G$,   and  $\mu$-a.e.  $x\in \Xcal$.

(2) Let us consider the semigroup $\widetilde{P}_{t} := \widetilde{\Phi}_{g} \circ P_{\ell_{g} t} \circ \widetilde{\Phi}_{g^{-1}}$, and let $\widetilde{\mathcal{E}}$ be the corresponding Dirichlet form. Let us prove that $\mathcal{E} = \widetilde{\mathcal{E}}$. First note that $\mathcal{D}_{\mathcal{E}} = \mathcal{D}_{\widetilde{\mathcal{E}}}$ since $\Xcal$ admits a $G$-dilation structure and $\widetilde{\Phi}_{g} ( \mathcal{D}_{\mathcal{E}} )= \mathcal{D}_{\mathcal{E}}$ for any $g\in G$. Then for any $f, h \in   \mathcal{D}_{\mathcal{E}}$
\begin{align*}
& \widetilde{\mathcal{E}} (f, h ) = \lim_{t\rightarrow 0} \frac{1}{t}  \langle f - \widetilde{T}_{t} f, h \rangle_{L^{2} (\Xcal, \mu )}= \lim_{t\rightarrow 0} \frac{1}{t} \int_{\Xcal} \left( f(x) -  ( \widetilde{P}_{t} f) (x) \right)h(x)  d\mu (x)
\\
& = \lim_{t\rightarrow 0} \frac{1}{t} \int_{\Xcal} \left( f(x) - \widetilde{\Phi}_{g} F(x) \right) h(x) d\mu (x),
\end{align*}
where $F:=P_{\ell_{g }t} \circ \widetilde{\Phi}_{g^{-1}} \circ f$, and hence
\begin{align*}
& \widetilde{\mathcal{E}} (f, h) =  \lim_{t\rightarrow 0} \frac{1}{t}  \int_{\Xcal} \left( f(x) - F( \Phi_{g} (x) \right) h(x)  d\mu  (x)
\\
&=  \lim_{t\rightarrow 0} \frac{1}{t}  \int_{\Xcal} \left( f ( \Phi^{-1}_{g} (x) ) - F(x)  \right) h ( \Phi^{-1}_{g} (x)) d \mu ( \Phi^{-1}_{g} (x) )
\\
& = J_{g}  \lim_{t\rightarrow 0} \frac{1}{t}  \int_{\Xcal} \left( (f\circ \Phi_{g^{-1}} ) (x) - \left( P_{\ell_{g} t} \circ \widetilde{\Phi}_{g^{-1}} \circ f \right)  (x)  \right)   (h\circ \Phi_{g^{-1}} ) (x)  d\mu (x)
\\
& = J_{g}^{\kappa(g)}  \lim_{t\rightarrow 0} \frac{1}{\ell_{g} t}  \int_{\Xcal} \left( (f\circ \Phi_{g^{-1}} ) (x) - \left( P_{\ell_{g} t} \circ \widetilde{\Phi}_{g^{-1}} \circ f \right)  (x)  \right)   (h\circ \Phi_{g^{-1}} ) (x)  d\mu (x)
\\
& = J_{g}^{\kappa(g)}  \lim_{t\rightarrow 0} \frac{1}{t}  \int_{\Xcal} \left( (f\circ \Phi_{g^{-1}} ) (x) - \left( P_{t} \circ \widetilde{\Phi}_{g^{-1}} \circ f \right)  (x)  \right)   (h\circ \Phi_{g^{-1}} ) (x) d\mu (x)
\\
& =  J_{g}^{\kappa (g)}  \mathcal{E} \left( f\circ \Phi_{g^{-1}}, h\circ \Phi_{g^{-1}}  \right) =  J_{g}^{\kappa (g) }    J_{g^{-1}}^{\kappa (g)}  \mathcal{E}(f, h) = \mathcal{E}(f, h),
\end{align*}
where in the last line we used the definition of $G$-dilation structure and part (1).
\end{proof}

\begin{corollary}\label{cor.G.struc}
Let $(\Xcal, d, \mu, \mathcal{E} )$ be a Dirichlet metric measure space admitting a $G$-dilation structure for some topological group $G$.

(1) The heat kernel $p_{t}$  satisfies
\begin{equation}\label{eqn.scalin.HK.groups}
J_{g} p_{t} (x,y) = p_{\ell_{g}  t} \left( x_{g}, y_{g} \right),
\end{equation}
for every $g\in \mathcal{Z}^{c}$ and $\mu$-a.e. $x, y \in \Xcal$. 

(2) The process $X_{t}$ satisfies the following scaling property
\begin{equation}\label{eqn.scalin.proc.group}
X_{t}^{x_{g}} \stackrel{(d)}{=}  \Phi_{g} \left( X_{\frac{t}{\ell_{g} }} \right).
\end{equation}

(3) For any $g\in \mathcal{Z}^{c}$, and an open set $\Ucal$ we have that
\begin{equation}\label{eqn.exit.time.group}
\Prob^{x_{g}} \left( \tau_{\Ucal_{g} } > t \right) = \Prob^{x} \left( \tau_{\Ucal} > \frac{t}{ \ell_{g} } \right),
\end{equation}
where $\Ucal_{g} := \Phi_{g} (\Ucal)$.
\end{corollary}

\begin{proof}
(1) Let $p_{t}$ be the heat kernel for $P_{t}$. For any $f \in L^{2} (\Xcal, \mu)$ and any $g\in \mathcal{Z}^{c}$ we have 
\begin{align*}
& \int_{\Xcal}  f(y) p_{t} (x, y)  d\mu (y) = (P_{t} f)(x) = \left( \widetilde{\Phi}_{g} \circ P_{\ell_{g} t} \circ \widetilde{\Phi}_{g^{-1}} \right) f(x)
\\
& = \int_{\Xcal} f (\Phi_{g^{-1}} (y) ) p_{\ell_{g}  t} ( \Phi_{g} (x), y )  d\mu (y) = \int_{\Xcal} f(z)  p_{\ell_{g}  t} ( \Phi_{g} (x), \Phi_{g} (z) )  d\mu ( \Phi_{g} (z) )
\\
& = J_{g}^{-1} \int_{\Xcal} f(z)  p_{\ell_{g}  t} ( x_{g}, z_{g} )  d\mu ( z ).
\end{align*}

(2) For any Borel set $A$ in $\Xcal$ we have that
\begin{align*}
& \Prob^{x_{g}} \left( X_{t} \in A \right) = \int_{A} p_{t} (x_{g}, y)  d\mu  (y) 
= \int_{\Phi_{g}^{-1} (A)} p_{t} (x_{g}, z_{g} ) d\mu  ( z_{g})
\\
& = \int_{\Phi_{g^{-1}} (A) } J_{g}^{-1} p_{t} ( x_{g}, z_{g} )  d\mu  (z) 
 = \int_{\Phi_{g^{-1}} (A) } p_{\frac{t}{\ell_{g} }} (x,z)  d\mu  (z)
\\
& = \Prob^{x} \left( X_{\frac{t}{\ell_{g} }} \in \Phi_{g^{-1}} (A) \right) = \Prob^{x} \left( \Phi_{g} \left( X_{\frac{t}{\ell_{g} }} \right) \in A \right).
\end{align*}

(3) By part (2)  we have that
\begin{align*}
& \Prob^{x_{g}} \left( \tau_{\Ucal_{g}} >t \right) = \Prob \left( X^{x_{g}}_{s} \in \Ucal_{g}    \text{ for all }  0 \leqslant s \leqslant t \right) 
\\
& =  \Prob \left( \Phi_{g}^{-1} \left(  X^{x_{g}}_{s} \right) \in \Ucal   \text{ for all }  0 \leqslant s \leqslant t \right)
= \Prob \left( X^{x}_{\frac{s}{\ell_{g} }}  \in \Ucal   \text{ for all }  0 \leqslant s \leqslant t \right) 
\\
& =\Prob \left( X^{x}_{s}  \in \Ucal   \text{ for all }  0 \leqslant s \leqslant \frac{t}{\ell_{g} }\right) = \Prob^{x} \left( \tau_{\Ucal} > \frac{t}{\ell_{g} } \right).
\end{align*}
\end{proof}

\begin{lemma}\label{lemma.d.alpha.beta}
Assume that the heat kernel  satisfies  bound \eqref{eqn.upper.and.lower.estimate.Phi}, that is, 
\[
c_{1} \,  t^{-\frac{\alpha}{\beta}} \,  \Phi \left( c_{2} \, \frac{d(x,y)}{t^{\frac{1}{\beta}}} \right) \leqslant p_{t} (x,y) \leqslant c_{3} \,  t^{-\frac{\alpha}{\beta}} \,  \Phi \left( c_{4} \, \frac{d(x,y)}{t^{\frac{1}{\beta}}} \right),
\]
for some  positive constants $c_{1}, c_{2}, c_{3}, c_{4}$, $\alpha$, and $\beta$, where $\Phi$ is a positive decreasing function on $[0,\infty)$. Then 
\begin{equation}\label{eqn.eqn}
\left| \kappa (g) - \left( 1 -\frac{\beta}{\alpha} \right) \right| \leqslant  \frac{\beta}{\alpha}  \ln \left( \frac{c_{3}}{c_{1}} \right) \, \frac{1 }{\vert \ln J_{g} \vert},
\end{equation}
for every $g\in \mathcal{Z}^{c}$.  Moreover, for every $g\in \mathcal{Z}^{c}$ we have that 
\begin{align*}
\lim_{n\rightarrow \infty} \kappa (g^{n} ) = 1 -\frac{\beta}{\alpha}.
\end{align*}
\end{lemma}

\begin{proof}
By \eqref{eqn.scalin.HK.groups} we have that for every $t>0$, $g\in \mathcal{Z}^{c}$,  and $\mu$-a.e. $x, y\in \Xcal$ 

\[
J_{g} = \frac{p_{\ell_{g} t} \left( x_{g}, y_{g} \right)  }{ p_{t} (x,y)},
\]
and hence by \eqref{eqn.upper.and.lower.estimate.Phi} with $x=y$ it follows that 
\begin{align*}
& \frac{c_{1}}{c_{3}} \leqslant \ell_{g}^{\frac{\alpha}{\beta}} J_{g} \leqslant \frac{c_{3}}{c_{1}},
\end{align*}
that is, 
\begin{align}\label{eqn.almost}
& \left| \left( 1- \frac{\alpha}{\beta} ( 1-\kappa (g) ) \right) \ln J_{g} \right| \leqslant \ln \frac{c_{3}}{c_{1}},
\end{align}
which proves \eqref{eqn.eqn}.

Let us fixed $g\in \mathcal{Z}^{c}$ and apply  \eqref{eqn.almost}  to $g^{n}$. Then 
\begin{align*}
& \left| \left( 1- \frac{\alpha}{\beta} ( 1-\kappa (g^{n}) ) \right) \ln J_{g} \right| \leqslant \frac{1}{n} \ln \frac{c_{3}}{c_{1}},
\end{align*}
for every $n\in \mathbb{N}$ since $J_{g^{n}} = J_{g}^{n}$. The result then follows by letting $n$ go to infinity. 
\end{proof}

When the group acting is $\R_{>0}$ or $\mathbb{Z}_{>0}$, then the small deviation principle is obtained by letting $\varepsilon\rightarrow 0$ or $r_{0}^{k}$ with $r_{0}<1$ and $k\rightarrow \infty$. We now explain how this argument can be applied to a more general $G$-dilation. Let $\Ucal^{n}_{g} := \Phi_{g^{n}} (\Ucal)$ for any measurable set $\Ucal$  and any $n \in\mathbb{N}$, and let $g\in \mathcal{Z}^{c}$  be such that $J_{g} >1$. Note that for such a $g$ one has that
\begin{align*}
\mu \left( \Phi_{g} (A) \right) = \mu \left( \Phi^{-1}_{g^{-1}} (A) \right) = \left( \Phi_{g^{-1}} \right)_{\ast} \mu (A) = J_{g}^{-1} \mu(A),
\end{align*}
for any measurable set $A$.  Thus, if $J_{g}>1$, then the map $\Phi_{g}: \Xcal \rightarrow \Xcal$ decreases the size of sets,  i.e. $\mu \left( \Phi_{g} (A) \right) < \mu(A)$,  and hence it is reasonable to study small deviations for $\Phi_{g^{n}}$ as $n \rightarrow \infty$ when $g$ is such that $J_{g}>1$.

\begin{theorem}\label{thm.small.dev.3}
Let $\{P_{t}\}_{t\geqslant 0}$ be a strongly continuous contraction semigroup on $L^{2} (\Xcal, \mu)$, and assume that  there exists a $t_{0}$ such that $p_{t_{0}}(x,y)$ is continuous for all $x, y \in \Xcal$, and that $(\Xcal, \mu)$  admits a  $G$-dilation structure. Let $\Ucal$ be an open set in  $\Xcal$ such that $0<\mu (\Ucal) < \infty$,  and $x\in \Ucal$, and assume that $P_{t}^{\Ucal}$ is irreducible. Suppose there exists a $\gamma >0$ such that 
\[
M_{\Xcal} (t)\leqslant c  t^{-\gamma}
\]
for all $t>0$.  Then, for any $g\in G$ for which $J_{g}>1$, and $x\in \Ucal$ we have that

\begin{equation}\label{eqn.small.dev.bis}
\lim_{n \rightarrow \infty} \exp \left(  \frac{\lambda_{1}^{\Ucal} t}{\ell_{g}^{n} } \right) \Prob^{x_{g^{n}}} \left( \tau_{\Ucal_{g}^{n}} >t \right) = c_{1} \varphi_{1} (x),
\end{equation}
where $\lambda_{1}^{\Ucal}>0$ is the first eigenvalue  of $-A_{\Ucal}$, and  $\varphi_{1}$ is the corresponding positive eigenfunction, and $c_{1} = c_{1} (\Ucal)$ is given by \eqref{eqn.c.n}, and   $\Ucal_{\varepsilon}:= \delta_{\varepsilon} (\Ucal)$, and $\Ucal_{g}^{n}:= \Phi_{g}^{n} (\Ucal)$.
\end{theorem}

\begin{proof}
By the results from Section \ref{s.Preliminaries} we have that the generator $A_{\Ucal}$ of $P^{\Ucal}_{t}$ has a discrete spectrum. In particular, the first eigenvalue $\lambda_{1}^{\Ucal}$ is simple, positive, and  its first eigenfunction $\varphi_{1}$ is continuous and it can be chosen to be strictly positive. By Proposition \ref{prop.eigenfunctions.expansion}, part (3) of Corollary~\ref{cor.G.struc},  and the fact that $J_{g^{n}} = (J_{g})^{n}$ it follows that 
\begin{align*}
& \Prob^{x^{n}_{g}} \left( \tau_{\Ucal^{n}_{g} } > t \right) = \Prob^{x} \left( \tau_{\Ucal} > \frac{t}{ \ell_{g}^{n} } \right) = \sum_{k=1}^{\infty} \exp \left( - \frac{\lambda_{k} t}{\ell_{g}^{n} } \right) c_{k} \varphi_{k}(x),
\end{align*}
that is,
\begin{align*}
& \exp \left( \frac{\lambda_{1}^{\Ucal} t}{\ell_{g}^{n} } \right)  \Prob^{x^{n}_{g}} \left( \tau_{\Ucal^{n}_{g} } > t \right) = \varphi_{1} (x) c_{1}(x) + \sum_{k=2}^{\infty} \exp \left( - \frac{( \lambda_{k}- \lambda_{1}^{\Ucal}) t}{\ell_{g}^{n} } \right) c_{k} \varphi_{k}(x),
\end{align*}
and the result follows since $0< \lambda_{1}^{\Ucal} < \lambda_{2} \leqslant \lambda_{3} \leqslant \cdots$, and $c_{1} \varphi_{1} (x)$ is well defined since $\varphi_{1}$ is continuous at $x$.
\end{proof}

\begin{remark}
Let $x$ be a point in $\Xcal$ such that $\Phi_{g^{n_{k}} }(x) =x$ along a subsequence $n_{k}$. Then by Theorem \ref{thm.small.dev.3} we have that 
\[
\lim_{k \rightarrow \infty} \exp \left(  \frac{\lambda_{1}^{\Ucal} t}{\ell_{g}^{n_{k}}  } \right) \Prob^{x} \left( \tau_{\Ucal_{g}^{n_{k}}} >t \right) = c_{1} \varphi_{1} (x),
\]
which recovers the case of a standard Brownian motion on $\R^{n}$, the hypoelliptic Brownian motion starting at the identity on a Carnot group in \cite{CarfagniniGordina2023}, and the Brownian motion on the Sierpi\'{n}ski gasket in \cite{BarlowPerkins1988}. Another example is a hypoelliptic Brownian motion on Carnot groups subordinated by an $\alpha$-stable subordinator. 
\end{remark}

\section{Lie group contractions, approximate dilations and sub-Riemannian structures}\label{s.Hei.SU2}

We describe the setting in which we can connect spectra of an approximating sequence of Lie groups equipped with a sub-Riemannian structure, and as a result prove a small deviation principle on the limiting Lie group. Our main example is a contraction of $\operatorname{SU}\left( 2 \right)$ to the Heisenberg group. The notion of such a contraction has been first introduced in physics literature starting with \cite{InonuWigner1953} and in \cite{Segal1951c}. We will concentrate on the contraction studied in \cite{Ricci1986} and used indirectly for the hypoelliptic heat kernel analysis in \cite{BaudoinBonnefont2009}. For more details on Lie group contractions we refer to \cite{GilmoreBook2008, Saletan1961}. There are different Lie group contractions, including In\"{o}n\"{u}–Wigner contractions and more general Saletan contractions, here we will concentrate on concrete examples. Our goal is to have a contraction to a Carnot group, thus making use of dilations on Carnot groups. This approximation is different from tangent cones a la Gromov, e.g. \cite{Bellaiche1996}, but illustrates how one can use natural dilations on Carnot groups for the spaces on which we do not have a dilation structure. 

We start with the standard definition of Lie algebra and Lie group contractions. Suppose $G$ and $H$ are two (real) Lie groups with corresponding Lie algebras $\mathfrak{g}$ and $\mathfrak{h}$. We assume that $\mathfrak{g}$ and $\mathfrak{h}$ have the same dimension $n$.

\begin{definition}[Lie algebra contraction] A contraction of $\mathfrak{g}$ to $\mathfrak{h}$ is a family $\left\{ U_{\varepsilon} \right\}_{0<\varepsilon \leqslant 1}$ of linear invertible maps from $\mathfrak{h}$ onto $\mathfrak{g}$ such that

\[
 \lim_{\varepsilon \to 0}
U_{\varepsilon}^{-1}\left( [ U_{\varepsilon}\left( V \right), U_{\varepsilon}\left( W \right)]_{\mathfrak{g}} \right) = [ V, W]_{\mathfrak{h}}
\]
for all $V, W \in \mathfrak{h}$.
\end{definition}
These Lie algebras are not isomorphic as Lie algebras in general, but it allows to transfer bracket-generating set to the contracted Lie algebra.

\begin{definition}[Lie group contraction]
Let $G$ and $H$ be two locally compact connected Lie groups of the same dimension. We say that a family $\left\{ \Phi_{\varepsilon} \right\}_{0<\varepsilon \leqslant 1}$ of differentiable maps
$\Phi_{\varepsilon}: H \longrightarrow G$ defines a \emph{contraction of $G$ to $H$}, if it maps the identity $e_{H}$ of $H$ to the identity $e_{G}$ of $G$, and if given any relatively compact open neighbourhood $V$ of $e_{H}$

(1) there is an $\varepsilon_{V}>0$ such that for any $\varepsilon<\varepsilon_{V}$ the map $\left.\Phi_{\varepsilon}\right|_{V}$ is a diffeomorphism;

(2) If $W$ is such that $W^{2}\subset V$ and $\varepsilon<\varepsilon_{V}$, then $\Phi_{\varepsilon}\left( W \right)^{2} \subset  \Phi_{\varepsilon}\left( V \right)$;

(3) for $x, y \in V \subset H$
\[
\lim_{\varepsilon \to 0}
\Phi_{\varepsilon}^{-1}\left( \Phi_{\varepsilon}\left( x \right)\Phi_{\varepsilon}\left( y \right)^{-1} \right) = x y^{-1}
\]
uniformly on $V \times V$ with all derivatives in $x$ and $y$.
\end{definition} 
Note that the first two conditions are needed for the limit in the third condition to be well-defined. The third condition gives an approximation of the group operations on $G$ by using a family of group operations on $H$. Lie group contractions and Lie algebra contractions are in one-to-one correspondence by differentiating and integrating. Denote by $V$ the vector space on which both $\mathfrak{h}$ and $\mathfrak{g}$ are modelled. Then $U_{\varepsilon}: V \longrightarrow V$ is a vector space isomorphism. Suppose $h: V \longrightarrow H$ and $g: V \longrightarrow G$ are coordinate maps which we assume to be local diffeomorphisms. Then we can define
\begin{align}\label{e.LieGroupContraction}
& \Phi_{\varepsilon}: H \longrightarrow G,
\notag
\\
& \Phi_{\varepsilon}:=g \circ U_{\varepsilon} \circ h^{-1}. 
\end{align}
For example, we can take $h$ to be  $\exp_{H}$ and $g$ to be $\exp_{G}$, then a Lie group contraction is
\begin{equation}
\Phi_{\varepsilon}\left( \exp_{H} X \right)=\exp_{G}\left( U_{\varepsilon} \left( X \right) \right), X \in \mathfrak{h} \cong V. 
\end{equation}
Instead of $\exp$ we could use a different coordinate system from a Lie algebra to the corresponding Lie group such as coordinates of the first or second kind for $H$, and a similar system for $G$ to define $\Phi_{\varepsilon}$. As we only use the structure of these groups as local Lie groups, we may disregard topological obstructions.

\begin{proposition}[Left-invariant vector fields on Lie group contractions] There is one-to-one correspondence between left-invariant vector fields on $G$ and $H$.
\end{proposition}

\begin{proof} Recall that for any $X \in \mathfrak{g}$ we can define the corresponding left-invariant vector field for $f: H \longrightarrow \mathbb{R}$ by
\[
\widetilde{X}f \left( x \right):=\left.\frac{d}{dt}\right|_{t=0}f\left( x \exp_{H}\left(t X\right)\right), x \in H, X \in \mathfrak{h}. 
\]

Equivalently we can take $u, v \in V$ such that $h\left( u \right)=x$ and $h\left( tv \right)=\exp_{H}\left(t X\right)$, then

\[
\widetilde{X}f \left( x \right):=\left.\frac{d}{dt}\right|_{t=0}f\left( h\left( u \right)h\left( tv \right)\right), x \in H, X \in \mathfrak{h}. 
\]
Then for any function $k: G \longrightarrow \mathbb{R}$ we can take $f:=k \circ \Phi_{\varepsilon}: H  \longrightarrow \mathbb{R}$ to induce the corresponding vector field by using
\[
 \Phi_{\varepsilon}^{-1}\left( \Phi_{\varepsilon}\left( h \left( u \right)\right)\Phi_{\varepsilon}\left( h \left( tv \right)\right) \right) \xrightarrow[\varepsilon \to 0]{}  h \left( u \right)h \left( tv \right)
\]
and
\[
\left.\frac{d}{dt}\right|_{t=0} \Phi_{\varepsilon}^{-1}\left( \Phi_{\varepsilon}\left( h \left( v \right)\right)\Phi_{\varepsilon}\left( h \left( tu \right)\right) \right) \xrightarrow[\varepsilon \to 0]{} \left.\frac{d}{dt}\right|_{t=0} h \left( v \right)h \left( tu \right).
\]
\end{proof}

A similar statement can be formulated for Haar measures using the parametrization by $h$, $g$ and $U_{\varepsilon}$. Instead of proving general statements, we focus on a concrete example. 

We now discuss a contraction of $\operatorname{SU}\left( 2 \right)$ to the three-dimensional Heisenberg group $\mathbb{H}$. This was used  in \cite{Ricci1986} to compare irreducible unitary representations of these groups, and then it was applied to a study of short time behavior of heat kernels in \cite{BaudoinBonnefont2009} and for higher-dimensional settings in \cite{CampbellMelcher2020}.

Consider $(\Xcal, \mu):= ( \operatorname{SU}\left( 2 \right), m)$, where $m$ is a bi-invariant Haar measure on the compact group of  $2 \times 2$ complex matrices which are unitary and have determinant $1$. The group identity of $\operatorname{SU}\left( 2 \right)$ is the identity matrix $e=I$. The corresponding Lie algebra $\mathfrak{su}\left( 2 \right)$, identified with the tangent space $T_{I} \operatorname{SU}\left( 2 \right)$, is the space of $2 \times 2$  complex matrices which are skew-Hermitian and have trace $0$.  The group has a natural sub-Riemannian structure described below. Theorem~\ref{thm-lim-eigen} shows that eigenvalues for the sub-Laplacian on $\operatorname{SU}\left( 2 \right)$ converge to eigenvalues for the sub-Laplacian in the unit ball in $\mathbb{H}$, and then we prove a small deviation principle on $\operatorname{SU}\left( 2 \right)$.   

We start by defining Lie group and Lie algebra contractions from $\mathfrak{su}\left( 2 \right)$ to the Lie algebra $\mathfrak{h}$ of $\mathbb{H}$. For this we use a Milnor basis $\mathfrak{su}\left( 2 \right)$, as described in \cite{EldredgeGordinaSaloff-Coste2018}, namely, $X$, $Y$ and $Z$ satisfying
\begin{align*}
& [ X, Y]=Z, [ Y, Z]=X, [ Z, X]=Y, 
\\
& X^{2}= Y^{2}=Z^{2}=-\frac{I}{4},
\\
& XY=\frac{Z}{2}, YZ=\frac{X}{2}, ZX=\frac{Y}{2}.
\end{align*}
The basis consisting of (properly scaled) Pauli matrices is an example of such a basis. 

The corresponding Lie group contraction can be written using cylindrical coordinates as in \cite{CowlingSikora2001}. Namely, we can write an element $g \in \operatorname{SU}\left( 2 \right)$ as
\begin{align*}
& g\left( r, \theta, z \right)=e^{r \cos \theta X +r \sin \theta Y}e^{z Z}=\left( \cos \frac{r}{2} I +2\sin \frac{r}{2} \left( \cos \theta X + \sin \theta Y \right) \right)\left( \cos \frac{z}{2} I +2\sin \frac{z}{2} Z\right)
\\
& =\cos \frac{r}{2} \cos \frac{z}{2} I 
+2\sin \frac{r}{2} \cos \left( \theta - \frac{z}{2}\right) X 
+ 2\sin \frac{r}{2} \sin \left( \theta -\frac{z}{2}\right) Y 
+ 2\cos \frac{r}{2} \sin \frac{z}{2}  Z,
\\
& 0 \leqslant r \leqslant \pi, \theta \in [0, 2\pi], z \in [-\pi, \pi].
\end{align*}
We can also view these coordinates as a map from $\mathfrak{su}\left( 2 \right)$ to $\operatorname{SU}\left( 2 \right)$
\[
x X+ y Y+ zX \longmapsto g\left( r, \theta, z \right), x=r \cos \theta, y=r \sin \theta. 
\]

We will also use cylindrical coordinates for $\mathbb{H}$ as follows.
\begin{align*}
& \mathbb{H} \cong \mathbb{R}^{3}, \mathfrak{h}  \cong \mathbb{R}^{3},
\\
& \left( x_{1}, y_{1}, z_{1} \right)\star \left( x_{2}, y_{2}, z_{2} \right)=\left( x_{1}+x_{2}, y_{1}+y_{2}, z_{1}+z_{2}+\frac{1}{2}\left(  x_{1}y_{2}-x_{2} y_{1}\right) \right),
\\
& \left[ \left( a_{1}, b_{1}, c_{1} \right),  \left( a_{2}, b_{2}, c_{2} \right)\right]_{\mathfrak{h}}=\left( 0, 0, \frac{1}{2}\left(  a_{1}b_{2}-a_{2} b_{1}\right) \right),
\\
& \left( x, y, z \right)=h\left( r, \theta, z \right)=\left( r \cos \theta, r \sin \theta, z \right) \in  \mathbb{H}, r>0, 0 \leqslant \theta < 2 \pi, z\in \R.
\end{align*}
Recall that $\mathbb{H}$ is a homogeneous Carnot group and therefore we can define dilations as a one-parameter group of Lie algebra isomorphisms and the corresponding family of automorphisms on the group. 
\begin{align*}
& \delta_{\lambda}: \mathfrak{h} \longrightarrow \mathfrak{h}, \lambda >0, 
\\
& \delta_{\lambda}\left( x, y, z \right):=\left( \lambda x, \lambda y, \lambda^{2} z \right),
\\
&  D_{\lambda}: \mathbb{H} \longrightarrow \mathbb{H}, \lambda >0,
\\
& D_{\lambda}\left( x, y, z \right):=\left( \lambda x, \lambda y, \lambda^{2} z \right),
\\
&  D_{\lambda}h\left( r, \theta, z \right)=h\left( \lambda r, \theta, \lambda^{2} z \right).
\end{align*}
We abuse notation and use $\delta$ for both dilations. The next statement can be interpreted as using Lie group contraction and coordinate maps to pushforward dilations on $\mathbb{H}$ to approximate dilations on $\operatorname{SU}\left( 2 \right)$.

\begin{proposition}[Contraction between $\operatorname{SU}\left( 2 \right)$ and $\mathbb{H}$]
For $0 < \varepsilon \leqslant 1$
\begin{align*}
\Phi_{\varepsilon}: &  \mathbb{H} \longrightarrow \operatorname{SU}\left( 2 \right),
\\
& h\left( r, \theta, z \right) \longmapsto g\left( \sqrt{\varepsilon} r, \theta, \varepsilon z \right).
\end{align*}
The corresponding Lie algebra contraction is given by
\begin{align*}
U_{\varepsilon}: & \hskip0.1in  \mathfrak{h} \longrightarrow \mathfrak{su}\left( 2 \right),
\\
& \left( a, b, c \right) \longmapsto \sqrt{\varepsilon}aX+\sqrt{\varepsilon}bY+\varepsilon cZ.
\end{align*}
\end{proposition}

\begin{proof}
We first show that $U_{\varepsilon}$ is a Lie algebra contraction 
\begin{align*}
& \left[ U_{\varepsilon}\left( \left( a_{1}, b_{1}, c_{1} \right) \right), U_{\varepsilon}\left( \left( a_{2}, b_{2}, c_{2} \right) \right)\right]_{ \mathfrak{su}\left( 2 \right)}
\\
& =\left[ \sqrt{\varepsilon}a_{1} X+ \sqrt{\varepsilon} b_{1} Y+  \varepsilon c_{1} Z, \sqrt{\varepsilon}a_{2} X+ \sqrt{\varepsilon} b_{2} Y+  \varepsilon c_{2} Z \right]_{ \mathfrak{su}\left( 2 \right)}
\\
& =\varepsilon^{3/2} \left( b_{1}c_{2}-c_{1}b_{2}\right) X + \varepsilon^{3/2} \left( c_{1}a_{2}-a_{1}c_{2} \right) Y +\varepsilon \left(a_{1}b_{2}-b_{1}a_{2}\right) Z,
\end{align*}
and therefore
\begin{align*}
& U_{\varepsilon}^{-1} \left( \left[ U_{\varepsilon}\left( \left( a_{1}, b_{1}, c_{1} \right) \right), U_{\varepsilon}\left( \left( a_{2}, b_{2}, c_{2} \right) \right)\right]_{ \mathfrak{su}\left( 2 \right)} \right)
\\
& =
\left( \varepsilon^{1/2} \left( b_{1}c_{2}-c_{1}b_{2}\right),  \varepsilon^{1/2} \left( c_{1}a_{2}-a_{1}c_{2} \right),  \left(a_{1}b_{2}-b_{1}a_{2}\right)\right) 
\\
& \xrightarrow[\varepsilon \to 0]{} \left( 0,  0,  \left(a_{1}b_{2}-b_{1}a_{2}\right)\right)
=[\left( a_{1}, b_{1}, c_{1} \right), \left( a_{2}, b_{2}, c_{2} \right)]_{\mathfrak{h}}. 
\end{align*}
We can use cylindrical coordinates instead of $\exp$ in \eqref{e.LieGroupContraction} to see that for any $\left( a, b, c \right) \in \mathfrak{h}$
\begin{align*}
\Phi_{\varepsilon}\left( h\left( a, b, c \right)  \right)=g\left( U_{\varepsilon} \left( \left( a, b, c \right) \right) \right)
=g\left( \sqrt{\varepsilon}aX+\sqrt{\varepsilon}bY+\varepsilon cZ \right)=g\left( \sqrt{\varepsilon} r\left( a, b \right), \theta\left( a, b \right), \varepsilon c \right). 
\end{align*}

\end{proof}

\begin{proposition}
Lie group contraction between $\operatorname{SU}\left( 2 \right)$ and  $\mathbb{H}$ induces convergence of the coefficients of left-invariant vector fields written in cylindrical coordinates.
\end{proposition}

\begin{proof}
For this, we consider two one-parameter subgroups in $\operatorname{SU}\left( 2 \right)$ generated by $X, Y \in \mathfrak{su}\left( 2 \right)$, as we only need these for the sub-Riemannian setting. Recall that
\begin{align*}
& \exp_{\operatorname{SU}\left( 2 \right)}\left( t X \right)=g\left( t, 0,  0 \right)=\cos \frac{t}{2} I+2 \sin \frac{t}{2} X,
\\
& \exp_{\operatorname{SU}\left( 2 \right)}\left( t Y \right)=g\left( t, \frac{\pi}{2},  0 \right)=\cos \frac{t}{2} I+2 \sin \frac{t}{2} Y.
\end{align*}
Denote
\[
g^{X}\left( r \left( \varepsilon, t \right), \theta\left( \varepsilon, t \right), z\left( \varepsilon, t \right) \right) :=\Phi_{\varepsilon} \left( h\left( r, \theta, z \right) \right)\Phi_{\varepsilon} \left( h\left( t, 0, 0 \right) \right),
\]
where $r \left( \varepsilon, 0 \right)= \sqrt{\varepsilon} r$, $\theta \left( \varepsilon, 0 \right)=\theta$, $z \left( \varepsilon, 0 \right)= \varepsilon z$. Then for a differentiable function $f$ of three variables we have
\[
\widetilde{X_{\varepsilon}}f\left( r, \theta, z \right):=\left.\frac{d}{dt}\right|_{t=0}f\left( r \left( \varepsilon, t \right), \theta\left( \varepsilon, t \right), z\left( \varepsilon, t \right) \right).
\]
\begin{align*}
& g^{X}\left( r \left( \varepsilon, t \right), \theta\left( \varepsilon, t \right), z\left( \varepsilon, t \right) \right)
\\
&
=
\cos \frac{r\left( \varepsilon, t \right)}{2} \cos \frac{z\left( \varepsilon, t \right)}{2} I 
+2\sin \frac{r\left( \varepsilon, t \right)}{2} \cos \left( \theta\left( \varepsilon, t \right) - \frac{z\left( \varepsilon, t \right)}{2}\right) X 
\\
&
+ 2\sin \frac{r\left( \varepsilon, t \right)}{2} \sin \left( \theta\left( \varepsilon, t \right) -\frac{z\left( \varepsilon, t \right)}{2}\right) Y 
+ 2\cos \frac{r\left( \varepsilon, t \right)}{2} \sin \frac{z\left( \varepsilon, t \right)}{2}  Z
\\
& =
\left(\cos \frac{\sqrt{\varepsilon} r}{2} \cos \frac{\varepsilon z}{2}\cos \frac{\sqrt{\varepsilon} t}{2}  
-\sin \frac{\sqrt{\varepsilon} r}{2} \cos \left( \theta - \frac{\varepsilon z}{2}\right) \sin \frac{\sqrt{\varepsilon} t}{2}\right) I
\\
&
+2\left( \sin \frac{\sqrt{\varepsilon} r}{2} \cos \left( \theta - \frac{\varepsilon z}{2}\right)\cos \frac{\sqrt{\varepsilon} t}{2} 
+\cos \frac{\sqrt{\varepsilon} r}{2} \cos \frac{\varepsilon z}{2}\sin \frac{\sqrt{\varepsilon} t}{2} \right)  X 
\\
&
+ 2\left( \sin \frac{\sqrt{\varepsilon} r}{2} \sin \left( \theta -\frac{\varepsilon z}{2}\right)\cos \frac{\sqrt{\varepsilon} t}{2} 
+ \cos \frac{\sqrt{\varepsilon} r}{2} \sin \frac{\varepsilon z}{2}  \sin \frac{\sqrt{\varepsilon} t}{2} \right) Y
\\
&
+ 2\left( \cos \frac{\sqrt{\varepsilon} r}{2} \sin \frac{\varepsilon z}{2} \cos \frac{\sqrt{\varepsilon} t}{2} 
-\sin \frac{\sqrt{\varepsilon} r}{2} \sin \left( \theta -\frac{\varepsilon z}{2}\right)\sin \frac{\sqrt{\varepsilon} t}{2} \right)  Z
\end{align*}
Taking the derivative in $t$ of the right-hand side at $t=0$ we get
\begin{align*}
& 
 - 
\frac{\sqrt{\varepsilon}}{2}\sin \frac{\sqrt{\varepsilon} r}{2} \cos \left( \theta - \frac{\varepsilon z}{2}\right) I 
+\sqrt{\varepsilon} \cos \frac{\sqrt{\varepsilon} r}{2} \cos \frac{\varepsilon z}{2} X
\\
& 
+ 
 \sqrt{\varepsilon}\cos \frac{\sqrt{\varepsilon} r}{2} \sin \frac{\varepsilon z}{2} Y 
-\sqrt{\varepsilon}\sin \frac{\sqrt{\varepsilon} r}{2} \sin \left( \theta -\frac{\varepsilon z}{2}\right) Z.
\end{align*}
Denote $r^{\prime}:=\partial_{t}r\left( \varepsilon, 0 \right)$, $\theta^{\prime}:=\partial_{t}\theta\left( \varepsilon, 0 \right)$, 
$z^{\prime}:=\partial_{t}z\left( \varepsilon, 0 \right)$, then the derivative in $t$ of the left-hand side at $t=0$ is
\begin{align*}
&  -\left(\frac{r^{\prime}}{2} \sin \frac{\sqrt{\varepsilon}r}{2} \cos \frac{\varepsilon z}{2}
+\frac{z^{\prime}}{2} \cos \frac{\sqrt{\varepsilon}r}{2} \sin \frac{\varepsilon z}{2} \right)
 I 
 \\
 &
+\left( r^{\prime}\cos  \frac{\sqrt{\varepsilon}r}{2} \cos \left( \theta - \frac{\varepsilon z}{2}\right)
-
\left( 2\theta^{\prime} - z^{\prime}\right)
\sin \frac{\sqrt{\varepsilon}r}{2}  \sin \left( \theta - \frac{\varepsilon z}{2}\right)
\right)
 X 
\\
&
+\left( r^{\prime} \cos\frac{\sqrt{\varepsilon}r}{2} \sin \left( \theta -\frac{\varepsilon z}{2}\right)
+ \left( 2\theta^{\prime} - z^{\prime}\right) \sin \frac{\sqrt{\varepsilon}r}{2} \cos \left( \theta -\frac{\varepsilon z}{2}\right)
\right)
 Y
\\
& 
+\left(-r^{\prime}\sin \frac{\sqrt{\varepsilon}r}{2} \sin \frac{\varepsilon z}{2} 
+z^{\prime} \cos \frac{\sqrt{\varepsilon}r}{2} \cos \frac{\varepsilon z}{2}
\right)
 Z.
\end{align*}
\begin{align*}
& 
\sqrt{\varepsilon} \sin \frac{\sqrt{\varepsilon} r}{2} \cos \left( \theta - \frac{\varepsilon z}{2}\right)
=
r^{\prime} \sin \frac{\sqrt{\varepsilon}r}{2} \cos \frac{\varepsilon z}{2}
+z^{\prime} \cos \frac{\sqrt{\varepsilon}r}{2} \sin \frac{\varepsilon z}{2},
\\
& 
\sqrt{\varepsilon} \cos \frac{\sqrt{\varepsilon} r}{2} \cos \frac{\varepsilon z}{2}
=
 r^{\prime}\cos  \frac{\sqrt{\varepsilon}r}{2} \cos \left( \theta - \frac{\varepsilon z}{2}\right)
-
\left( 2\theta^{\prime} - z^{\prime}\right)
\sin \frac{\sqrt{\varepsilon}r}{2}  \sin \left( \theta - \frac{\varepsilon z}{2}\right),
\\
& 
\sqrt{\varepsilon} \cos \frac{\sqrt{\varepsilon} r}{2} \sin \frac{\varepsilon z}{2} 
 = r^{\prime} \cos\frac{\sqrt{\varepsilon}r}{2} \sin \left( \theta -\frac{\varepsilon z}{2}\right)
+ \left( 2\theta^{\prime} - z^{\prime}\right) \sin \frac{\sqrt{\varepsilon}r}{2} \cos \left( \theta -\frac{\varepsilon z}{2}\right), 
 \\
 &
\sqrt{\varepsilon} \sin \frac{\sqrt{\varepsilon} r}{2} \sin \left( \theta -\frac{\varepsilon z}{2}\right)
=r^{\prime}\sin \frac{\sqrt{\varepsilon}r}{2} \sin \frac{\varepsilon z}{2} 
-z^{\prime} \cos \frac{\sqrt{\varepsilon}r}{2} \cos \frac{\varepsilon z}{2}.
\end{align*}
Therefore
\begin{align*}
& r^{\prime}=\sqrt{\varepsilon} \cos\left( \varepsilon z-\theta \right),
\\
& z^{\prime}=\sqrt{\varepsilon} \tan\frac{\sqrt{\varepsilon}r}{2} \sin\left( \varepsilon z-\theta \right),
\\
& \theta^{\prime}=\sqrt{\varepsilon} \frac{\sin\left( \varepsilon z-\theta \right)}{\sin \sqrt{\varepsilon}r},
\end{align*}
thus
\begin{align*}
& \widetilde{X}_{\varepsilon}=
\alpha^{r}\left( r, \theta, z, \varepsilon \right) \partial_{r}
+\alpha^{z}\left( r, \theta, z, \varepsilon \right) \partial_{z}
+\alpha^{\theta}\left( r, \theta, z, \varepsilon \right) \partial_{\theta}
\\
& =
\sqrt{\varepsilon}\cos\left( \varepsilon z-\theta \right) \partial_{r}
+\sqrt{\varepsilon}\tan\frac{\sqrt{\varepsilon}r}{2} \sin\left( \varepsilon z-\theta \right) \partial_{z}
+\sqrt{\varepsilon}\frac{\sin\left( \varepsilon z-\theta \right)}{\sin \sqrt{\varepsilon}r} \partial_{\theta}.
\end{align*}
Note that we can induce the action $\Phi_{\varepsilon}^{-1}$ on $\widetilde{X}_{\varepsilon}$ by
\[
\Phi_{\varepsilon}^{-1} \circ \widetilde{X}_{\varepsilon}=
\frac{\alpha^{r}\left( r, \theta, z, \varepsilon \right)}{\sqrt{\varepsilon}} \partial_{r}
+\frac{\alpha^{z}\left( r, \theta, z, \varepsilon \right)}{\varepsilon} \partial_{z}
+\alpha^{\theta}\left( r, \theta, z, \varepsilon \right) \partial_{\theta}.
\]
Finally, we see that
\begin{align*}
& \lim_{\varepsilon \to 0}\frac{\alpha^{r}\left( r, \theta, z, \varepsilon \right)}{\sqrt{\varepsilon}} =
\lim_{\varepsilon \to 0} \cos\left( \varepsilon z-\theta \right)=\cos  \theta,
\\
&\lim_{\varepsilon \to 0}\alpha^{\theta}\left( r, \theta, z, \varepsilon \right) =\lim_{\varepsilon \to 0} \sqrt{\varepsilon}\frac{\sin\left( \varepsilon z-\theta \right)}{\sin \sqrt{\varepsilon}r}=-\frac{\sin \theta}{r},
\\
&\lim_{\varepsilon \to 0}\frac{\alpha^{z}\left( r, \theta, z, \varepsilon \right)}{\varepsilon} 
=\lim_{\varepsilon \to 0}\frac{\tan\frac{\sqrt{\varepsilon}r}{2} \sin\left( \varepsilon z-\theta \right)}{\sqrt{\varepsilon}}=-\frac{r}{2} \sin \theta. 
\end{align*}
These are exactly the coefficients of the left-invariant vector field $\beta^{r} \partial_{r} +\beta^{\theta} \partial_{\theta}+\beta^{z} \partial_{z}$ corresponding to $X$ in $\mathbb{H}$ written in cylinder coordinates. Indeed, if $\left( 1, 0, 0 \right) \in \mathfrak{h}$, then the corresponding left-invariant field is given by differentiating the cylindrical coordinates below in $t$ at $t=0$. 
\begin{align*}
& h\left( r, \theta, z \right)\star h\left( t, 0, 0 \right)=\left( r \cos \theta +t,  r \sin \theta, z -\frac{t r \sin \theta}{2} \right)
\\
& =h\left( r\left( r, \theta, z, t \right), \theta\left( r, \theta, z, t \right), z\left( r, \theta, z, t \right)\right).
\end{align*}
Thus
\begin{align*}
& \beta^{r} =\frac{\frac{d}{dt}\left|_{t=0}\left(\left( r \cos \theta +t\right)^{2}+ r^{2} \sin^{2} \theta\right)\right.}{2r}=\cos \theta,
\\
& \beta^{\theta}=\frac{d}{dt}\left|_{t=0} \arctan\left( \frac{r \sin \theta}{r \cos \theta +t}\right)\right.=-\frac{\sin \theta}{r},
\\
&  \beta^{z} =-\frac{ r}{2}\sin \theta.
\end{align*}
The part for $Y$ is similar. 
\end{proof}

The left-invariant vector fields on \SUtwo\ corresponding to any fixed Milnor basis will also be denoted by $X, Y$, and $Z$. We consider the following sub-Laplacian on \SUtwo
\begin{align*}
\mathcal{L}^{\SUtwo}: = X^{2} + Y^{2}.
\end{align*}

Let $d_{\SUtwo}$ be the Carnot-Carath\'eodory distance  induced by $\mathcal{L}^{\SUtwo}$. The tangent cone to $\left(\SUtwo,  d_{\SUtwo} \right)$ in the Gromov-Hausdorff sense is the three-dimensional Heisenberg group $\Hei$, as described in \cite{Bellaiche1996}. We use it below to define a local dilation structure on $\SUtwo$. Note that  Lie group contractions can be used to prove that there is one-to-one correspondence between horizontal paths in these two groups. Similarly one can treat  Haar measures by taking the pushforward measure under $\Phi_{\varepsilon}$, show that the limit is translation-invariant to see that it is a Haar measure. Alternatively one can compute the Jacobian in the coordinates we used previously.

We now introduce a local dilation structure on $\SUtwo$. 

\begin{definition}[Local dilations] On a neighborhood of the identity of $\SUtwo$ we define
\begin{equation}\label{eq-def-dilation}
\delta^{\SUtwo}_r:=\Phi\delta_r^\Hei\Phi^{-1}
\end{equation}	
where $\delta_r^\Hei$ is the standard dilation of \Hei, and $\Phi = \Phi_{1}$ is the contraction between $\SUtwo$ and $\Hei$.
\end{definition}
Note that this local (approximate) dilation on $\SUtwo$ is closely related to the notion of approximate dilations on nilpotent groups modelled on a Euclidean space. Most relevant for our applications is \cite{Varopoulos1990a}, and this approach was recently exploited in \cite{ChenZQKumagaiSaloff-CosteWangZheng2022} to study limit groups. Unlike in these settings we use approximate dilations on a non-nilpotent group $\SUtwo$. Recall that if these dilations form  a continuous automorphism group on a group $G$, then by \cite[Proposition~5.4]{Siebert1986a} $G$ is a Carnot group, more precisely, a locally compact group admits a continuous automorphism group if and only if this group is a Carnot group.

Let us now collect some known results about the approximation of $\SUtwo$ by $\Hei$. 

\begin{proposition}\label{prop.properties.approx}
We have that 

(1) For any $f\in C^2  (\Hei)$ 
\begin{equation}\label{eq-Lap-SUtwo-Hei}
\lim\limits_{r\to0}
r^2\Lcal^{\SUtwo}\left(f\circ\delta^{\Hei}_{1/r}\circ\Phi^{-1}\right)
=
 \Lcal^\Hei f,  
\end{equation}
uniformly on any bounded set in $\Hei$.

(2)  Let $m^{\SUtwo}\circ \delta_r^{\SUtwo}\circ\Phi$ denote   the smooth measure on \Hei\ defined using compositions with $ \Phi$ and the dilation $  \delta_r^{\SUtwo} $.  Then 
\begin{equation}\label{eq-Haar-SUtwo-Hei}
\lim\limits_{r\to0}
\frac{1}{r^4}\frac{d\,\left(m^{\SUtwo}\circ \delta_r^{\SUtwo}\circ\Phi\right) }{d\,m^\Hei}= 1, 
\end{equation}
uniformly on any bounded set in $\Hei$.

(3) For each $t>0$ 
\begin{equation}\label{eq-Ptxy-SUtwo-H}
\lim\limits_{r\to0} 
\,
r^4
\,
p^{\SUtwo}_{r^2t}
\left(\Phi\left(\delta_r^\Hei(x)\right),\Phi\left(\delta_r^\Hei(x)\right)\right)
=
p^{\Hei}_{t}(x,y).
\end{equation}
uniformly for $x,y\in \Hei$.

(4)   For any $x,y \in \Hei$
\begin{equation}\label{eq-dist-SUtwo-Hei}
\lim\limits_{r\to0} r^{-1} d^{\SUtwo}\left(\Phi\left(\delta_r^\Hei(x)\right),\Phi\left(\delta_r^\Hei(x)\right)\right)=
d^\Hei(x,y).
\end{equation}
uniformly on any bounded set in $\Hei$.
\end{proposition}

\begin{proof}
We only give a sketch of the proof for completeness. For more details we refer to \cite[Section 3.4]{BaudoinBonnefont2009}. First, to obtain \eqref{eq-Haar-SUtwo-Hei} note that 
\begin{align}
 &\label{eq-Lcal-r}\Lcal^r:= r^{2} \Lcal^{\SUtwo} \left(\cdot\circ\delta^{\Hei}_{1/r}\circ\Phi^{-1}\right)
&\\\notag& = 
   \frac{\partial^{2}}{\partial \rho^{2}} + 2 r \cot (2 r \rho) \frac{\partial}{\partial \rho} +
\left( 2 r^{2}+  r^{2}\cot^{2} (r \rho) + r^{2} \tan^{2} (r \rho) \right) \frac{\partial^{2}}{\partial \theta^{2}}  &\\\notag& + \frac{1}{r^{2}} \tan^{2} (r \rho) \frac{\partial^{2}}{\partial z^{2}} + 2 (1 + \tan^{2} (r \rho) )\frac{\partial^{2}}{\partial z \partial \theta}
&\\\notag&
 \longrightarrow \  \frac{\partial^{2}}{\partial \rho^{2}} + \frac{1}{\rho} \frac{\partial}{\partial \rho} + \frac{1}{\rho^{2}} \frac{\partial^{2}}{\partial \theta^{2}} + \rho^{2} \frac{\partial^{2}}{\partial z^{2}} + 2 \frac{\partial^{2}}{\partial z \partial \theta} = \mathcal{L}^{\Hei}, \qquad \text{ as }  r\rightarrow 0.&
\end{align}

To obtain \ref{eq-Haar-SUtwo-Hei}, note that the normalized Haar measure on $\SUtwo$ in cylindrical coordinates is given by $d m = \frac{1}{2} \sin (2\rho) d\rho d \theta dz$. We can write  $m^{\SUtwo}\circ \delta_r^{\SUtwo}\circ\Phi=m^{\SUtwo} \circ\Phi \circ \delta_r^{\Hei} $ using \eqref{eq-def-dilation}, and hence 
\begin{align*}
& \lim\limits_{r\to0}
\frac{1}{r^4}\frac{d\,\left(m^{\SUtwo}\circ \delta_r^{\SUtwo}\circ\Phi\right) }{d\,m^\Hei} =   \lim\limits_{r\to0} \frac{1}{2r} \frac{\sin ( 2 r \rho )}{\rho} =1.
\end{align*}

Formula \eqref{eq-Ptxy-SUtwo-H} follows from \cite[Proposition 3.13]{BaudoinBonnefont2009}. Indeed, using cylindrical coordinates both in $\SUtwo$ and $\Hei$
\begin{align*}
& \lim_{r \to 0} r^{4} t^{2} p_{r^{2}t}^{\SUtwo} \left( \Phi \left( \delta_{r \sqrt{t}} (\rho,  \theta, z) \right) \right) = p_{1}^{\Hei} ( \rho, \theta,  z),
\end{align*}
and hence 
\begin{align*}
& \lim_{r \to 0} r^{4} p_{r^{2}t}^{\SUtwo} \left( \Phi \left( \delta_{r} (\rho,  \theta, z) \right) \right) = \frac{1}{t^{2}} p_{1}^{\Hei} \left(  \delta_{\frac{1}{\sqrt{t}}} ( \rho, \theta,  z) \right) = p_{t}^{\Hei} (  \rho, \theta,  z),
\end{align*}
where in the last equality we used scaling properties of the heat kernel on $\Hei$. 

Formula \eqref{eq-dist-SUtwo-Hei} follows from the Gromov-Hausdorff convergence of $\SUtwo$ to $\Hei$, see  \cite[Section 3.4]{BaudoinBonnefont2009} and references therein.
\end{proof}

We can now state an prove the main result of this section, that is, the convergence of the re-normalized eigenvalues of small balls in $\operatorname{SU}(2)$ to corresponding eigenvalues in the unit ball of $\Hei$.    Analytically, this corresponds to proving an analogue of  \eqref{eq-Ptxy-SUtwo-H} for the Dirichlet heat kernel.

\begin{theorem}\label{thm-lim-eigen}
Let $0<\lambda_1^\Hei<\lambda_2^\Hei\leqslant\lambda_3^\Hei\leqslant...$ be the Dirichlet eigenvalues in the unit ball $B_1^{\Hei}$ of \Hei, counted with multiplicity. Let $0<\lambda_1^r<\lambda_2^r\leqslant\lambda_3^r\leqslant...$ be the Dirichlet eigenvalues in the   $r$-ball $B_r^{\SUtwo}$ of $\operatorname{SU}(2)$, counted with multiplicity. Then for each $n\geqslant1$ we have 
\begin{equation}\label{eq-eigen-SUtwo-Hei}
\lim\limits_{r\to0}r^{2}\lambda_n^r=\lambda_n^\Hei.
\end{equation}
\end{theorem}

 As a byproduct of this result we have  the following theorem. 
\begin{theorem}\label{cong.1}\label{theorem5.4} Let $g_{t}$ be the $\SUtwo$-valued diffusion having $\mathcal{L}^{\SUtwo}$ for infinitesimal generator. Then the small deviation principle for $g_{t}$ holds with $\lambda_1=\lambda_1^\Hei$.
\end{theorem}

\begin{conjecture} We conjecture that a similar result holds for more general sub-Riemannian diffusions killed at exiting a small metric ball.
\end{conjecture}

The proof of Theorem \ref{thm-lim-eigen} consists of two main parts. The first part is relatively short and simple. It shows that every eigenvalue of $ \lambda_n^\Hei$ can be approximated by rescaled eigenvalues in small balls of \SUtwo. The second part of the proof is longer and, using tools of stochastic analysis, demonstrates that the limit 
\eqref{eq-eigen-SUtwo-Hei} exists for each $n$. 

\begin{proof}[Proof of Theorem \ref{thm-lim-eigen}, Part 1.]
	Note that by  \eqref{eq-dist-SUtwo-Hei}  it follows that for any   $0<\epsilon<1$ there exists  $r_\epsilon>0$ such that 
	the restriction of $\Phi^{-1}$ to $B_{r_\epsilon}^{\SUtwo}$ is a diffeomorphism
	and 
	for any $0<r<r_\epsilon$
	\begin{equation}\label{eq-B-SUtwo-Hei}
	B_{1-\epsilon}^{\Hei}\subset 
	\delta_{1/r}^\Hei 
	\left( \Phi^{-1} \left( B_r^{\SUtwo} \right)\right)
	\subset  B_{1+\epsilon}^{\Hei}.
	\end{equation}

By \eqref{eq-Lap-SUtwo-Hei},  \eqref{eq-Haar-SUtwo-Hei}, and  \eqref{eq-Ptxy-SUtwo-H}, the  proof follows from the classical Rayleigh quotient techniques \cite[page 633]{Strichartz2012a}, also called the mini-max theorem, for positive operators with a discrete spectrum. The general references are \cite{ReedSimonI, KatoBook1995} and the implementation in a wide range of non-smooth situations can be found in \cite{HinzRozanova-PierratTeplyaev2023, KuwaeShioya2003, BuragoIvanovKurylev2019}.  The key observation is that any function of finite energy can be approximated by smooth functions with compact support, which allows to use  \eqref{eq-dist-SUtwo-Hei}, \eqref{eq-Lap-SUtwo-Hei} and \eqref{eq-Haar-SUtwo-Hei}.
\end{proof}

\begin{proof}[Proof of Theorem \ref{thm-lim-eigen}, Part 2.]
	In the second part of the proof we show that   the Dirichlet heat kernels, meaning the heat kernels with zero boundary conditions outside of the balls, in small balls in \SUtwo\ converge uniformly, after re-scaling. 
In this proof $p^{B_1^{\Hei}}_{t}(\cdot,\cdot)$ is the   Dirichlet heat kernel in the unit ball $B_1^{\Hei}$ of   \Hei, and $p^{B_r^{\SUtwo}}_{t}(\cdot,\cdot)$ is the Dirichlet heat kernel in the $r$-ball $B_r^{\SUtwo}$ of   \SUtwo, where the balls are centered at the identity of the groups.  We note that this result is substantially different from  the content of \cite[Proposition 3.13]{BaudoinBonnefont2009}, which does not involve Dirichlet (zero) boundary conditions.

For each $t>0$ 
\begin{equation}\label{eqn.conv.drch.heat.kernel}
\lim\limits_{r\to0} 
\,
r^4
\,
p^{B_r^{\SUtwo}}_{r^2t}
\left(\Phi\left(\delta_r^\Hei(x)\right),\Phi\left(\delta_r^\Hei(x)\right)\right)
=
p^{B_1^{\Hei}}_{t}(x,y),
\end{equation}
uniformly for $x,y\in B_1^\Hei$.
 We use the explicit uniform asymptotic formula in \cite{BaudoinBonnefont2009} for the heat kernels on \SUtwo\ and \Hei\ and the Dynkin-Hunt  formula \eqref{eqn.Dirichlet.HK}. Since the quantities under expectations in  \eqref{eqn.Dirichlet.HK} are uniformly bounded, the result follows from 
\eqref{eq-B-SUtwo-Hei} and   local convergence of stochastic flows. 
Let $X_{s}$ be the standard hypoelliptic Brownian motion on $\Hei$, and 
\begin{align*} 
 \tau_{B_r^\Hei}:= \inf\left\{ s>0,    X_{s}  \notin B_r^{\Hei}  \right\}.  
 \end{align*}
 For any $0<r<\frac17r_{1/7}$ there exists a continuous stochastic process $Y_s^r$ in \Hei\ such that 
\begin{equation}
 Y_{s}^{r}:=:\delta_{1/r}^\Hei \Phi^{-1}\left( g_{r^2s} \right) \qquad \text{\ for \ } s<\tau_{B_{3}^\Hei}
 \end{equation}
  and 
\begin{equation}
\lim\limits_{r\to0}
\,
\mathbbm{1}_{{\{s<\tau_{B_{3}^\Hei}\}}}
\,
\sup\limits_{0\leqslant s \leqslant T}|Y_{s}^{r}-X_{s} |=0
\end{equation}
in probability. Here $r_{1/7}$ is defined in terms of the condition \eqref{eq-B-SUtwo-Hei} with $\epsilon=1/7$, and the symbol $:=:$ stands for the equality of laws.
Let $\Lcal^r$ be the generator of the process $ \delta_{1/r}^\Hei \Phi^{-1}\left( g_{r^2s} \right) $ which is well defined in  the ball $B_4^{\Hei}$ because of our assumptions on $r$. Here $ \Lcal^r$ is the operator defined in \eqref{eq-Lcal-r} the proof of Proposition~\ref{prop.properties.approx}. 
Let $\rho(r):[0,\infty)\to[0,1]$ be a smooth nondecreasing function such that $\rho[0,2]=0$
and $\rho[3,\infty)=1$. 
Then we define the generator $\Lcal^{Y^r_\cdot}$ of the process $Y_{s}^{r}$ 
by 
\begin{equation}
 \Lcal^{Y^r_\cdot} f(x) := \rho(d^\Hei(I,x))\Lcal^\Hei f(x)+\left(1-\rho(d^\Hei(I,x))\right)\Lcal^r f(x)
\end{equation}
or, in other terms, the diffusion process $Y_r^s$ behaves as $ \delta_{1/r}^\Hei \Phi^{-1}\left( g_{r^2s}\right) $ in $B_2^{\Hei}$ and behaves as $X_s$ in $\left(B_3^{\Hei}\right)^c$.
 We can claim that by \cite[Theorem 3.3.1, page 76]{KunitaBook1986}
\begin{equation}
\lim\limits_{r\to0}Y^r_s=X_s
\end{equation}
in the strong sense, \cite[Definition 3.1.2, page 65]{KunitaBook1986}. 
\end{proof}

\section{Generalized heat content}\label{s.heat.content}
In this section, we consider a metric measure space $(\Xcal, \mu)$ endowed with a Dirichlet form $\left( \mathcal{E}, \mathcal{D}_{\mathcal{E}}\right)$ satisfying a Nash inequality \eqref{eqn.Nash.inequal}. We assume that the semigroup $\{P_{t}\}_{t\geqslant 0}$ admits a kernel $p_{t} (x,y)$ for any $t$ and for $\mu$-a.e. $x,y \in \Xcal$. Let  $\Ucal \subset \Xcal$ be an open connected set in $\Xcal$ such that $0< \mu(\Ucal) < \infty$, and $A_{\Ucal}$ be the generator of $\left( \mathcal{E}, \mathcal{D}_{\mathcal{E}} (\Ucal)\right)$. We assume that the semigroup $\{P^{\Ucal}_{t}\}_{t\geqslant 0}$ is strongly continuous and that \eqref{eqn.heat.kernel.small.one} holds. Then by Proposition \ref{thm.spectral.gap} the operator $-A_{\Ucal}$ has a discrete spectrum and the first eigenvalue is strictly positive. Moreover, the conditions of Proposition \ref{prop.eigenfunctions.expansion} are satisfied and hence the Dirichlet heat kernel has the following series expansion
\begin{align*}
p_{t}^{\Ucal} (x,y) = \sum_{n=1}^{\infty} e^{-\lambda_{n} t} \varphi_{n} (x) \varphi_{n}(y), 
\end{align*}
where the convergence is  uniform on $\Ucal \times \Ucal \times [\varepsilon, \infty)$ for any $\varepsilon>0$.

\begin{definition}\label{def.gen.heat.cont}
Let $X_{t}^{x}$ be an $\Xcal$-valued Hunt process with $X_{0}^{x}=x\in \Xcal$ a.s. and whose transition density is given by the heat kernel $p_{t} (x, \cdot)$. The generalized heat content associated to a set $\Ucal$ is given by 
\begin{equation}\label{eqn.gen.heat.cont}
Q_{\Ucal} (t) := \int_{\Ucal} \Prob^{x} \left( \tau_{\Ucal} > t \right)  d\mu  (x),
\end{equation}
where $\tau_{\Ucal}$ denotes the exit time of $X_{t}$ from $\Ucal$.
\end{definition}

  Let $\Xcal = \R^{n}$ and $\Ucal = \Omega$ be an open connected regular set with finite Lebesgue measure, and   $\mathcal{P}$ be an elliptic second order differential operator. The classical heat content is defined in terms of the solution to the following Dirichlet problem
\begin{align}\label{eqn.par.BVP}
& \left( \partial_{t} - \mathcal{P} \right) u(x,t) =0, &  (t,x)\in (0,\infty) \times \Omega, \notag
\\
& u(t,x) =0, & (t,x)\in (0,\infty) \times \partial \Omega,
\\
& u(0,x)=1, & x\in \Omega. \notag
\end{align}
More precisely, if $u$ is a solution to the boundary value problem \eqref{eqn.par.BVP}, then the classical heat content is defined as 
\begin{align*}
Q_{\Omega} (t) := \int_{\Omega} u(t,x)dx.
\end{align*}
Regularity of $\Omega$ ensures a existence and uniqueness for the solution to \eqref{eqn.par.BVP}. Moreover, one can show that $u(x,t) = \Prob^{x} \left( \tau_{\Omega} > t \right)$, where $\tau_{\Omega}$ is the exit time for the process having  $\mathcal{P}$ as generator. The case of sub-Laplacian operators on homogeneous Carnot groups has been considered in  \cite[Section 5.2.]{CarfagniniGordina2023} and on fractals in \cite{Hambly2011}.

On a general metric measure space $(\Xcal, \mu)$ solvability of the heat equation \eqref{eqn.par.BVP} is a delicate topic, and it depends on the boundary behavior of the set $\Ucal$, which is why the generalized heat content in Definition \eqref{def.gen.heat.cont} is in terms of the exit probability and not in terms of the solution to the classical heat equation. Note that, if the heat equation for $A$ is solvable in $\Ucal$ then the generalized heat content coincide with the classical heat content. 

We consider the generalized heat equation 
\begin{align}\label{eqn.par.BVP-gen}
& \left( \partial_{t} - A_\mathcal{U} \right) u(x,t) =0, &  (t,x)\in (0,\infty) \times \Ucal, \notag
\\
& u(t,x) =0, & (t,x)\in (0,\infty) \times   \Ucal^c,
\\
& u(0,x)=1, & x\in \Ucal. \notag
\end{align}
In our situation, the technical results, including the   measurability of the heat kernel, are obtained in \cite{GrigoryanHuHu2021b}.

\begin{theorem}\label{thm.heat.content}
Let  $(\Xcal, \mu)$ be a metric measure space,    $\left( \mathcal{E}, \mathcal{D}_{\mathcal{E}}\right)$ be an ultracontracive Dirichlet form, and $P_{t}$ be the associated semigroup satisfying the assumptions stated at the beginning of this section. Let  $\Ucal \subset \Xcal$ be an open connected set in $\Xcal$ such that $0<\mu (\Ucal) < \infty$, and $A_{\Ucal}$ be the generator of $\left( \mathcal{E}, \mathcal{D}_{\mathcal{E}} (\Ucal)\right)$.  Then
\begin{align*}
\lim_{t\rightarrow \infty} e^{\lambda_{1}^{\Ucal} t} Q_{\Ucal} (t) = \sum_{n=1}^{M_{1}} c_{n}^{2}, 
\end{align*}
where $c_{n} := \int_{\Ucal} \phi_{n} (x)  d\mu (x)$, and $M_{1}$ denotes the multiplicity of $\lambda_{1}^{\Ucal}$.
\end{theorem}

\begin{proof}
By  \eqref{eqn.heat.kernel.expansion} we have that 
\begin{align*}
&Q_{\Ucal} (t) = \int_{\Ucal}  \sum_{n=1}^{\infty} e^{-\lambda_{n} t} c_{n} \phi_{n} (x)d\mu (x),
\end{align*}
Note that the series $\sum_{n=1}^{\infty} e^{-\lambda_{n} t} c_{n} \phi_{n} (x)$ converges uniformly on $[\varepsilon, \infty) \times \Ucal$ for any $\varepsilon>0$ because of \cite[Theorem 2.1.4]{DaviesBook1989} (see Proposition \ref{prop.eigenfunctions.expansion}). 

Thus, 
\begin{align*}
& Q_{\Ucal} (t) = \int_{\Ucal}  \sum_{n=1}^{\infty} e^{-\lambda_{n} t} c_{n} \phi_{n} (x) d\mu(x)
\\
& = \sum_{n=1}^{\infty} e^{-\lambda_{n} t} c_{n}  \int_{\Ucal}  \phi_{n} (x)d\mu(x) = \sum_{n=1}^{\infty} e^{-\lambda_{n} t} c_{n}^{2},
\end{align*}
for any $t>0$. If $M_{1}$ denotes the multiplicity of $\lambda_{1}^{\Ucal}$ then we can write 
\begin{align*}
& Q_{\Ucal} (t)  = e^{-\lambda_{1}^{\Ucal} t} \sum_{n=1}^{M_{1}} c_{n}^{2} + \sum_{n=M_{1} +1}^{\infty} c_{n}^{2} e^{-\lambda_{n} t},
\end{align*}
and the result follows by letting $t$ go to infinity since $\lambda_{1}^{\Ucal} < \lambda_{n}$ for any $n > M_{1}$.

Note that, in cases when Nash inequality is satisfied, by   \eqref{eq-Linfty} and  \eqref{e.HKPolyUpperBound}  we moreover have that, for any $x\in \Ucal$ and $t\geqslant\varepsilon$
\begin{align*}
&\vert e^{-\lambda_{n} t} c_{n} \phi_{n} (x) \vert  \leqslant e^{-\lambda_{n} t} c_{n}  \Vert \phi_{n} \Vert_{L^{\infty} (\Ucal, \mu )} \leqslant  \mu(\Ucal)    e^{-\lambda_{n} t} \Vert \phi_{n} \Vert_{L^{\infty} (\Ucal, \mu )}^{2}
\\
& \leqslant  \mu(\Ucal) ^{2}  e^{-\lambda_{n} t} C( \lambda_{n},  \Ucal)^{2} =  \mu(\Ucal) ^{2}  e^{-\lambda_{n} t}  \left(  \inf_{t>0} t^{-\frac{\nu}{2}} e^{\lambda_{n} t} \right)^{2}
\\
&\leqslant d( \nu, \Ucal)   e^{-\lambda_{n} t}   \lambda_{n}^{2\nu} \leqslant  d( \nu, \Ucal)   e^{-\lambda_{n} \varepsilon}   \lambda_{n}^{2\nu}, 
\end{align*} 
for some constant $d( \nu, \Ucal)$ not depending on $n$.  Thus $\sum_{n=1}^{\infty} e^{-\lambda_{n} t} c_{n} \phi_{n} (x)$ converges uniformly on the set $[\varepsilon, \infty) \times \Ucal$ for any $\varepsilon>0$ by Weierstra{\ss}' M-test since the series $\sum_{n=1}^{\infty} e^{-\lambda_{n} \varepsilon} \lambda_{n}^{2 \nu}$ is convergent. 
\end{proof}

\appendix \section{Proof of Theorem~\ref{thm-discrete-compact}}\label{s.ProofThm2.1}

\begin{proof}[Proof of Theorem~\ref{thm-discrete-compact}]
The strategy of the proof is to show that $(1) \implies (2) \implies (3) \implies (1)$. 

$(2) \implies (3):$ if $T_{t_{0}}$ is compact, then by the spectral theorem for compact operators there exists a sequence of eigenvalues $\{ \lambda_{n}(t_{0})\}_{n=1}^{\infty} \subset \R$ and corresponding eigenfunctions $\{\varphi_{n, k}^{(t_{0})} \}_{n, k=1}^{\infty}$ such that
\begin{align}
& \sigma (T_{t_{0}} ) \backslash \{ 0 \} = \sigma_{pp} (T_{t_{0}} ) \backslash \{ 0 \} = \{  \lambda_{n} (t_{0})\}_{n=1}^{\infty}, \quad \lim_{n\rightarrow \infty} \lambda_{n} (t_{0})=0,  \notag
\\
&  \ker \left( \lambda_{n} (t_{0}) - T_{t_{0}}  \right)= \overline{\operatorname{Span}  \{ \varphi_{n,k}^{(t_{0})}, k\in \mathbb{N} \} } \text{ for every } n\in \mathbb{N},  \notag
\\
&H=\bigoplus_{n=1}^{\infty} \overline{\operatorname{Span} \{ \varphi_{n,k}^{(t_{0})}, k \in \mathbb{N}  \}}.\label{eqn.orth.bas}
\end{align}
The semigroup  $\{P_{t}\}_{t>0}$ is strongly continuous and hence  by the spectral mapping theorem for semigroups \cite[Theorem 6.3]{ArendtGraboschGreinerMoustakasNagelNeubranderSchlotterbeck1986}
\begin{align}\label{eqn.spec.mapp.thm}
\sigma_{pp}  (  P_{t}  ) \backslash \{0\} = \exp \left( t\,  \sigma_{pp} ( \mathcal{L} ) \right), \text{ for any }  t>0.
\end{align}
Thus $\sigma (T_{t_{0}} ) \backslash \{ 0 \} =\{e^{-\lambda_{n} t_{0}} \}_{n=1}^{\infty}$, where $\{-\lambda_{n} \}_{n=1}^{\infty} = \sigma_{pp}(\mathcal{L})$, and $0\leqslant \lambda_{1}^{\Ucal} \leqslant \lambda_{2} \leqslant \cdots$. Let $-\lambda_{n} \in \sigma_{pp} ( \mathcal{L} )$, and let $\{ \varphi_{n,k} \}_{n,k \in \mathbb{N}}$ be an orthonormal basis of $\ker \left( -\lambda_{n}  - \mathcal{L}\right)$. By \cite[Corollary 6.4]{ArendtGraboschGreinerMoustakasNagelNeubranderSchlotterbeck1986} it follows that
\begin{align*}
& \overline{\operatorname{Span}\{ \varphi_{n,k}^{(t_{0})}, k \in \mathbb{N} \}} = \ker \left( e^{-\lambda_{n} t}-P_{t_{0}} \right)
\\
& = \overline{\operatorname{Span} \left\{ \ker \left( -\lambda_{n}  + \frac{2 \pi j}{t_{0}} i - \mathcal{L} \right), j \in \mathbb{Z}\right\}}
\end{align*}
The point spectrum of  $\mathcal{L}$ is real since $\mathcal{L}$ is self-adjoint, and hence
\[
\ker \left( -\lambda_{n}  + \frac{2 \pi j}{t_{0}} i - \mathcal{L} \right)= \{ 0 \}
\]
for all $j\not= 0$. Thus we have that
\begin{align}
& \overline{ \operatorname{Span}\{ \varphi_{n,k}^{(t_{0})}, k\in\mathbb{N} \}} = \ker \left( e^{-\lambda_{n} t_{0}}-T_{t_{0}} \right)   \notag
\\
& = \ker \left( -\lambda_{n}   - \mathcal{L} \right)= \overline{ \operatorname{Span} \{ \varphi_{n,k},  k\in\mathbb{N}\}}, \text{ for all }  n\in \mathbb{N}. \notag
\end{align}
The operator $T_{t_{0}}$ is compact, and thus for every $n\in\mathbb{N}$ the eigenspace $\ker \left( e^{-\lambda_{n} t_{0}}  - T_{t_{0}} \right)$ is finite-dimensional. Therefore for every $n\in \mathbb{N}$ there exists an $M_{n}$ such that for every $t>0$
\begin{equation}\label{eqn.orth.bas2}
\ker \left( e^{-\lambda_{n} t_{0}}  - T_{t_{0}} \right)  = \operatorname{Span} \{ \varphi_{n,k}, k=1,\ldots, M_{n} \}.
\end{equation}
We proved that for every $n\in \mathbb{N}$ there exists an orthonormal basis $\{ \varphi_{n,k} \}_{k=1}^{M_{n}}$ of $\ker \left( e^{-\lambda_{n} t_{0}}  - T_{t_{0}} \right)$ such that for any $k=1,\ldots,  M_{n}$
\begin{align}
& T_{t_{0}} \varphi_{n,k} = e^{-\lambda_{n} t_{0}} \varphi_{n,k},   \notag
\\
& \mathcal{L} \varphi_{n,k} = -\lambda_{n} \varphi_{n,k} . \notag
\end{align}
By \eqref{eqn.orth.bas} and \eqref{eqn.orth.bas2} it follows that
\begin{align}\label{eqn.decomposition}
&H=\bigoplus_{n=1}^{\infty} \operatorname{Span} \{ \varphi_{n,k}, k=1,\ldots,M_{n} \},
\end{align}
that is,  $\{\varphi_{n,k} \}_{k=1, n=1}^{M_{n}, \infty}$ is an orthonormal basis for $H$ made of eigenfunctions of $\mathcal{L}$.
Thus by \cite[Theorem XIII.64 p.245]{ReedSimonIV} it follows that $(-\mu -\mathcal{L} )^{-1}$ is a compact operator for every $\mu$ in the resolvent set of $\mathcal{L}$, proving that $\mathcal{L}$ has a pure point spectrum.

$(3) \implies (1)$: if $-\mathcal{L}$ is a positive self-adjoint operator, then by \cite[Theorem XIII.64 p.~245]{ReedSimonIV}  there exists a complete orthonormal basis $ \{ \varphi_{n} \}_{n=1}^{\infty}$ of $H$ such that
\[
- \mathcal{L} \varphi_{n} = \lambda_{n} \varphi_{n}, \quad \varphi_{n} \in D_{\mathcal{L}},
\]
where $0\leqslant \lambda_{1}^{\Ucal} \leqslant \lambda_{2}^{\Ucal} \leqslant \cdots$, and $\lambda_{n}^{\Ucal} \nearrow \infty$. Then by the spectral mapping theorem \eqref{eqn.spec.mapp.thm}
\[
\sigma_{pp} (T_{t_{0}} ) \backslash \{ 0 \} = \{  e^{-\lambda_{n} t }\}_{n=1}^{\infty},  \text{ for any } t>0.
 \]
Proceeding as in the previous step we have that
\begin{align*}
 \ker \left( e^{-\lambda_{n} t}-P_{t} \right)   = \ker \left( -\lambda_{n}   - \mathcal{L} \right) \text{ for any } t>0,
\end{align*}
and hence
\begin{align*}
&H=\bigoplus_{n=1}^{\infty}\ker \left( e^{-\lambda_{n} t}-P_{t} \right), \text{ for any }  t>0,
\end{align*}
that is, for any $u \in H$
\[
P_{t} u = \sum_{n=1}^{\infty} e^{-\lambda_{n} t} \langle u, \varphi_{n} \rangle \varphi_{n}.
\]
Thus $P_{t}$ is compact since it is a limit of finite rank operators.
\end{proof}

\section{Markovian semigroups, restricted Dirichlet forms, heat kernels}\label{s.RestrictedDirichletForm}

\subsection{Generalities}

We recall some general definitions in \cite{FukushimaOshimaTakedaBook2011, ChenFukushimaBook2012, Sturm1996a}, mostly following \cite[Section 1.4, Section 4.3]{FukushimaOshimaTakedaBook2011}. We consider $H= L^{2}(\Xcal, \mu )$ to be the Hilbert space of square integrable  $\mu$-measurable real-valued  functions  on a metric space  $( \Xcal, d)$, and we assume that Assumption \ref{assumption.general} and Assumption \ref{assumption.heat.kernel} are satisfied. 

We consider a \emph{Dirichlet form} $\left( \mathcal{E}, \mathcal{D}_{\mathcal{E}}\right)$, that is, $\mathcal{E}$ is a non-negative closed quadratic symmetric form defined on a dense subset $\mathcal{D}_{\mathcal{E}}$  of $L^{2}(\Xcal, \mu)$ such that  for every $u \in \mathcal{D}_{\mathcal{E}}$
\[
u_{\ast} := \min ( \max (u, 0), 1) \in \mathcal{D}_{\mathcal{E}}
\]
and $\mathcal{E}(u_{\ast}, u_{\ast} )  \leqslant \mathcal{E}(u, u)$. The set $\mathcal{D}_{\mathcal{E}}$ is then a Hilbert space with respect to the inner product
\begin{align*}
\mathcal{E}_{1} (u,v) := \mathcal{E} (u,v) + \langle u, v \rangle_{L^{2}(\Xcal, m)}.
\end{align*}
A \emph{core of a Dirichlet form} $\mathcal{E}$ is by definition a subset $\mathcal{C}$ of $\mathcal{D}_{\mathcal{E}}\cap C_{c} (\Xcal)$ such that $\mathcal{C}$ is dense in  $\mathcal{D}_{\mathcal{E}}$ with respect to the $\mathcal{E}_{1}$-norm and dense in  $C_{c}(\Xcal)$ with respect to the uniform norm. Here $C_{c}(\Xcal)$ denotes the space of continuous compactly supported functions on $\Xcal$. A  Dirichlet form $\mathcal{E}$  is called \emph{regular} if it possesses a core. We denote by $A$ the non-negative self-adjoint operator corresponding to $\mathcal{E}$.

The heat kernel  can be viewed as the transition density of an $\Xcal$-valued Hunt process $\{ X_{t} \}_{t>0}$, that is,
\begin{align*}
\Prob^{x} \left( X_{t} \in E \right) = \int_{E} p_{t} (x,y) d\mu(y).
\end{align*}

We repeat below the argument found in \cite[Proposition 2.1]{Grigoryan1994a} for Riemannian manifolds. 
Recall that by the semigroup property (Chapman-Kolmogorov equations) we have that for any $x, y \in \Xcal$, $t>0$
\begin{align*}
& p_{t}\left( x, y \right)=\int_{\Xcal} p_{t/2}\left( x, z \right) p_{t/2}\left( z, y \right) d\mu \left( z \right),
\end{align*}
therefore if the heat kernel is symmetric, then
\[
p_{t}\left( x, x \right)=\int_{\Xcal} p_{t/2}^{2}\left( x, y \right)  d\mu \left( y \right)
\]
and
\begin{align*}
& p_{t}\left( x, y \right)=\int_{\Xcal} p_{t/2}\left( x, z \right) p_{t/2}\left( z, y \right) d\mu \left( z \right)
\\
& \leqslant \left( \int_{\Xcal} p_{t/2}^{2}\left( x, z \right) d\mu \left( z \right)\right)^{1/2}\left( \int_{\Xcal} p_{t/2}^{2}\left( z, y \right) d\mu \left( z \right)\right)^{1/2}=\sqrt{p_{t}\left( x, x \right)p_{t}\left( y, y \right)}.
\end{align*}
Thus on-diagonal heat kernel estimates imply the off-diagonal heat kernel estimates.

\subsection{Strongly regular Dirichlet forms}

A Dirichlet form $\left( \mathcal{E}, \mathcal{D}_{\mathcal{E}}\right)$ is said to be \emph{strongly local} if $\mathbb{E} (u,v)=0$ whenever $u\in  \mathcal{D}_{\mathcal{E}}$ is constant in a neighborhood of the support of $v \in \mathcal{D}_{\mathcal{E}}$. Any regular, strongly local Dirichlet form on $L^{2}(\Xcal, \mu)$ can be written as
\begin{align*}
\mathcal{E} (u,v) = \int_{\Xcal}  d\Gamma (u,v),
\end{align*}
where $\Gamma$ is a positive semi-definite, symmetric, bilinear form on $\mathcal{D}_{\mathcal{E}}$ with values in the space of signed Radon measures on $\Xcal$. The \emph{energy measure} (carr\'{e} du champ operator) $\Gamma$ can be used to define the intrinsic metric on $\Xcal$ as follows. 
\begin{equation}\label{eqn.intrinsic.metric}
\rho (x, y):= \sup\left\{ u(x) - u(y):  u \in \mathcal{F}_{loc}  \cap C(\Xcal),  \Gamma (u, u) \leqslant \mu  \text{ on } \Xcal \right\},
\end{equation}
where $C(\Xcal)$ is the space of continuous functions on $\Xcal$ and
\[
\mathcal{F}_{loc}:= \left\{  u\in L^{2}_{loc} (\Xcal, \mu):  \Gamma (u,u) \text{ is a Radon measure } \right\}.
\]
By $\Gamma (u, u) \leqslant \mu$ in \eqref{eqn.intrinsic.metric} we mean that the energy measure $\Gamma (u, u)$ is absolutely continuous with respect to $\mu$ with the Radon-Nikodym derivative bounded by $1$ on $\Xcal$.  The metric $\rho$ is also known as a Carath\'eodory metric \cite[Equation (1.1)]{Sturm1996a}. 

\begin{definition}\label{def.strongly.regular}
A strongly local Dirichlet form $\left( \mathcal{E}, \mathcal{D}_{\mathcal{E}}\right)$  is called \emph{strongly regular} if it is regular and if $\rho$ is a metric on $\Xcal$ whose topology coincides with the original one. 
\end{definition}
Following \cite{Sturm1996a}, if $\left( \mathcal{E}, \mathcal{D}_{\mathcal{E}}\right)$   is strongly regular then the underlined space $\Xcal$ is connected and $\rho$ is a non-degenerate metric inducing the original topology on $\Xcal$.  If $\Xcal$ is complete, then closed balls  $\overline{B}_{r}(x)$ are complete with respect to the metric $\rho$.

\subsection{Restricted Dirichlet forms}
We now describe a restriction $A_{\Ucal}$ of $A$ to a set $\Ucal$ in $\Xcal$ conditioned on being zero on the complement of $\Ucal$. Our goal is to describe spectral properties of $A_{\Ucal}$ with minimal assumptions on $\Ucal$ and no assumptions on its boundary $\partial \Ucal$. We first describe the restriction of the Dirichlet form. Let $\Ucal$ be a Borel set in $\Xcal$ and define
\[
\mathcal{D}_{\mathcal{E}} (\Ucal) :=  \overline{\left\{ f\in \mathcal{D}_{\mathcal{E}}:  \operatorname{supp} f \subset \Ucal \right\}}^{\mathcal{E}_{1}},
\]
where $\mathcal{E}_{1} (f) = \mathcal{E}(f) + \Vert f \Vert^{2}_{L^{2} (\Xcal, \mu)}$.  We refer to $\left( \mathcal{E}, \mathcal{D}_{\mathcal{E}} (\Ucal) \right)$ as the \emph{restricted Dirichlet form} \cite[Theorem 4.3.1]{FukushimaOshimaTakedaBook2011}. If $\Ucal$ is open and $\left( \mathcal{E}, \mathcal{D}_{\mathcal{E}}  \right)$ is regular, then $\left( \mathcal{E}, \mathcal{D}_{\mathcal{E}} (\Ucal) \right)$ is a regular Dirichlet form on $L^2 (\Ucal, \mu )$ by \cite[(2.3.6) and  Corollary 2.3.1]{FukushimaOshimaTakedaBook2011} and \cite[Section 3.2 Theorem 3.3]{GrigoryanHuLau2014}. We denote by $A_{\Ucal}$ the negative self-adjoint operator on $L^2\left( \Ucal, \mu \right)$ corresponding to the Dirichlet form $\left( \mathcal{E}, \mathcal{D}_{\mathcal{E}} (\Ucal) \right)$.

For an open set $\Ucal$ we denote by
\begin{align*}
\tau_{\Ucal}:= \inf \left\{ t>0: \, X_{t} \notin \Ucal \right\}
\end{align*}
the exit time from $\Ucal$. Then we can use Dynkin-Hunt's formula, see \cite[Equation (3.15)]{GetoorSharpe1982}, \cite[Equation (4.1.6)]{FukushimaOshimaTakedaBook2011}, \cite[Theorem 2.4]{ChungZhaoBook1995}, \cite[Equation (1.1)]{BarlowBassBurdzy1997}, and Hunt's original paper \cite[Equation (10), p.~305]{Hunt1956b}
\begin{equation}\label{eqn.Dirichlet.HK}
p^{\Ucal}_{t}(x,y):= p_{t}(x,y) - \E^x \left[ \mathbbm{1}_{\{ \tau_{\Ucal} < t
\} } \,  p_{ t- \tau_{\Ucal}} \left( X_{\tau_{\Ucal}}, y\right) \right]
\end{equation}
for the transition density $p^{\Ucal}_{t} (x,y)$ of the \emph{killed Markov process} $\{X_{t}^{\Ucal}\}_{t>0}$ given by
\begin{equation*}
X_{t}^{\Ucal}:=
\left\{ \begin{array}{cc}
X_{t} & t < \tau_{\Ucal},
\\
\partial & t \geqslant \tau_{\Ucal},
\end{array} \right.
\end{equation*}
where $\partial$ is the \emph{cemetery} point. More precisely, we have that
\begin{equation}\label{eqn.density}
\Prob^{x} \left( X^{\Ucal}_{t} \in E \right) = \Prob^{x} \left( X_{t} \in E,  \, t<\tau_{\Ucal} \right) = \int_{E} p^{\Ucal}_{t}(x,y) d\mu(y),
\end{equation}
for any $x\in \Ucal$. In particular,
\begin{equation}\label{eqn.exit.time}
\Prob^{x} \left( \tau_{\Ucal} > t \right) = \int_{\Ucal} p^{\Ucal}_{t}(x,y) d\mu (y).
\end{equation}
The function $p_{t}^{\Ucal}$ is called the \emph{Dirichlet heat kernel}, and it is the kernel of the semigroup $P^{\Ucal}_{t}$.

 \end{document}